\documentclass[10pt]{amsart}

\usepackage{tikz}

\usepackage{mathrsfs}
\usepackage[colorlinks=true,urlcolor=blue, citecolor=red,linkcolor=blue,linktocpage,pdfpagelabels, bookmarksnumbered,bookmarksopen]{hyperref}
\usepackage[hyperpageref]{backref}
\usepackage{cleveref}
\usepackage{xcolor}
\usepackage{amsthm} 
\usepackage{latexsym,amsmath,amssymb}
\usepackage{accents}
\usepackage[colorinlistoftodos,prependcaption,textsize=tiny]{todonotes}
\usepackage[a4paper, total={6.5in, 9in}]{geometry}
\usepackage{soul}
\usepackage{mathtools} 
\usepackage{xparse} 
\usepackage{enumitem}

\definecolor{indigo}{rgb}{0.29, 0.0, 0.51}
\definecolor{p1}{gray}{0.4}
\definecolor{p2}{gray}{0.6}
\definecolor{p3}{gray}{0.98}
\definecolor{p4}{gray}{0.8}
\definecolor{p5}{gray}{0.9}

plus 6pt minus 12pt
\def\eps{\varepsilon}

\def\vp{\varphi}

\def \r {\rho}

\def\A{{\mathcal A}}
\def\B{{\mathcal B}}
\def\C{{\mathcal C}}

\def\N{{\mathbb N}}

\def\U{{\mathcal U}}

\def\I{{\mathcal I}}

\def\p{{\partial}}

\def\e{\varepsilon}

\def\pp{{\mathrm p}}
\def\qq{{\mathrm q}}
\def\rr{{\mathrm r}}
\def\ss{{\mathrm s}}

\newtheorem{theorem}{Theorem}
\newtheorem{claim}{Claim}
\newtheorem{lemma}[theorem]{Lemma}

\newtheorem{proposition}[theorem]{Proposition}

\newtheorem{remark}[theorem]{Remark}
\newtheorem{definition}[theorem]{Definition}

\newcommand{\loc}{\mathrm{loc}}

\def\dist{{\rm dist\,}}

\def\supp{{\rm supp\,}}
\def\loc{{\rm loc}}
\def\tr{{\rm tr}}

\newcommand{\Int}{{\rm Int}}
\newcommand{\vol}{{\rm vol}}
\newcommand{\geucl}{ g_{{\rm eucl}}}
\newcommand{\scal}[2]{\left\langle #1 , #2 \right\rangle}
\newcommand{\BMO}{\mathrm{BMO}}
\newcommand{\Hr}{\mathcal{H}}
\newcommand{\hol}{\mathrm{hol}}

\newcommand{\dif}{\,\mathrm{d}}
\newcommand{\dx}{\dif x}

\newcommand{\R}{\mathbb{R}}

\newcommand{\s}{\mathbb{S}}

\newcommand{\pl}{\partial}

\newcommand{\abs}[1]{\left |#1 \right |}

\newcommand{\norm}[1]{\left\|{#1}\right\|}

\newcommand{\barint}{
\rule[.036in]{.12in}{.009in}\kern-.16in \displaystyle\int }

\newcommand{\barcal}{\mbox{$ \rule[.036in]{.11in}{.007in}\kern-.128in\int $}}

\def\mvint_#1{\mathchoice
          {\mathop{\vrule width 6pt height 3 pt depth -2.5pt
                  \kern -8pt \intop}\nolimits_{\kern -3pt #1}}%
          {\mathop{\vrule width 5pt height 3 pt depth -2.6pt
                  \kern -6pt \intop}\nolimits_{#1}}%
          {\mathop{\vrule width 5pt height 3 pt depth -2.6pt
                  \kern -6pt \intop}\nolimits_{#1}}%
          {\mathop{\vrule width 5pt height 3 pt depth -2.6pt
                  \kern -6pt \intop}\nolimits_{#1}}}

\numberwithin{theorem}{section} \numberwithin{equation}{section}

\newcommand{\lap}{\Delta }
\newcommand{\g}{\nabla }

\newcommand*\oline[1]{%
  \kern0.1em            
  \vbox{%
    \hrule height 0.5pt 
    \kern0.4ex          
    \hbox{%
      \kern-0.1em       
      $#1$
      \kern-0.1em       
    }
  }
  \kern0.1em            
}

\def\XXint#1#2#3{{\setbox0=\hbox{$#1{#2#3}{\int}$}
     \vcenter{\hbox{$#2#3$}}\kern-.5\wd0}}

\let\latexchi\chi
\makeatletter
\renewcommand\chi{\@ifnextchar_\sub@chi\latexchi}
\newcommand{\sub@chi}[2]{
  \@ifnextchar^{\subsup@chi{#2}}{\latexchi^{}_{#2}}%
}
\newcommand{\subsup@chi}[3]{
  \latexchi_{#1}^{#3}%
}
\makeatother

\title[]{A weak energy identity for \((n+\alpha)\)-harmonic maps with a free boundary in a sphere}

\author{Dorian Martino}
\address[Dorian Martino]{
	Department of Mathematics,
	ETH Zürich, 
	101 Rämistrasse, 
	8092 Zürich,
	Switzerland}
\email{dorian.martino@math.ethz.ch}

\date{\today}

 \author{Katarzyna Mazowiecka}
 \address[Katarzyna Mazowiecka]{
 Institute of Mathematics,
 University of Warsaw,
 Banacha 2,
 02-097 Warszawa, Poland}
 \email{k.mazowiecka@mimuw.edu.pl}

 \author{R\'emy Rodiac}
 \address[R\'emy Rodiac]{
 Institute of Mathematics,
 University of Warsaw,
 Banacha 2,
 02-097 Warszawa, Poland}
 \email{rrodiac@mimuw.edu.pl}

\begin{document}

\begin{abstract}
In this article, we show that sequences of \((n+\alpha)\)-harmonic maps with a free boundary in $\s^{d-1}$, where \(\alpha\) is a parameter tending to zero, converge to a bubble tree. For such sequences, we prove in detail that the limiting energy is equal to the energy of the macroscopic limit plus the sum of the energies of certain ``bubbles'', each multiplied by a corresponding coefficient.

\end{abstract}
\sloppy

\maketitle

\sloppy

\section{Introduction}

Conformally invariant energies are of special interest in geometric analysis. Due to their invariance under an infinite-dimensional group, solutions to the corresponding variational problems lack compactness, placing us in a limiting case where the Palais--Smale condition fails.  A famous example of a conformally invariant energy is the Dirichlet energy in dimension 2, whose natural generalization in dimension \(n\) is provided by the \(n\)-energy.

Let \(n,d \geq 2\) be two integers and let \((\Sigma^n,g)\) be a smooth compact manifold with non-empty boundary \(\p \Sigma^n\neq \emptyset\). Given $p\geq 1$, we consider the \(p\)-energy, defined for \( u\in W^{1,p}\left(\Sigma^n;\R^d \right)\) by
\begin{equation}\label{def:energy}
E_p(u)\coloneqq\int_{\Sigma^n}|du|_g^p\, d \vol_g.
\end{equation}
This work is motivated by the study of $E_n$ restricted to the space 
\begin{equation}\label{def:fct_space_I}
\I\coloneqq \left\{v\in W^{1,n}\left( \Sigma^n; \R^d \right)\colon \abs{\tr\big\rvert_{\p \Sigma^n }u}=1 \text{ on } \p \Sigma^n \right\}.
\end{equation}
Critical points of \(E_n\) in \(\I\) are called \textit{\(n\)-harmonic maps with a free boundary in the sphere \(\mathbb{S}^{d-1}\)}. They are weak solutions to the following system
\begin{equation}\label{eq:system_formal}
\left\{
\begin{array}{rcll}
d^{*_g}(|du|_g^{n-2}du)&=&0 & \text{ in } \Int(\Sigma^n), \\
|u|&=&1 & \text{ on } \p \Sigma^n,\\
|du|_g^{n-2}\p_{\nu_g} u &\perp & T_{u} \mathbb{S}^{d-1} & \text{ on } \p \Sigma^n.
\end{array}
\right.
\end{equation}
Here we have denoted $d^{*_g}$ as the codifferential of the metric $g$ and \(\nu_g\) as the unit normal vector to the \(\partial \Sigma^n\). More precisely, they satisfy the following definition.
\begin{definition}\label{def:solution_m}
Let \(u \in \I\), where \(\I\) is the space defined in \eqref{def:fct_space_I}. We say that \(u\) is a weak solution to \eqref{eq:system_formal} if
\begin{equation}\label{eq:main_equation}
\int_{\Sigma^n} \langle |du|_g^{n-2}du,d\Phi \rangle_g \ d \vol_g =0 \quad \text{ for all } \Phi \in W^{1,n}(\Sigma^n;\R^d) \text{ with } \Phi \cdot u =0 \text{ a.e.\ on } \p \Sigma^n.
\end{equation}
\end{definition}

It was proved in \cite{MazowieckaRodiacSchikorra} that any weak solution to \eqref{eq:system_formal} is actually $\C^{1,\beta}$ for some $\beta$ depending only on $n$ and $d$. In order to understand better the solutions to \eqref{eq:system_formal}, the next step is the study of their compactness properties. Namely, we consider a sequence $(u_k)_{k\in\N}\subset \I$ of solutions to \eqref{eq:system_formal} with bounded $n$-energy and analyze its asymptotic behavior.

There are at least two obstructions that prevent strong convergence in the $W^{1,n}$ topology in full generality. First, the energy \(E_n\) is conformally invariant. As a result, one can construct sequences of solutions to \eqref{eq:system_formal} that are bounded in \(W^{1,n}\) but do not converge strongly in \(W^{1,n}\). For example, if \(\Sigma^n=\overline{\mathbb{B}^n}\) is the closed unit ball of $\R^n$, we can take a sequence of M\"obius maps \(M_{a_k}(x)=\psi_{a_k}(x)/|\psi_{a_k}(x)|^2\) where
\[\psi_{a_k}(x)=a_k+(1-|a_k|^2)\frac{a_k-x}{|a_k-x|^2}, \quad x\in \overline{\mathbb{B}^n},
\]
and \((a_k)_k\) is a sequence of points in \(\mathbb{B}^n\) converging to \((1,0,\dots,0)\). It is known that \(M_{a_k}\) is conformal and maps \(\p \mathbb{B}^n\) into \(\p \mathbb{B}^n\) with degree 1. Therefore each \(M_{a_k}\) minimizes the \(n\)-energy with prescribed degree 1 on \(\p \mathbb{B}^n\) and hence satisfies \eqref{eq:system_formal}. Furthermore each \(M_{a_k}\) has the same \(n\)-energy, given by a constant times the volume of \(\mathbb{B}^n\). However, it can be shown that \(M_{a_k}\) converges weakly in \(W^{1,n}\) to a constant function and this convergence is not strong. For a detailed explanation in the case $n=2$ we refer to \cite[Section 4.1]{Berlyand_Mironescu_2008} or \cite[Section 3]{Berlyand_Mironescu_Rybalko_Sandier_2014}, for the general case, see \cite[Chapter 5]{RodiacPhD}.
The second obstruction arises from the fact that we are dealing with the critical case of Sobolev embeddings: while \(W^{1,n}(\Sigma^n;\R^d)\) embeds into \(L^p(\Sigma^n;\R^d)\) for every \(1\leq p<+\infty\), it fails to embed into \( L^{\infty}(\Sigma^n;\R^d)\).
For these reasons the existence of solutions to \eqref{eq:system_formal} and their regularity up to the boundary is a non-trivial problem.

Such free-boundary $n$-harmonic maps appear in the study of Steklov eigenvalues. For $n=2$, Fraser--Schoen proved \cite{fraser2013} that one can associate a free-boundary harmonic map in $\s^N$, for some \(N\in \mathbb{N}\), to each critical point of the Steklov eigenvalues in a given conformal class. This point of view was used to study extremal problems for Steklov eigenvalues, see for instance \cite{petrides2014,petrides2019}. In the case $n\geq 3$, an analogous property was proved by Karpukhin--Metras \cite{karpukhin2022}. We believe that the present work could be of help for further developments of these optimization problems, see for instance \cite[Conjecture 1]{karpukhin2022}. When \(n=2\) and the target manifold on the boundary is \(\mathbb{S}^1\), the existence of free-boundary harmonic maps has been studied in various settings: for \(\Sigma^2\) as a disk in \(\R^2\) \cite{Berlyand_Mironescu_Rybalko_Sandier_2014,Millot_Sire_2015}, for \(\Sigma^2\) as a polygon \cite{Smyrnelis_2015}, and for \(\Sigma^2\) as an annulus \cite{Hauswirth_Rodiac_2016}.

The general strategy to study compactness properties of solutions to equations like \eqref{eq:system_formal} is by now well understood, see for instance \cite{sacks1981,Struwe_1984,Brezis_Coron_1985,parker1996,Struwe_2008} and in particular \cite{laurain2014} for a systematic approach in dimension 2. Thanks to the $\eps$-regularity proved in \cite{MazowieckaRodiacSchikorra,martino2024}, one obtains that a sequence $(u_k)_{k\in\N}\subset \I$ of $W^{1,n}$-bounded solutions to \eqref{eq:system_formal} converges strongly on $\Sigma^n$, away from a finite number of points, to some limit $u\in\I$ which is also a solution to \eqref{eq:system_formal}. At each of these points, one can blow-up a finite number of bubbles $(\omega_i)_{1\leq i\leq \ell}$, defined as follows:
\begin{definition}\label{def:bubble}
	We say that a map \(\omega \colon\R^n_+\rightarrow \R^d\) is a bubble provided that \(\int_{\R^n_+} |d\omega|^n dx <+\infty\) and
	\begin{equation}
		\left\{
		\begin{array}{rcll}
			d^{*_{g_{\text{eucl}}}} (|d\omega|^{n-2}d\omega)&=&0 & \text{ in } \R^n_+, \\
			|\omega|&=&1 & \text{ on } \p \R^n_+, \\
			|d\omega|^{n-2}\p_{\nu} \omega &\perp & T_{\omega} \mathbb{S}^{d-1} & \text{ on } \p \R^n_+ ,
		\end{array}
		\right.
	\end{equation}
	where these equations are understood as in \eqref{eq:main_equation} and where \(d^{*_{g_{\text{eucl}}}}\) is the codifferential associated to the Euclidean metric.
\end{definition}
This procedure of extraction of bubbles can be understood as a type of weak limit behavior, called weak bubbling convergence. In particular, one can show the following inequality, up to a subsequence,
\begin{align*}
	\lim_{k\to +\infty} E_n(u_k) \geq E_n(u) + \sum_{i=1}^{\ell} E_n(\omega_i).
\end{align*}
The question is whether equality holds in this limit. If so, we say that the strong bubbling convergence holds.
This problem was first studied in the case of harmonic maps on manifolds without boundary for \(n=2\), see, for example,\cite{ding1995,qing1997,Chen_Tian_1999,Jost_1991,parker1996,Duzaar_Kuwert_1998}. More recently, still for $n=2$, the case of harmonic maps with free boundary has been investigated in various settings by different authors (see, for instance,\cite{laurain2017,laurain2019,Jost_Liu_Zhu_2019_b,liu2022}). In this case, the question can be stated in terms of $\frac{1}{2}$-harmonic maps and we refer to \cite{dalio2015,dalio2020}.

For $n\geq 3$, Mou and Wang \cite{Mou_Wang_1996} proved an energy identity for some Palais--Smale sequences associated with the \(n\)-harmonic energy, in the case of maps from a manifold without boundary into a Riemannian homogeneous space, see also \cite{Wang_Wei_2002}.
If the target manifold is not homogeneous, the following weaker result is known: a sequence of $W^{1,n}$-bounded $n$-harmonic maps must converge strongly to a limiting $n$-harmonic map, except at a finite number of points, see e.g.\ \cite{wang2005,miskiewicz2016}. Although regularity properties of  $n$-harmonic maps with values into homogeneous manifolds are now well studied, see e.g.\ \cite{Fuchs_1993,Strzelecki,Toro_Wang_1995,Mou_Yang_1996}, we point out that the case of general manifolds is not yet fully understood.  For recent progress on this topic, we refer to  \cite{schikorra2017,Miskiewicz_Petraszczuk_Strzeleci_2023,MS2,MS1}. The energy identity for minimizing free-boundary $n$-harmonic maps was proved in \cite{muller2000}.

One of the motivation of the present work is to prove the strong bubbling convergence for sequences of free-boundary $n$-harmonic maps into a round sphere. However, we will consider the more general setting of sequences of free-boundary \((n+\alpha_k)\)-harmonic maps, where \(\alpha_k\geq 0\) and $\alpha_k\to 0$ as $k\to +\infty$. Studying a sequence of solutions to sub-critical variational problems to understand a critical one is reminiscent of the Sacks--Uhlenbeck approach \cite{sacks1981}, where the authors examined the existence of harmonic maps from a closed, compact surface. When $\alpha_k>0$, it can be shown that the  \((n+\alpha_k\))-energy satisfies the Palais--Smale condition. Hence, we can hope to find non-trivial free-boundary \((n+\alpha_k)\)-harmonic maps, either through minimization or min-max procedures, within the set of maps with prescribed degrees on \(\p \Sigma^n\), in the case where the target manifold is \(\mathbb{S}^{n-1}\). We then can expect to use an energetic argument and the strong bubbling convergence to obtain the existence of non-trivial solutions to \eqref{eq:main_equation} with prescribed degrees. We mention that a Sacks--Uhlenbeck approach to prove existence of minimizers of fractional energies in homotopy classes has been employed in \cite{Mazowiecka_Schikorra_2023,Mazowiecka_Schikorra_2024}.

A weak energy identity for \(\alpha\)-harmonic maps from a closed surface to a general manifold, where coefficients possibly different from 1 appear in front of the energies of the bubbles,  was derived by Li--Wang \cite{li2010}. They also showed in \cite{li2015} that we cannot expect the genuine energy identity (the one with all the coefficients equal to $1$) to hold when the homotopy classes degenerate. However, in some cases, the genuine energy identity is known to hold, cf.\ \cite{Duzaar_Kuwert_1998,Chen_Tian_1999,Lamm_2010}. This is the case in particular when the target manifold is the sphere \cite{Li_Zhu_2019} or a homogeneous manifold \cite{Bayer_Roberts_2025}. We point out that  \(\alpha\)-harmonic maps are still the object of ongoing research. For example, it is remarkable that \(\alpha\)-harmonic maps do not approximate every harmonic maps, and there is a selection process at work when using the \(\alpha\)-energy (see \ \cite{Lamm_Malchiodi_Micallef_2020a, Lamm_Malchiodi_Micallef_2021}). Moreover, the location of blow-up points can sometimes be described, as shown in \cite{Sharp_2024},  and Morse index estimates can be obtained as explained in \cite{DaLio_Riviere_Schlagenhauf_2025} (see also \cite{DalioGianoccaRiviere,HirschLamm}).
 The case of free-boundary \(\alpha\)-harmonic maps was  considered in \cite{Jost_Liu_Zhu_2022} where a weak energy identity is derived. Again, in the case where the target manifold is a sphere, the genuine energy identity holds, see \cite{Li_Zhu_Zhu_2023}.  We note that in the articles mentioned above, \(\alpha\)-harmonic maps are defined as critical points of \(\int_{\Sigma^2}(1+|du|_g^2)^{1+\alpha}d \vol_g\). This definition allows to avoid the degeneracy character of the \(p\neq 2\)-harmonic maps equations and to obtain elliptic estimates in a simpler way. We are able to deal directly with the degenerate \( (n+\alpha)\)-harmonic maps equation. Thus, to the best of our knowledge, our result seems new even in dimension \(n=2\).

In the setting of $(n+\alpha)$-harmonic maps with values into the sphere, we prove a weak energy identity and, for sequences of \(n\)-harmonic maps, we obtain that the strong bubbling convergence always holds. First, we observe that no degeneracy behavior can occur in the interior of $\Sigma^n$, since the system reduces to $\lap_{n+\alpha_k,g} u_k=0$. The interior regularity for the $p$-Laplacian, as proved by Uhlenbeck \cite{Uhlenbeck_1977}, shows that we always have local strong convergence. Thus, the only possible concentration points occur at the boundary of $\Sigma^n$.
Our main result is the following.

\begin{theorem}\label{th:main_bubbling}
	Let $n,d\geq 2$ and $(\Sigma^n,g)$ be a compact manifold with non-empty boundary. Let \( (p_k)_{k\in\N} \subset [n,n+1] \) be a sequence of numbers such that $p_k\to n$ as $k\to \infty$. Let \(M>0\) and let \( (u_k)_{k\in\N} \subset W^{1,n}(\Sigma^n;\R^d)\) be a sequence of maps that satisfy:
\begin{itemize}
\item[1)] for all \(k\), \(u_k \in W^{1,p_k}(\Sigma^n;\R^d)\) and  \(|u_k|=1\) on \(\p \Sigma^n\);

\item[2)] for all \(\Phi \in W^{1,p_k}(\Sigma^n;\R^d)\) such that \(\Phi \cdot u_k =0\) on \(\p \Sigma^n\), we have \[\int_{\Sigma^n} \langle |du_k|_g^{p_k-2} du_k , d\Phi \rangle_g d \vol_g =0;\]

\item[3)] \( \displaystyle{\sup_{k\in\N}}\ \|du_k\|^{p_k}_{L^{p_k}(\Sigma^n)} \leq M\).
\end{itemize}
Then there exist 
\begin{itemize}[label=$\bullet$]
	\item \(u\in \C^1(\Sigma^n;\R^d)\), with \(|u|=1\) on \(\p \Sigma^n\) satisfying \eqref{eq:system_formal};
	\item a finite number \(\ell\) of bubbles \(\omega_i\in \C^1(\overline{\R^n_+};\R^d)\) for $1\leq i\leq \ell$;
	\item \(\ell\) sequences of points \( (a_k^i)_{k\in\N} \subset \Sigma^n, 1\leq i \leq \ell\), converging to points \( a^1,\ldots,a^{\ell}\in \partial \Sigma^n\);
	\item \( \ell\) sequences of real numbers \( (\lambda_k^i)_{k\in\N}\subset \R^+\), \(1\leq i\leq \ell\);
	\item a subsequence in \(k\), (non-relabelled)
\end{itemize}
 such that
\begin{enumerate}[label=\alph*)]
	
\item\label{item:speeds_thm} For all \(1\leq i \neq j\leq \ell\), it holds
\begin{equation}\label{eq:sep_bubbles_th}
	\max \left\{ \frac{\lambda_k^i}{\lambda_k^j}, \frac{\lambda_k^j}{\lambda_k^i}, \frac{|a_k^i-a_k^j|}{\lambda_k^i+\lambda_k^j} \right\}\xrightarrow[k\to+\infty]{}{ +\infty},
\end{equation}

\item For all $i\in\{1,\ldots,\ell\}$, the sequence $\left( u_k\left(\exp_{a^i}\left(a_k^i + \lambda_k^i\cdot\right) \right)\right)_{k\in\N}$ converges to $\omega_i$ in the strong $\C^1_{\loc}$ topology on $\overline{\R_+^n}$ away from a finite set,

\item \label{item:nEnergy} Let $\lambda^i_* \coloneqq  \displaystyle{\liminf_{k\to +\infty}} \left( \lambda_k^i\right)^{n-p_k}$. Then we have $\lambda_*^i \in [1,+\infty)$ together with the following energy-identity
\begin{align}\label{eq:energy_identity_th}
	\lim_{k\to +\infty} E_{p_k}(u_k) = E_n(u) + \sum_{i=1}^{\ell} \lambda_*^i E_n(\omega_i).
\end{align}
\end{enumerate}
\end{theorem}
In particular when \(p_k=n\) for all \(k\) we obtain the energy identity for a sequence of free-boundary \(n\)-harmonic maps:
$$
\lim_{k \to +\infty} E_n(u_k)=E_n(u)+\sum_{i=1}^\ell E_n(\omega_i).
$$

In the case \(d=n\), we can have a useful information on the degrees of \(u_k\) on the connected components of \(\p \Sigma^n\):

\begin{proposition}\label{prop:degrees}
Assume that \(d=n\), we call \(\Gamma_1,\dots, \Gamma_m\) the connected components of \(\p \Sigma^n\). Under the assumptions of Theorem \ref{th:main_bubbling}, for $k$ large enough we have that for all \(1\leq i\leq m\),
\begin{equation}
	\deg \left(u_k\big\rvert_{\Gamma_i} \right)= \deg \left( u\big\rvert_{\Gamma_i} \right) +\sum_{j\in \text{Conc}(\Gamma_i)} \deg \left( \omega_j\big\rvert_{\p \R^n_+} \right).
\end{equation}
where \(  \text{Conc}(\Gamma_i)=\Big\{ j \in \{1,\dots, \ell\} \colon a^j \in \Gamma_i \Big\}\).
\end{proposition}

We now comment on the method of the proof of Theorem \ref{th:main_bubbling} and the organization of the paper.

The first step is to show that near the boundary, if there is no energy concentration, then the sequence $(u_k)_{k\in\N}$ converges strongly in $\C^1_{\loc}$. This is done in \Cref{sec:Compactness}, where an \(\e\)-regularity result is obtained. Our starting point is the \(\e\)-boundary regularity for free-boundary \(n\)-harmonic maps obtained  in \cite{MazowieckaRodiacSchikorra}, where it is shown that such maps are \(\C^{1,\beta}\left( \overline{\Sigma^n} ; \R^d \right)\) for some \(0<\beta<1\). We need to adapt this result to obtain uniform estimates in \(p\in [n,n+1]\) for \(p\)-harmonic maps.  At the $\C^{0,\beta}$-level this is obtained by an adaptation of the arguments used in \cite{MazowieckaRodiacSchikorra}. However, to obtain uniform $\C^{1,\beta}$ estimates, we cannot apply \cite[Theorem 3.1]{HLp} as in \cite{MazowieckaRodiacSchikorra}, since we require the estimates to be uniform. Instead, we perform the reflection near the boundary, following Scheven \cite{scheven2006}. This leads to solutions of a critical $p$-Laplace system with respect to a metric that is \textit{a priori} merely Hölder-continuous. Such systems have been recently studied by the first author \cite{martino2024}.

The second step is the study of bubbles. In particular, two classical ingredients that we need here are: an energy gap and a singularity removability result. They are proved in \Cref{sec:Bubbles}.

 The key point in establishing an energy identity like \ref{item:nEnergy} is to prove a so-called no-neck energy property: we  need to show that no energy remains in the regions between the bubbles, which are annuli of degenerating conformal type called the neck regions. To do so, we prove some energy comparison lemmas in \Cref{sec:Energy}. More precisely, the crucial ingredient, Lemma \ref{lem:comparison_2}, says that under an assumption of smallness of the energy, the energy of a free-boundary \((n+\alpha)\)-harmonic map in an annulus intersecting the boundary is smaller than the energy of its non-constrained \( (n+\alpha)\)-harmonic extension in this annulus. When the map is constant on the boundary of the annulus, the energy of the \((n+\alpha)\)-harmonic extension is easily estimated. In our case, the maps considered are almost constant on the boundaries and we can then show that the energy vanishes in neck regions. This argument can be understood as an alternative to Pohozaev identity, generally used to prove the no-neck property for the radial part of the gradient, see for instance \cite[Section 3]{laurain2014}.

We prove the strong bubbling convergence in the final \Cref{sec:Bubbling_decompo}. Our strategy follows the one in \cite{Brezis_Coron_1985,Mou_Wang_1996}. We introduce the energy-concentration function to detect bubbles, or more precisely their concentration points and concentration speeds. Using this procedure, near a given concentration point, we extract the bubbles in a precise order: from the most concentrated to the least concentrated. We then study the convergence of the blow-up sequence towards the bubbles. A difference with the case of harmonic maps or $H$-systems in dimension 2 is that after extracting the first bubble, it is not clear that the difference between $u_k$ and the bubble still satisfies a PDE similar to the original sequence. Thus, we cannot use directly the \(\e\)-regularity lemma after having extracted the first bubble and the induction for extracting all bubbles is not as straightforward as in the elliptic case. To prove the no-neck energy property we depart from \cite{Mou_Wang_1996} where only the \(W^{1,n}_{\text{loc}}\)-convergence towards the bubble is used and is obtained from the \(\C^{0,\beta}\)-estimate, cf.\ \cite[Proposition 2.5]{Mou_Wang_1996}. Since we work with sequences of \( (n+\alpha)\)-harmonic maps, we cannot use this proposition and we have to employ the \(\C^{1}_{\text{loc}}\)-convergence that we obtained thanks to the results in \cite{martino2024}. Since multiple bubbles can appear at the same concentration point, we extract them generation by generation, and this iterative process is detailed in this article.

\subsection*{Open problems}

We end this introduction with a few open problems:
\begin{enumerate}

\item We obtained a weak energy identity for sequences of \((n+\alpha)\)-harmonic maps with a free boundary in a sphere \(\mathbb{S}^{d-1}\),  with \(\alpha\) tending to zero. It is very likely that the result remains true if \(\mathbb{S}^{d-1}\) is replaced by a homogeneous Riemannian manifold. In the case \(n=2\), when the target manifold  is a sphere and when \( \Sigma^2\) is a compact manifold without boundary, the genuine energy identity holds for sequences of ``Sacks--Uhlenbeck'' \(\alpha\)-harmonic maps \cite{Li_Zhu_2019}. The same result was recently proved for homogeneous targets \cite{Bayer_Roberts_2025}. The genuine energy identity also holds true when \(n=2\) in the free-boundary case when the target manifold is a sphere \cite{Li_Zhu_Zhu_2023}. Hence, one may ask whether our weak energy identity could be improved into a genuine energy identity, i.e., with all the coefficients \(\lambda_*^i\) appearing in \eqref{eq:energy_identity_th} equal to~1.

\medskip

\item In dimension \(n=2\), it is known that the only possible bubbles are the Blaschke products, i.e., products of M\"obius maps, see for instance \cite[Lemma 3.5]{Berlyand_Mironescu_Rybalko_Sandier_2014} or \cite[Theorem 4.25]{Millot_Sire_2015}. The proof of this fact relies on an argument involving the Hopf differential and can be viewed as an analogue of the fact that harmonic maps from \(\mathbb{S}^2\) to \(\mathbb{S}^2\) are automatically conformal, see \cite{sacks1981}. However in higher dimension, Xu--Yang \cite{Xu_Yang_1992} constructed examples of \(n\)-harmonic maps from \(\mathbb{S}^n\) to \(\mathbb{S}^n\) which are not conformal. We can ask if there exist non-conformal maps satisfying Definition \ref{def:bubble} in any dimension \(n \geq 3\).

\medskip 

\item For harmonic maps from compact surfaces, Li--Wang \cite{li2010} proved that, when bubbling occurs, the images of the ``necks'' converge to geodesics in the target manifold. For harmonic maps with free boundaries, a more complicated phenomenon occurs, see \cite{Jost_Liu_Zhu_2022}, but the images still converge to a curve satisfying an equation close to the geodesic equation. Is it possible to obtain similar results for \(n\)-harmonic maps in the case without boundary and in the case of free-boundary maps?

\end{enumerate}

\subsection*{Notation}

We denote by \(B(a,r)\coloneqq \{x\in \R^n\colon |x-a|<r\}\) the ball in \(\R^n\) with center \(a\in \R^n\) and radius \(r>0\). We denote \(\R^n_+=\{x\in \R^n\colon x_n>0\}\) and \(\R^n_-=\{x\in \R^n\colon x_n<0\}\) with \(x=(x_1,\dots,x_n)\). We will often consider the restriction of a ball to the upper half-space \(B(a,r)^+\coloneqq B(a,r)\cap \R^n_+\). When considering the boundary of the intersection of a ball with the upper-half space we distinguish \(\p B(a,r)\cap \R^n_+\) and
\begin{equation}\label{eq:def_partial_D0}
\p_0B(a,r)^+\coloneqq B(a,r) \cap (\R^{n-1}\times\{0\}).
\end{equation}
We also set \(B(a,r)^{-}\coloneqq B(a,r)\cap \R^n_-\). The balls in \(\Sigma^n\) are denoted by \(B_{(\Sigma^n,g)}(a,r)\coloneqq \{x\in \Sigma^n\colon d_g(x,a)<r\}\) where \(d_g\) is the Riemannian distance.

We also precise that, in a coordinate system and by using Einstein's summation convention, if \(u,\Phi\in W^{1,n}(\Sigma^n;\R^d)\), then
\begin{equation*}
	\langle|du|_g^{n-2} du,d \Phi \rangle_g=  \left(g^{ij} \scal{\p_i u}{\p_j u}\right)^{\frac{n-2}{2}} g^{\lambda \mu} \scal{\p_\lambda u}{\p_\mu \Phi}.
\end{equation*}
For \(p>1\), the \(p\)-Laplace operator with respect to a metric \(h\) is defined in local coordinates for \(f\in W^{1,p}(\Sigma^n)\) by
\begin{equation*}
\Delta_{p,h}f = - d^{*_h}(|df|_h^{p-2} df) = \frac{1}{\sqrt{\det h}} \p_{\lambda}\left( |df|^{p-2}_h h^{\lambda \mu} \sqrt{\det h}\, \p_{\mu} f \right).
\end{equation*}
In the rest of the article \( (\alpha_k)_{k\in\N}\) is a sequence of non-negative numbers such that \(\alpha_k \xrightarrow[k\to +\infty]{} 0\) and 
\[p_k \coloneqq  n+\alpha_k.\]

\subsection*{Acknowledgments}
The project is co-financed by:
\begin{itemize}
 \item (D.M.) Swiss National Science Foundation, project SNF $200020\text{\textunderscore}219429$;
 \item (K.M.) the Polish National Agency for Academic Exchange within Polish Returns Programme - BPN/PPO/2021/1/00019/U/00001;
 \item (K.M.) the National Science Centre, Poland grant No. 2023/51/D/ST1/02907;
 \item  (R.R) project No. 2021/43/P/ST1/01501 co-funded by the National Science Centre and the European Union Framework Programme for Research and Innovation Horizon 2020 under the Marie Skłodowska-Curie grant agreement No. 945339;
  \item  Thematic Research Programme \emph{Geometric Analysis and PDEs}, which received funding from the University of Warsaw via IDUB (Excellence Initiative Research University).
\end{itemize}
This work started when D.M. was in the University of Warsaw, he would like to thank it for its hospitality and excellent working conditions.

\section{Compactness outside concentration points}\label{sec:Compactness}

One first crucial argument to obtain Theorem \ref{th:main_bubbling} is an \(\e\)-compactness result. Such a result allows us to say that outside a finite number of points located on the boundary \(\p \Sigma^n\), a sequence of \(p_k\)-harmonic maps with free boundary in a sphere with bounded \(n\)-energy strongly converges towards its limit. The goal of this section is to prove \Cref{pr:Convergence} below, that is to say, if the $n$-energy is uniformly small on a given ball, then we obtain $\C^1_{\loc}$ estimates. We start with a uniform \(\e\)-regularity result.

\begin{proposition}\label{prop:uniform_epsilon_reg}
Let \(p_0>n\). There exists \(\e_0>0\) and \(\beta_0>0\) depending only on $p_0,n,\Sigma,g$, and $d$ such that the following holds.
Let $p\in[n,p_0]$ and \( u\in W^{1,p}(\Sigma^n;\R^d)\) satisfy the following properties:
\begin{enumerate}[label=\arabic*)]
\item
\begin{equation}\label{eq:alpha_eq}
\left\{
		\begin{array}{rcll}
		|u|&=&1 &\text{ on }\p \Sigma^n\\
 \int_{\Sigma^n} \scal{|du|_g^{p-2} d u}{d \Phi}_g\, d\vol_g&=&0 &\text{ for all } \Phi \in W^{1,p}(\Sigma^n;\R^d) \text{ with } \Phi \cdot u =0 \text{ on } \p \Sigma^n,
 \end{array}
 \right.
\end{equation}
\item the uniform bound $E_{p}(u) \leq M$,
\item for some \(x_0\in \Sigma^n\) and $R>0$, it holds
\begin{equation}\label{eq:assumption_eps_reg}
\int_{  B_{(\Sigma^n,g)}(x_0,R)} |du|_g ^n\, d \vol_g <  \e_0^n.
\end{equation}
\end{enumerate}
Then, there exists $C>0$ depending only on $M,\,n,\,\Sigma,\,g,\,d,\,R$, and $p_0$ such that $u~\in~\C^{0,\beta_0}\left(B_{(\Sigma^n,g)}(x_0,R/2);\R^d\right)$ with the estimate
\[
\|u\|_{\C^{0,\beta_0}(B_{(\Sigma^n,g)}(x_0,R/2))} \leq C.
\]
\end{proposition}
%
Several remarks are in order before starting the proof of the above proposition. First, the interior regularity for \(p\)-harmonic (vectorial) functions is known (see, for instance, \cite{Tolksdorf_1983,Uhlenbeck_1977}) and does not require \(\e\)-smallness. It can also be checked that we can obtain uniform \(\C^{1,\beta}\) regularity in the interior of \(\Sigma^n\) for \((n+\alpha)\)-harmonic maps with the previous techniques, provided that $\alpha\geq 0$ is bounded from above. Hence, we will focus on the regularity near boundary points.

Second, a maximum principle for \(p\)-harmonic (vectorial) functions is available, see \cite[Lemma A.1]{MazowieckaRodiacSchikorra} or \cite[Proposition 2.11]{martino2024}, and this allows us to deduce that a solution $u\colon\Sigma\to \R^d$ to \eqref{eq:alpha_eq} such that $|u|=1$ on $\p \Sigma$ satisfies
\begin{equation}\label{eq:Linfty_bound}
	\|u\|_{L^\infty(\Sigma^n)}\leq 1.
\end{equation}
Third, near a boundary point \(x_0\), we can find an open set \(V\) contained in \(B_{(\Sigma^n,g)}(x_0,R)\) and a smooth diffeomorphism \(\psi\colon B(0,r)^+ \rightarrow V\), for some \(r>0\). By making a change of variables, we obtain that the new map \(\tilde{u}\coloneqq u\circ \psi\) satisfies the following estimate, where $\tilde{g} = \psi^*g$:
\begin{equation*}
\int_{B(0,r)^+}|d\tilde{u}|_{\tilde{g}}^{n} \ d \vol_{\tilde{g}} < \e_0.
\end{equation*}
Furthermore, \(\tilde{u}\) satisfies the following system:
\begin{equation*}
\int_{B(0,r)^+} \scal{ |d\tilde{u}|_{\tilde{g}}^{p-2} d \tilde{u}}{ d \Phi}_{\tilde{g}}\, d\vol_{\tilde{g}}=0,
\end{equation*}
for all \(\Phi\in W^{1,p}(B(0,r)^+;\R^d)\) such that \(\Phi \cdot \tilde{u}=0\) on \(\p_0B(0,r)^+\) where \(\p_0B(0,r)^+\) is defined in \eqref{eq:def_partial_D0}.
Thus, for the proof of Proposition \ref{prop:uniform_epsilon_reg}, we may restrict ourselves to the case where \(\Sigma^n\) is an open half-disk.

The first step of the proof of Proposition \ref{prop:uniform_epsilon_reg} is to rewrite the weak formulation of the equation in a more convenient way. This is done in the following lemma due to \cite[Lemma 2.4]{MazowieckaRodiacSchikorra}.

\begin{lemma}[{{\cite[Lemma 2.4]{MazowieckaRodiacSchikorra}}}]\label{lem:2nd_rewritting}
Let \(u\in W^{1,p}(B(0,r)^+;\R^d)\) satisfy \eqref{eq:alpha_eq} and let \[A_{ij}\coloneqq u^i\, d u^j-u^j\, d u^i.\] Then, for all \(1\leq i \leq n\) and for all \(\phi \in W^{1,p}(B(0,r)^+)\), it holds
\begin{align*}
 & \int_{B(0,r)^+} \scal{|du|_g^{p-2}d u^i}{d \phi}_g\, d\vol_g \\
 = & \int_{B(0,r)^+} \scal{ |du|_g^{p-2}d u^k }{ A_{ik}\phi}_g  +\scal{ |du|_g^{p-2} d u^i}{ d (|u|^2-1) \phi}_g\, d\vol_g \\
 &+\frac12 \int_{B(0,r)^+} \scal{|du|_g^{p-2}d \phi}{ d (|u|^2-1)u^i}_g\, d\vol_g.
\end{align*}
\end{lemma}

We are now ready to sketch the proof of Proposition \ref{prop:uniform_epsilon_reg}.

\begin{proof}[Proof of Proposition \ref{prop:uniform_epsilon_reg}]
The proof follows the argument of \cite{MazowieckaRodiacSchikorra}. Thus, we only sketch it here to show the reader that the constants involved in the estimates can be made uniform in the exponent \(p\in[n,p_0]\).

As mentioned in one of the remarks following Proposition \ref{prop:uniform_epsilon_reg}, we can assume that \(x_0=0\) and \(\Sigma^n\cap B_{(\Sigma^n,g)}(x_0)=B(0,r)^+\). We notice that, since \(B(0,r)^+\) is a Lipschitz domain, we can extend each \(u\) to \(\R^n\) in such a way that \(\|u\|_{W^{1,n}(\R^n)}\leq C \|u\|_{W^{1,n}(B(0,r)^+)}\) (see, e.g., \cite[Theorem 4.7]{Evans_2015}).  Let \(y_0\in B(0,r)^+\), let \( 0<\r<\frac{r}{4}\) be such that \(B(y_0,4\rho)\subset B(0,r)^+\).  For \(\eta\) a cut-off function satisfying \(\eta \in \C^\infty_c(\R^n)\) with \(\eta\equiv 1\) in \(B(0,1)\) and \(\eta\equiv 0\) in \(B(0,2)^c\) we define
\begin{align*}
	& \eta_{B(y_0,\r)}\coloneqq \eta \left(\frac{x-y_0}{\rho} \right), \\
	& (u)_{B(y_0,2\r)^+}=\frac{1}{|B(y_0,2\r)^+|}\int_{B(y_0,2\r)^+}u \, d \vol_g, \\
	& \tilde{u}\coloneqq  \eta_{B(y_0,\r)}(u-(u)_{B(y_0,2\r)^+}), \\
 	&   \hat{u}\coloneqq  (1-\eta_{B(y_0,\r)}) \eta_{B(y_0,\r)} (u-(u)_{B(y_0,2\r)^+})=(1-\eta_{B(y_0,\r)})\tilde{u}.
\end{align*}
 Since \(\eta_{B(y_0,\r)}\equiv 1\) in \(B(y_0,\r)\)
\begin{equation*}
\int_{B(y_0,\r)^+} |d u|_g^{p}\, d\vol_g \leq \int_{B(y_0,\r)^+} \scal{ |du|_g^{p-2}d\tilde{u}}{d\tilde{u}}_g\, d\vol_g.
\end{equation*}
We observe that
\begin{align}
\langle d \tilde{u},d \tilde{u} \rangle_g =&\ \langle du,d\tilde{u}\rangle_g-\langle du,d\hat{u}\rangle_g \nonumber\\
 & -\langle d \eta_{B(y_0,\r)} \otimes \tilde{u},du \rangle_g+ \langle (u-(u)_{B(y_0,2\r)^+})\otimes d \eta_{B(y_0,\r)} ,d \tilde{u} \rangle_g \label{eq:2.10}.
\end{align}
Using \(|d \eta_{B(y_0,\r)}|\leq \frac{C}{\r}\), Young's, and Poincar\'e's inequalities we obtain that for any \(\mu>0\),
\begin{multline*}
\left| \int_{B(y_0,\r)^+} |d u|_g^{p-2} \langle (\tilde{u}\otimes d \eta_{B(y_0,\r)} ),d u \rangle_g\, d\vol_g\right| \leq \frac{C}{\rho}\int_{B(y_0,2\r)^+\setminus B(y_0,\r)^+}|du|_g^{p-1} |\tilde{u}|\, d\vol_g \\
 \leq C \left( \frac{1}{\mu}\frac{p-1}{p}\int_{B(y_0,2\r)^+\setminus B(y_0,\r)^+}|du|_g^p\, d\vol_g +\frac{1}{p}\mu^{p-1} \int_{B(y_0,2\r)^+}|du|_g^p\, d\vol_g \right).
\end{multline*}
We thus get, for a constant \(C\) which does not blow up as long as $p$ is bounded:
\begin{align*}
\int_{B(y_0,\r)^+} |du|_g^p\, d\vol_g  \leq &\ \left|\int_{B(y_0,\r)^+} |du|_g^{p-2}\langle du,d \tilde{u}\rangle_g \right|+\left| \int_{B(y_0,\r)^+} |du|_g^{p-2} \langle du,d \hat{u}\rangle_g \right| \nonumber \\
&+C\left( \frac{1}{\mu}\int_{B(y_0,2\r)^+\setminus B(y_0,\r)^+}|du|_g^p +\mu^{p-1}\int_{B(y_0,2\r)^+}|du|_g^p\right).
\end{align*}
We now estimates the two terms \[\left|\int_{B(y_0,\r)^+} |du|_g^{p-2}\langle du,d \tilde{u} \rangle_g \, d\vol_g \right|+\left| \int_{B(y_0,\r)^+} |du|^{p-2} \langle du,d \hat{u} \rangle_g \, d\vol_g\right|.\]
 Actually we focus only on the first term since the second one can be treated analogously. By observing that \(\tilde{u}\in W^{1,p}\cap L^\infty(\R^n;\R^d)\) and by using the equation satisfied by \(u\) in the form of Lemma \ref{lem:2nd_rewritting}, for all \(i\in\{1,\ldots,n\}\), we write
\begin{equation*}
\int_{B(y_0,\r)^+} \scal{ |du|_g^{p-2}d u^i }{ d \tilde{u}^i}_g\, d\vol_g =I+II+\frac12 III,
\end{equation*}
with
\begin{align*}
& I=\int_{B(y_0,\r)^+} \scal{ |du|_g^{p-2}d u^k}{ A_{ik} \tilde{u}^i}_g\, d\vol_g, \\
&  II= \int_{B(y_0,\r)^+} \scal{ |du|_g^{p-2} d u^i }{ d (|u|^2-1) \tilde{u}^i}_g\, d\vol_g, \\
& III= \int_{B(y_0,\r)^+} \scal{ |du|_g^{p-2} d \tilde{u}^i}{ d (|u|^2-1)u^i}_g\, d\vol_g.
\end{align*}
The terms \(I, II\) and $III$ are estimated as in \cite[Proposition 2.1]{MazowieckaRodiacSchikorra}. The key point is that all these terms contain products of divergence-free and curl-free quantities which makes possible the use of the div-curl lemma in \cite{Coifman_Lions_Meyer_Semmes_1993}. It can be checked that the constant appearing in that lemma does not blow-up as long as $p\in[n,p_0]$. The div-curl lemma involves the BMO semi-norm
\begin{equation*}
[f]_{\BMO} \coloneqq \sup_B \frac{1}{|B|}\int_{B} |f-(f)_B|,
\end{equation*}
and for example, for the term \(I\) we find an estimate of the form
\begin{align*}
|I| & \leq C \||du|^{p-2}A_{ik}\|_{L^{p'}(B(y_0,4\rho)^+)}\|du\|_{L^p(B(yu_0,4\rho)^+)}[\tilde{u}]_{BMO} \\
& \leq C\|du\|_{L^p(B(y_0,4\rho)^+)}^p[\tilde{u}]_{BMO}.
\end{align*}
 The Poincar\'e inequality allows to estimate all the BMO semi-norms by 
\begin{equation*}
[f]_{\BMO} =\sup_B \frac{1}{|B|}\int_{B} |f-(f)_B|\leq C(n) \sup_B \|df \|_{L^n(B)}.
\end{equation*}
Then, the key difference is that we use the assumption
\[
\int_{B_{(\Sigma^n,g)}(x_0,R)} |du|_g ^n\, d\vol_g < \e_0^n
\]
to derive decay estimates on $\|du\|_{L^p(B(y_0,\rho)^+)}^p$,
which in turn lead to the desired Hölder regularity via the Campanato characterization of Hölder spaces (see \cite[proof of Proposition 3.1]{MazowieckaRodiacSchikorra}).

One can see that all the constants appearing in the proof remain bounded for \( p \in [n, p_0] \).

\end{proof}

Once we have obtained uniform \(\C^{0,\beta}\) estimates up to the boundary for \(p\in[n,n+1]\), we aim to improve these into \(\C^{1,\delta}\) estimates. The strategy involves performing a reflection across the boundary \(\p \Sigma^n\) to reduce the problem of boundary regularity to a problem of interior regularity for a modified equation. This is the content of the following lemma.

\begin{lemma}\label{lm:reflexion}
Let $p_0>n$ and $p\in[n,p_0]$. Consider \( u \in W^{1,p}(\Sigma^n;\R^d)\) satisfying \eqref{eq:alpha_eq}. Let \( x_0\in \p \Sigma^n\) and consider a chart near $x_0$ so that we can identify $B_{(\Sigma,g)}(x_0,R)$ with $B(0,R)^+$ for $R>0$ small enough.

Assume that \(u\in \C^{0,\beta}\left(  \overline{B(0,R)^+};\R^d \right)\) with the estimate \( [u]_{\C^{0,\beta}(B(0,R))^+}\leq c_{\hol}\). Then, there exist \(v\in W^{1,p}\cap \C^{0,\beta}(B(0,R_1);\R^d)\), \(0<R_1<R\), and \(C,c_0>0\) depending only on $n,p_0,d,c_{\hol},\Sigma$, and $g$ such that
\begin{equation}\label{eq:system_reflexion}
	\left\{
	\begin{aligned}
		[v]_{\C^{0,\beta}(B(0,R_1))]} & \leq C[u]_{\C^{0,\beta}(B(0,R)^+)},\\
		v&=u \text{ in } B(0,R_1)^+, \\
		\lap_{p,h}v &=f(x,v,dv) \text{ in }B(0,R_1),
	\end{aligned}
	\right.
\end{equation}
where \(h\) is a Lipschitz metric on $B(0,R_1)$ satisfying $\mathrm{Lip}(h)\leq 12\|g\|_{\C^1}$, and \(f(\cdot,v,dv)\in L^1(B(0,R_1);\R^d)\) satisfies the pointwise estimate \( |f(\cdot,v,dv)|\leq c_0|dv|^{p}\).
\end{lemma}

\begin{proof}
As mentioned previously, the idea is to use a reflection of \(u\) across \(\p \Sigma^n\), a technique previously used by Scheven in \cite[Lemma 3.1]{scheven2006} (see also, \cite[Proposition 5.3]{MazowieckaRodiacSchikorra}). For clarity we divide the proof into several steps.

\underline{Step 1: Definition of the geometrical reflection} 

We consider the reflection across the axis $\{x_n=0\}$ denoted by \( \sigma :  B(0,R) \to \R^d\) and defined by
\( (x_1,\dots,x_n) \mapsto (x_1,\dots,-x_n)\).
The differential of $\sigma$ satisfies
\begin{align*}
	d\sigma = \begin{pmatrix}
		\text{id}_{n-1} & 0\\
		0 & -1
	\end{pmatrix}, & & d\sigma (d\sigma)^{\top} = \text{id}_n.
\end{align*}
Since $u$ is Hölder continuous up to the boundary, there exists $R_1\in(0,R)$ depending only on $c_{\hol}$ such that $|u|>\frac{1}{2}$ on $B(0,R_1)^+$. Given $q\in \R^d$, we define $\iota(q) \coloneqq \frac{q}{|q|^2}$, that is to say, $\iota$ is the inversion with respect to $\s^{d-1}$. We now define the extension $v$ of $u$ to the full ball $B(0,R_1)$:
\begin{equation*}
	\forall x\in B(0,R_1),\quad v(x) = \left\{
	\begin{aligned}
		&u(x) & \text{ if } x\in B(0,R_1)^+,\\
		&\iota\circ u\circ \sigma(x) & \text{ if }x\in B(0,R_1)^-.
	\end{aligned}
	\right.
\end{equation*}
We also extend the metric $g$ to the ball $B(0,R_1)$:
\begin{equation*}
\forall x\in B(0,R_1),\quad h(x)\coloneqq \begin{cases}
g(x) & \text{ if } x \in B(0,R_1)^+, \\
(\sigma^{-1})^*g(x) & \text{ if } x \in B(0,R_1)^-.
\end{cases}
\end{equation*}
\begin{claim}
	The metric $h$ is Lipschitz with the estimate $\mathrm{Lip}(h)\leq 12\|g\|_{\C^1}$.
\end{claim}
\begin{proof}
	By definition, we have $h\in \C^{\infty}(B(0,R_1)^+)$ and $h\in \C^{\infty}(B(0,R_1)^-)$. We only need to check the regularity at the junction. Let $x\in B(0,R_1)$ satisfying $x_n=0$. For any $y\in B(0,R_1)\setminus \left(\R^{n-1}\times\{0\} \right)$, it holds
	\begin{align*}
		|h(x) - h(y)| \leq | h(x) - h(x_1,\ldots,x_{n-1},y_n)| + |h(x_1,\ldots,x_{n-1},y_n) - h(y)|.
	\end{align*}
	Since either $(x_1,\ldots,x_{n-1},y_n),y\in B(0,R_1)^+$ or $(x_1,\ldots,x_{n-1},y_n),y\in B(0,R_1)^-$, the second term satisfies
	\begin{align*}
		|h(x_1,\ldots,x_{n-1},y_n) - h(y)| \leq \left( \|h\|_{\C^1(B(0,R_1)^+)} + \|h\|_{\C^1(B(0,R_1)^-)} \right)|x-(y_1,\ldots,y_{n-1},0)|.
	\end{align*}
	For the first term, since \(g\in \C^\infty\left( \overline{B(0,R_1)^+} \right)\) and \((\sigma^{-1})^*g\in  \C^\infty\left( \overline{B(0,R_1)^-} \right)\), we have:
	\begin{align*}
		| h(x) - h(x_1,\ldots,x_{n-1},y_n)| \leq \left( \|h\|_{\C^1(B(0,R_1)^+)} + \|h\|_{\C^1(B(0,R_1)^-)} \right)|y_n|.
	\end{align*}
	Thus, we obtain
	\begin{align*}
		|h(x) - h(y)| \leq 2\left( \|h\|_{\C^1(B(0,R_1)^+)} + \|h\|_{\C^1(B(0,R_1)^-)} \right)|x-y|.
	\end{align*}
	If now $y\in B(0,R_1)^+$ and $z\in B(0,R_1)^-$, there exists $x\in B(0,R_1)$ such that $x_n=0$ and 
	\begin{align*}
		|y-x|+|z-x| \leq 3|y-z|.
	\end{align*}
	Thus, it holds
	\begin{align*}
		|h(y) - h(z)| & \leq |h(y) - h(x)| + |h(x) - h(z)| \\
		& \leq 6\left( \|h\|_{\C^1(B(0,R_1)^+)} + \|h\|_{\C^1(B(0,R_1)^-)} \right)|z-y|.
	\end{align*}
\end{proof}

\underline{Step 2: System satisfied by $v$}

Let \(\Phi \in \C^\infty_c(B(0,R_1);\R^d)\). We denote 
\begin{equation*}
\tilde{u}(x) \coloneqq  \begin{cases}
u (x) & \text{ for } x \in B(0,R_1)^+\\
u\circ\sigma (x) & \text{ for } x \in B(0,R_1)^-, 
\end{cases} \quad  \quad \tilde{\Phi} \coloneqq   \begin{cases}
\Phi(x)  & \text{ for } x \in B(0,R_1)^+\\
\Phi\circ\sigma (x) & \text{ for } x \in B(0,R_1)^-.
\end{cases} 
 \end{equation*}
In particular, we have $v=\iota\circ \tilde{u}$ in \(B(0,R_1)^-\). For \(q\in \R^n\) we set 
\begin{equation*}
\Xi(q)\coloneqq d \iota (q)=\frac{\text{id}}{|q|^2}-2 \frac{q\otimes q}{|q|^4}.
\end{equation*}
We observe that for all \(x\in B(0,R_1)^-\) we have
\begin{align}\label{eq:relation_dv_du}
d v(x) =\Xi(\tilde{u}(x)) d\tilde{u}(x).
\end{align}
We have that \(\Xi\) is symmetric and \(|\Xi(q)w|=|w|/|q|^2\) for all \(q,w\in \R^n\). Indeed, we can see that \( q/|q|\) is an eigenvector of \(\Xi\) associated to the eigenvalue \(-1/|q|^2\) and every orthonormal basis of \( (q/|q|)^\perp\) is a basis of the eigenspace associated to the eigenvalue \(1/|q|^2\). In particular, it holds
\begin{align*}
|dv|_h=\begin{cases}
|du|_g & \text{ in } B(0,R_1)^+ \\
\frac{|d\tilde{u}|_{\tilde{g}}}{|\tilde{u}|^2} & \text{ in } B(0,R_1)^-,
\end{cases}
\end{align*}
and since \(1/2\leq |\tilde{u}|\leq 1\) we have that \(|dv|_h\) and \(|d\tilde{u}|_h\) are comparable in \(B(0,R_1)\).
For \(q\in \mathbb{S}^{d-1}\) we define \(\Pi(q)\coloneqq \text{id}-q\otimes q\), this is the orthogonal projection on \(T_q \mathbb{S}^{d-1}\). We also define \(\Pi^\perp (q)\coloneqq q\otimes q\), the orthogonal projection on \( (T_q \mathbb{S}^{d-1})^\perp=\text{span}\{q\}\).
Since \(\iota^{-1}=\iota\), it holds, for any $x\in B(0,R_1)^-$:
\begin{align}\label{eq:prop_sigma}
\Xi(v(x))^{-1}=(d\iota (v(x)))^{-1} =d \iota \big(\iota \circ v(x)\big) = d\iota \big(u \circ \sigma (x) \big)=\Xi(\tilde{u}(x))=\Xi(v(\sigma(x))).
\end{align}
For all \(\Phi \in \C^\infty_c(B(0,R_1);\R^d)\) we define
\begin{align*}
\Phi_e(x) &= \frac12 [ \Phi(x)+\Xi(v(\sigma(x)))\Phi(\sigma(x))], \\
\Phi_a(x) &= \frac12 [ \Phi(x)-\Xi(v(\sigma(x)))\Phi(\sigma(x))].
\end{align*}
From \eqref{eq:prop_sigma} we can check that the following relations hold for all $x\in B(0,R_1)^-$:
\begin{align}
\Phi_e(\sigma(x))&=\frac12 [ \Phi(\sigma(x))+\Xi(v(x))\Phi(x)]=\Xi (v(x))\Phi_e(x), \label{eq:tilde_phie}\\
\Phi_a(\sigma(x))& = \frac12 [\Phi(\sigma(x))-\Xi(v(x))\Phi(x)]=-\Xi(v(x))\Phi_a(x).\label{eq:tilde_phia}
\end{align}
Now we remark, from the definition of \(\Pi\) and \(\Xi\) that for all \(x_0\in \p \R^n_+\cap B(0,R_1)\) we have
\begin{align*}
\Phi_e(x_0) &=\frac12 \Big[ \Phi(x_0)+\Xi (v(\sigma(x_0)))\Phi(\sigma(x_0)) \Big]\\
& =\frac12 \Big[\Pi(u(x_0))\Phi(x_0)+\Pi^\perp(u(x_0))\Phi(x_0) +\Pi(u(x_0))\Phi(x_0)-\Pi^\perp(u(x_0))\Phi(x_0) \Big] \\
&=\Pi(u(x_0))\Phi(x_0) \in T_{u(x_0)}\mathbb{S}^{d-1}.
\end{align*}
Hence \(\Phi_e\) is admissible as a test function for the free-boundary condition \eqref{eq:main_equation} and 
\begin{equation}\label{eq:admissible}
\int_{B(0,R_1)^+} |du|_g^{p-2}\langle du, d \Phi_e\rangle_g \, d \vol_g=0.
\end{equation}
A change of variables combined with \eqref{eq:tilde_phie} shows that 
\begin{align}\label{eq:phie}
0=\int_{B(0,R_1)^+} |du|^{p-2}_g \langle du,d \Phi_e\rangle_g\, d\vol_g &= \int_{B(0,R_1)^-} |d \tilde{u}|_h^{p-2}\langle d \tilde{u}, d (\Phi_e\circ \sigma)  \rangle_h\, d \vol_h \nonumber\\
&= \int_{B(0,R_1)^-}|d\tilde{u}|^{p-2}_h \langle d \tilde{u}, d(\Xi(v)\Phi_e) \rangle_h\, d \vol_h.
\end{align}
On the other hand, by using \eqref{eq:tilde_phia} we find that 
\begin{align}\label{eq:phia}
\int_{B(0,R_1)^+} |du|^{p-2}_g \langle du,d \Phi_a\rangle_g\, d\vol_g &= \int_{B(0,R_1)^-} |d \tilde{u}|_h^{p-2}\langle d \tilde{u}, d (\Phi_a\circ \sigma) \rangle_h\, d \vol_h \nonumber\\
&= -\int_{B(0,R_1)^-}|d\tilde{u}|^{p-2}_h \langle d \tilde{u}, d(\Xi(v)\Phi_a) \rangle_h\, d \vol_h.
\end{align}
We have that $v= \frac{\tilde{u}}{|\tilde{u}|^2}$ on $B(0,R_1)^-$ and thus, we obtain
\begin{equation}\label{eq:ling_sigma_uv}
\forall x\in B(0,R_1)^-,\quad \Xi (v(x))=|\tilde{u}(x)|^4\, \Xi (\tilde{u}(x)).
\end{equation}
Since $\Phi=\Phi_e+\Phi_a$, by using \eqref{eq:phie},\eqref{eq:phia}, and \eqref{eq:ling_sigma_uv}, we find that
\begin{align*}
& \int_{B(0,R_1)} |d v|^{p-2}_h \langle dv ,d \Phi\rangle_h d \vol_h \\
= &\ \int_{B(0,R_1)^+} |d v|^{p-2}_h \langle dv ,d (\Phi_a-\Phi_e)\rangle_h d \vol_h + \int_{B(0,R_1)^-} |d v|^{p-2}_h \langle dv ,d \Phi\rangle_h d \vol_h \\
= &\ -\int_{B(0,R_1)^-}|d\tilde{u}|_h^{p-2} \langle  d \tilde{ u}, d \big(\Xi (v)\Phi \big)\rangle_h d \vol_h  + 
\int_{B(0,R_1)^-} |d v|^{p-2}_h \langle dv ,d \Phi\rangle_h d \vol_h.
\end{align*}
We focus on the first term, and we observe that, for any \(x\in B(0,R_1)^-\) we have \(d \tilde{u}(x)=\Xi(v(x)) dv(x)\) thanks to \eqref{eq:relation_dv_du} and \eqref{eq:prop_sigma}. Then,
\begin{align*}
	& \int_{B(0,R_1)^-}|d\tilde{u}|_h^{p-2} \langle  d \tilde{ u}, d \big(\Xi (v)\Phi \big)\rangle_h d \vol_h \\
= &\ \int_{B(0,R_1)^-}|d\tilde{u}|_h^{p-2} \langle d \tilde{ u}, \big( d \Xi (v)\big) \Phi  + \Xi(v) d\Phi\rangle_h d \vol_h \\
=&\ \int_{B(0,R_1)^-} |d\tilde{u}|_h^{p-2} \langle   d \tilde{ u}, \big( d \Xi (v)\big) \Phi\rangle_h  +|\tilde{u}|^4\, |d\tilde{u}|_h^{p-2} \langle \Xi(\tilde{u}) d\tilde{u},  d\Phi\rangle_h d \vol_h \\
=& \ \int_{B(0,R_1)^-} |\tilde{u}|^{2(p-2)}|dv|_h^{p-2} \langle \Xi(v)\, dv, \big( d \Xi (v)\big) \Phi\rangle_h  +|\tilde{u}|^{4+2(p-2)}\, |dv|_h^{p-2} \langle dv,  d\Phi\rangle_h d \vol_h.
\end{align*}
We obtain
\begin{align*}
 \int_{B(0,R_1)} |d v|^{p-2}_h \langle dv ,d \Phi\rangle_h d \vol_h 
	= &\ -\int_{B(0,R_1)^-} |\tilde{u}|^{2(p-2)}|dv|_h^{p-2} \langle \Xi(v)\, dv, \big( d \Xi (v)\big) \Phi\rangle_h\, d \vol_h \\
	& + 
	\int_{B(0,R_1)^-} \left(1 - |\tilde{u}|^{2p}\right) |d v|^{p-2}_h \langle dv ,d \Phi\rangle_h d \vol_h.
\end{align*}
We decompose the left-hand side as an integral on $B(0,R_1)^-$ and $B(0,R_1)^+$ to obtain
\begin{align*}
	& \int_{B(0,R_1)^-}  |\tilde{u}|^{2p} |d v|^{p-2}_h \langle dv ,d \Phi\rangle_h\, d \vol_h + \int_{B(0,R_1)^+} |d v|^{p-2}_h \langle dv ,d \Phi\rangle_h d \vol_h \\
	= & -\int_{B(0,R_1)^-} |\tilde{u}|^{2(p-2)}|dv|_h^{p-2} \langle \Xi(v)\, dv, \big( d \Xi (v)\big) \Phi\rangle_h\, d \vol_h .
\end{align*}
Now we define 
\begin{equation}
m(x)= \begin{cases} 
1 &\text{ if } x \in B(0,R_1)^+, \\
|\tilde{u}(x)|^{2p} &\text{ if } x \in B(0,R_1)^-.
\end{cases}
\end{equation}
We have $m\in L^{\infty}\cap W^{1,p}(B(0,R_1))$ with \(m\geq 1/2\) and the pointwise estimate $|dm|\leq p 2^{2p}|du\circ\sigma|_{g\circ \sigma}$. We then find that 
\begin{align*}
& \int_{B(0,R_1)} |d v|^{p-2}_h \langle dv ,d(m \Phi)\rangle_h d \vol_h \\
=& -\int_{B(0,R_1)^-} \left[|\tilde{u}|^{2(p-2)}|dv|_h^{p-2} \langle \Xi(v)\, dv, \big( d \Xi (v)\big) \Phi\rangle_h + |d v|^{p-2}_h \langle dv ,(dm) \Phi\rangle_h \right]\, d \vol_h .
\end{align*}
We consider the change of variables $\Psi\coloneqq m\Phi$ and obtain for any $\Psi\in W^{1,p}_0(B(0,R_1);\R^d)$:
\begin{align*}
	& \int_{B(0,R_1)} |d v|^{p-2}_h \langle dv ,d\Psi \rangle_h d \vol_h \\
	=& -\int_{B(0,R_1)^-} |\tilde{u}|^{-4}|dv|_h^{p-2} \langle \Xi(v)\, dv, \big( d \Xi (v)\big) \Psi \rangle_h + |d v|^{p-2}_h \langle dv ,(d\log m) \Psi\rangle_h \, d \vol_h .
\end{align*}
Thus, \(v\) satisfies a system of the form \eqref{eq:system_reflexion}.
\end{proof}

Finally, we obtain the following $\eps$-regularity result for free boundary $p$-harmonic maps into the sphere.

\begin{proposition}\label{pr:Convergence}
	Let \(p_0>n\). There exists \(\e_0>0\) and \(\beta_0>0\) depending only on $n,\Sigma,g,p_0$ and $d$ such that the following holds. 
	Let $p\in[n,p_0]$ and \( u\in W^{1,p}(\Sigma^n;\R^d)\) satisfying the following properties:
	\begin{enumerate}[label=\arabic*)]
\item \begin{equation*}
\left\{
		\begin{array}{rcll}
		|u|&=&1 &\text{ on }\p \Sigma^n\\
 \int_{\Sigma^n} \scal{|du|_g^{p-2} d u}{d \Phi}_g\, d\vol_g&=&0 &\forall \Phi \in W^{1,p}(\Sigma^n;\R^d) \text{ with } \Phi \cdot u =0 \text{ on } \p \Sigma^n,
 \end{array}
 \right.
\end{equation*}
		\item the uniform bound $E_{p}(u) \leq M$,
		\item for some \(x_0\in \Sigma^n\) and $R>0$, it holds
		\begin{equation*}
			\int_{  B_{(\Sigma^n,g)}(x_0,R)} |du|_g ^n\, d \vol_g<  \e_0^n.
		\end{equation*}
	\end{enumerate}
	Then there exists $C>0$ depending only on $M,n,\Sigma,g,d,R$ and $p_0$ such that $u\in \C^{1,\beta_0}\left(B_{(\Sigma^n,g)}(x_0,R/2);\R^d\right)$ with the estimate $\|u\|_{\C^{1,\beta_0}(B_{(\Sigma^n,g)}(x_0,R/2))} \leq C$.
\end{proposition}

\begin{proof}
	If $x_0\in\Int(\Sigma)^n$, then we have the system $\lap_{p,g}u = 0$ in a neighbourhood of $x_0$. Hence, the uniform estimate follows from \cite[Main Theorem]{Uhlenbeck_1977}.
	
	If $x_0\in\p\Sigma^n$, we consider the map $v$ introduced in \Cref{lm:reflexion}. Thanks to Proposition \ref{prop:uniform_epsilon_reg} we can apply \cite[Theorem 3.1]{martino2024} and we find that $v\in \C^{1,\beta_1}_{\loc}(B(0,R_1);\R^d)$ with an estimate on $[dv]_{\C^{0,\beta_1}(B(0,R_1/2))}$ depending only on $M,n,\Sigma,g,d,R$ and $p_0$. Therefore, we obtain that $u$ is uniformly bounded in $\C^{1,\beta_1}\left( \overline{ B(0,R_1/2)^+};\R^d \right)$.
\end{proof}

\section{Properties of the bubbles}\label{sec:Bubbles}

In this section, we first prove two properties which are fundamental to obtain the energy identity. The first one is an energy gap in \Cref{lem=quantum_energy}. The second one is a singularity removability result in \Cref{prop:sing_remov_result}. Then we give some properties of the bubbles, defined in \eqref{def:bubble}, that will also be useful in the proof of the main result Theorem \ref{th:main_bubbling}.

\begin{lemma}\label{lem=quantum_energy}
	There exists $\eps_b=\eps_b(n)>0$ such that the following holds. Let $\Omega\subset \R^n$ be a smooth, bounded, and connected open set, and let $u\in W^{1,n}(\Omega;\R^d)$ be a solution to \eqref{eq:system_formal} for $\Sigma^n=\overline{\Omega}$. If $\|du\|_{L^n(\Omega)}< \e_b$, then $u$ is constant.
\end{lemma}
\begin{proof}
	Let $x_0\in \pl\Omega$ and consider \(\iota_{x_0}\) the inversion in $\R^n$ around $B(x_0,1)$ defined by $\iota_{x_0}(x)=(x-x_0)/|x-x_0|^2$. Thanks to conformal invariance of the $n$-energy, the map $v \coloneqq  u\circ \iota_{x_0}^{-1}$ is a free-boundary $n$-harmonic map as well. That is to say, $v$ satisfies on the noncompact domain $\tilde{\Omega} \coloneqq  \iota_{x_0}(\Omega)$:
	\begin{align*}
		\left\{
		\begin{array}{rcll}
			\lap_n v &=& 0 & \text{ in }\tilde{\Omega},\\
			|v| &=& 1 & \text{ on }\pl\tilde{\Omega},\\
			|\g v|^{n-2} \pl_\nu v &\perp& T_v \s^{n-1} & \text{ on }\pl\tilde{\Omega}.
		\end{array}
		\right.
	\end{align*}
	Furthermore, $v$ satisfies $\|dv\|_{L^n(\tilde{\Omega})}\leq \e_b$.	If $\eps_b>0$ is small enough, we can apply the $\eps$-regularity of \cite[Theorem 1.2]{MazowieckaRodiacSchikorra}. Indeed, the arguments of \cite{MazowieckaRodiacSchikorra} can be applied to unbounded domains with bounded geometry, see the first paragraph in \cite[Section 2]{MazowieckaRodiacSchikorra}, under the assumption $v\in L^\infty$. This is satisfied in our case: since $|v|=1$ on $\p\tilde{\Omega}$, we obtain $|v|\leq 1$ in $\tilde{\Omega}$ by \cite[Proposition 2.11]{martino2024}. Hence, there exist $\alpha>0$ and $C>1$ such that for any $p\in \pl\tilde{\Omega}$ and any $R>0$, it holds
	\begin{align*}
		[v]_{C^{0,\alpha}\left( \tilde{\Omega}\cap B(p,R) \right) } \leq \frac{C}{R^\alpha} \|\g v\|_{L^n(\tilde{\Omega})}.
	\end{align*}
	By taking the limit $R\to +\infty$, we deduce that $v$ is constant and $u=v\circ \iota_{x_0}$ as well.
\end{proof}
We now prove the singularity removability result.
\begin{proposition}\label{prop:sing_remov_result}
Let \(x_0\in \p \Sigma^n\), let \(u\in \C^1( \Sigma^n \setminus \{x_0\};\R^d)\cap W^{1,n}(\Sigma^n;\R^d)\) be such that \(|u|=1\) on \(\p \Sigma^n\) and such that  
\begin{equation}\label{eq:free_bdr}
\int_{\Sigma^n} \scal{ |du|^{n-2}du}{d\Phi}_g\, d\vol_g =0,
\end{equation}
for all \(\Phi\in W^{1,n}(\Sigma^n;\R^d)\) such that \(\supp \Phi \subset \Sigma^n\setminus \{x_0\}\) and such that \(\Phi\cdot u=0\) on \(\p \Sigma^n \setminus \{x_0\}\). 
Then \(u\) extends to a \(\C^1\) map that satisfies the equation \eqref{eq:free_bdr} for all \(\Phi\in W^{1,n}(\Sigma^n;\R^d)\) such that \(\Phi\cdot u=0\) on  \(\p \Sigma^n\).
\end{proposition}

\begin{proof}
Let \(0<R_q\) with \(R_q\to 0\) as \(q\to \infty\). Using the fact that in the critical Sobolev space a point has zero capacity (see, e.g., \cite[Theorem 5.1.9]{Adams-Hedberg}) we can find \(\xi_q\in \C^\infty(\overline{\Sigma^n}\setminus \{x_0\})\) such that \(\xi_q =1\) on \(B_{(\Sigma^n,g)}(x_0,R_q^2)\), \(\xi_q=0\) outside \(B_{(\Sigma^n,g)} (x_0,R_q)\) and \(\|d \xi_q\|_{L^n}\rightarrow 0\) as \(q\to +\infty\). For example we can take
\begin{equation*}
\xi_q(x)=\begin{cases}
1 & \text{ if } x \in B(x_0,R_q^2), \\
\left(\log\left(\frac{1}{R_q}\right) \right)^{-1}\log\frac{R_q}{\dist_{(\Sigma,g)}(x,x_0)} & \text{ if } x \in B(x_0,R_q)\setminus B(x_0,R_q^2),\\
0 & \text{ if } x \in B(x_0,R_q)^c.
\end{cases}
\end{equation*}

 For any \( \Phi \in W^{1,n}\cap L^{\infty}(\Sigma^n;\R^d)\) with \(\Phi \cdot u=0\) on \(\p \Sigma^n\) we have that \(\Phi_q\coloneqq (1-\xi_q)\Phi\in \C^\infty_c(\overline{\Sigma^n}\setminus \{x_0\};\R^d)\) and
\begin{equation*}
d\Phi_q=d\Phi(1-\xi_q)-\Phi \otimes d\xi_q.
\end{equation*}
Since \(\Phi_q\) is admissible in \eqref{eq:free_bdr} we find that 
\begin{equation*}
\int_{\Sigma^n} \scal{|du|^{n-2}_g du}{d\Phi}_g (1-\xi_q)\, d\vol_g=\int_{\Sigma^n} \scal{ |du|^{n-2}_g du}{ \Phi \otimes d \xi_q}_g\, d\vol_g.
\end{equation*}
By using the dominated convergence for the left-hand side and the H\"older inequality for the right-hand side we find that 
\begin{equation*}
\int_{\Sigma^n} \scal{|du|^{n-2} du}{d\Phi}_g\, d\vol_g =0.
\end{equation*}
Since this is valid for any \( \Phi \in W^{1,n}(\Sigma^n;\R^d)\) with \(\Phi \cdot u=0\) on \(\p \Sigma^n\) we have found that \(u\) is a free-boundary \(n\)-harmonic map near \(x_0\) and by \cite[Theorem 1.2]{MazowieckaRodiacSchikorra}, we obtain that \(u \in \C^{1,\beta}_{\loc}(\Sigma^n;\R^d)\) for some \(\beta>0\).
\end{proof}

\begin{lemma}\label{lem:prop_bubbles}
	Let \(\omega \in W^{1,n}(\overline{\R^n_+};\R^n)\) satisfy \Cref{def:bubble}. The following properties are satisfied:
	\begin{enumerate}
		\item \(\omega\) has a limit at infinity, i.e., \(\omega(x) \xrightarrow[|x|\to +\infty]{} \omega(\infty),\)
		\item \(d \omega\) is bounded on \(\R^n_+\), i.e., there exists \(C>0\) such that \(|d\omega(x)|\leq C\) for all \(x\in \overline{\R^n_+}\).
		\item If $\omega$ is not constant, then it holds \(\int_{\R^n_+} |d \omega|^n \, dx\geq \e_b^n\) where \(\e_b^n\) is defined in Lemma \ref{lem=quantum_energy}.
	\end{enumerate}
\end{lemma}

\begin{proof}
Let \(\pi: \overline{B(0,1)}\rightarrow \overline{\R^n_+}\) be a conformal map such that \( \pi: \p B(0,1) \rightarrow \p \R^n_+\) and \(\pi(N)=0\) where \(N=(0,0,\dots,1)\). We can take, for \(x\in B(0,1)\):
\begin{align*}
	\pi(x)= \frac{1}{x_1^2+\dots+x_{n-1}^2+(1-x_n)^2}\Bigl(2x_1,\cdots, 
	2x_{n-1},1-x_1^2-\dots-x_n^2 \Bigr).
\end{align*}
We define \(\tilde{\omega}\coloneqq \omega \circ \pi\). Since \(\omega\in \C^{1,\beta}_{\text{loc}}(\overline{\R}^n_+;\R^d)\) we find that \(\tilde{\omega}\in \C^{1,\beta}(\overline{B(0,1)}\setminus \{N\};\R^d)\). By the conformal invariance of the \(n\)-energy we have that \(\int_{B(0,1)}|d \tilde{\omega}|^n dx  <+\infty\) and \(\tilde{\omega}\) satisfies the equation \eqref{eq:main_equation} in \(\overline{B(0,1)}\setminus \{N\}\). By the removability singularity result Proposition \ref{prop:sing_remov_result}, we obtain that \(\tilde{\omega}\) is in \(\C^{1,\beta}(\overline{B(0,1)};\R^d)\) and satisfies the equation in \(\overline{B}(0,1)\). This shows that \(\tilde{\omega}\) has a limit at \(N\) and that there exists \(C>0\) such that \(|d \tilde{\omega}|\leq C\) in \(\overline{B}(0,1)\). Using the inverse of the conformal map \(\pi\)  defined for \(x\in \R^n_+\) by
\begin{align*}
\pi^{-1}(x)= \frac{1}{x_1^2+\dots+x_{n-1}^2+(1+x_n)^2} \Bigl(2x_1,\cdots,
 2x_{n-1},x_1^2+\dots+x_n^2-1 \Bigr).
\end{align*}
we deduce that \(\omega\) has a limit at \(\infty\) and, noticing that \(d\pi^{-1}\) is bounded, we deduce that \(d\omega\) is bounded in all \(\overline{\R^n_+}\). In the case where $\omega$ is not constant then neither is $\tilde{\omega}$. By Lemma \ref{lem=quantum_energy}, we have that \(\int_{B(0,1)}|d \tilde{\omega}|^n \, dx \geq \e_b^n\). The conformal invariance of the \(n\)-energy then implies that \(\int_{\R^n_+}|d\omega|^n  \, dx \geq \e_b^n\).
\end{proof}

\section{Some energy comparison lemmas}\label{sec:Energy}

In this section we prove several results that will be used in the proof of our main result Theorem~\ref{th:main_bubbling}. In particular Lemma \ref{lem:comparison_2} is the key result to prove the weak energy identity when more than one generation of bubbles are involved.

The triangle inequality states that for two functions $u$ and $v$, we have \(\| d(u+v)\|_{L^p}\leq \|du\|_{L^p}+\|dv\|_{L^p}\). We will often use the following lemma, stating that if $\|dv\|_{L^p}$ is small, then the inequality is nearly an equality.

\begin{lemma}\label{lem:decoupling}
	Let $1\leq q_{\inf}\leq q_{\sup}<+\infty$ and $(q_k)_{k\in\N}\subset [q_{\inf},q_{\sup}]$ be a bounded sequence. Let \((U_k)_{k\in\N}\) be a sequence of open subsets of \((\Sigma^n,g)\), let \( (u_k)_{k\in\N}, (v_k)_{k\in\N}\) be two bounded sequences of maps in \(W^{1,q_k}(U_k;\R^d)\) with $\displaystyle{\sup_{k\in\N} }\ \|d u_k\|_{L^{q_k}(U_k)}<\Lambda$. Assume that \(\displaystyle{\lim_{k \to +\infty}} \norm{d v_k}_{L^{q_k}(U_k)} =0\). Then there exists $\theta_k\geq 0$ such that $\theta_k \le C\norm{d v_k}_{L^{q_k}(U_k)}$ with a constant $C$ depending only on $q_{\inf},q_{\sup}, \Lambda$ and
	\begin{equation*}
		\forall k\in\N,\quad \int_{U_k}|d(u_k+v_k)|_g^{q_k} \, d\vol_g =\int_{U_k}|du_k|^{q_k}_g \, d\vol_g +\theta_k.
	\end{equation*}
\end{lemma}
\begin{proof}
	We use the following inequality:  for any \(x,y\in \R^{n\times d}\) we have
	\begin{equation}\label{eq:this_inequality}
		\left||x|^{q_k}-|y|^{q_k} \right|\leq q_k |x-y| (|x|+|y|)^{q_k-1}.
	\end{equation}
	Indeed, thanks to the fundamental theorem of calculus we can write that 
	\begin{align*}
		|x|^{p_k}-|y|^{q_k} &=\int_0^1 \frac{d}{dt} (|tx+(1-t)y|^{q_k}) dt \\
		&=\int_0^1 \langle q_k (x-y),(tx+(1-t)y ) \rangle |tx+(1-t)|^{q_k-2} dt .
	\end{align*}
	But \( |tx+(1-t)y| \leq |x|+|y|\) and thus we find that \eqref{eq:this_inequality} holds.
	Now we use inequality \eqref{eq:this_inequality}, H\"older inequality and Minkowski inequality to write that
	\begin{align*}
		 & \left| \int_{U_k}\left( |d u_k+d v_k|_g^{q_k}-|d u_k|_g^{q_k}\right) \, d\vol_g\right| \\
		 \leq &\ q_k \int_{U_k} |d v_k|_g (|d u_k|_g+|d v_k|_g)^{q_k-1} \, d\vol_g \\
		 \leq &\ q_k\left(\int_{U_k} |d v_k|_g^{q_k} \, d\vol_g \right)^{\frac{1}{q_k}} \left(\int_{U_k}(|d u_k|_g+|d v_k|_g)^{q_k}\, d\vol_g \right)^{\frac{q_k-1}{q_k}} \\
		  \leq &\ q_k  \|d v_k\|_{L^{q_k}(U_k)} (\|du_k\|_{L^{q_k}(U_k)}+\|d v_k\|_{L^{q_k}(U_k)})^{q_k-1} \\
		  \leq &\ (\Lambda+1)^{q_k-1} q_k \|dv_k\|_{L^{q_k}(U_k)} \coloneqq \theta_k.
	\end{align*}
\end{proof}

We now prove that if for a free-boundary $p$-harmonic map $u$ one knows that $E_n(u)$ is small, then we can estimate its $p$-energy by the $p$-energy of its unconstrained $p$-harmonic extension.
\begin{lemma}\label{lem:comparison_2}
Let \(B\subset \R^n\) be a ball, $p\in[n,n+1]$ and $g$ be a smooth metric on $\R^n_+$. Let \(B^+\coloneqq B\cap \R^n_+\) and let \(u\in W^{1,p}(B^+;\R^d)\) with \(|u|=1\) on \( \p \R^n_+ \cap B\) and
\begin{equation}\label{eq:partial_bdr}
\int_{B^+} \scal{|du|^{p-2}_g du}{d\Phi}_g\, d\vol_g=0 
\end{equation}
for all \(\Phi \in W^{1,p}(B^+;\R^d)\) such that \(\Phi\cdot u =0\) on \(\p \R_+^n \cap B\) and \(\Phi=0\) on \(\p B \cap \R_+^n\). There exists \(\e_*>0\) and \(C>0\) depending only on $n,d$ and $g$, such that the following holds. Consider any annulus \(A(x_0,R_1,R_2) = B_{(\R^n_+,g)}(x_0,R_2)\setminus B_{(\R^n_+,g)}(x_0,R_1) \Subset B\) where $x_0\in \R^n_+$, we denote by \(A(x_0,R_1,R_2)^+=A(x_0,R_1,R_2)\cap \R^n_+\). Assume that \(\int_{A(x_0,R_1,R_2)^+}|du|^n_g\, d\vol_g \leq \e_*^n\). We define \(v\in W^{1,p}(A(x_0,R_1,R_2)^+;\R^d)\) to be the solution of 
\begin{equation}\label{eq:free_m_harmonic}
\int_{A(x_0,R_1,R_2)^+} \scal{|dv|^{p-2}_g dv}{d\Phi}_g\, d\vol_g =0,
\end{equation}
for all \(\Phi\) in \(W^{1,p}(A(x_0,R_1,R_2)^+;\R^d)\) such that \(\Phi=0\) on \(\R^n_+\cap \p A (x_0,R_1,R_2)\), with the boundary condition \( v=u\) on $  \R^n_+ \cap \p A(x_0,R_1,R_2) $. Then we have
\begin{equation*}
\int_{A(x_0,R_1,R_2)^+}|du|^p_g\, d\vol_g \leq C\int_{ A(x_0,R_1,R_2)^+}|dv|^p_g\, d\vol_g.
\end{equation*}
\end{lemma}
We note that the map \(v\) appearing in this lemma satisfies homogeneous Neumann boundary conditions \(|dv|^{p-2}\p_{\nu_g}v=0\) on \(\p \R^n_+\cap A(x_0,R_1,R_2)\). We construct such a map by extending $u$ to $A(x_0,R_1,R_2)$ by a reflection and then solve $\lap_p v=0$ in $A(x_0,R_1,R_2)$ with $v=u$ on $\p A(x_0,R_1,R_2)$. By uniqueness, the map $v$ is also symmetric: $v(x_1,\ldots,x_n) = v(x_1,\ldots,x_{n-1},-x_n)$. Hence we obtain $\p_{\nu}v=0$ on \(\p_0 A(x_0,R_1,R_2)\).
\begin{proof}
	We consider the maps \(A_{ik}\) defined in Lemma \ref{lem:2nd_rewritting}. Thanks to Lemma \ref{lem:2nd_rewritting} we see that Equation \eqref{eq:partial_bdr} is equivalent to the following equations: for all \(1\leq i\leq m,\) and for all \(\varphi \in W^{1,p}(B^+;\R^d)\) with \( \varphi=0\) on \(\p B\cap \R^n_+\):
\begin{equation}\label{eq:partial_bdr2}
	\begin{aligned}
		& \int_{B^+} \scal{ |du|^{p-2}_g d u^i }{ d \vp}_g\, d\vol_g \\
		= & \int_{B^+} \scal{ |du|^{p-2}_g A_{ik} }{ d u^k \vp }_g + \scal{ |du|^{p-2}_g d u^i}{ d(|u|^2-1)\varphi }_g\, d\vol_g \\
		& + \int_{B^+} u^i\, \scal{ u\otimes d \varphi }{ |du|^{p-2} du }_g\, d\vol_g.
	\end{aligned}
\end{equation}
We set \(\varphi=u^i-v^i\) and we observe that \(\varphi \in W^{1,p}(A(x_0,R_1,R_2)^+;\R^d)\) with \(\varphi=0\) on \(\p A(x_0,R_1,R_2)\cap \R^n_+\). 
We take that \(\varphi\) as a test function in \eqref{eq:partial_bdr2} and \eqref{eq:free_m_harmonic} and we subtract the two identities to find
\begin{equation}\label{eq:decompo_u-v}
	\begin{aligned}
		& \int_{ A(x_0,R_1,R_2)^+} \scal{ |dv|^{p-2} dv - |du|^{p-2}du }{d(v-u)}_g\, d\vol_g\\ 
		=& \int_{ A(x_0,R_1,R_2)^+} \scal{|du|^{p-2} A_{ij}}{ (v^i-u^i)\, du^j}_g\, d\vol_g \\
		& + \int_{A(x_0,R_1,R_2)^+} \scal{|du|^{p-2} du ^i}{ (v^i - u^i)\, d(|u|^2-1)}_g\, d\vol_g \\
		& +\frac12 \int_{A(x_0,R_1,R_2)^+} \scal{|du|^{p-2}d(|u|^2 - 1) }{ u^i\, d(v^i - u^i)}_g\, d\vol_g  \\
		=&\ I+II+III.
	\end{aligned}
\end{equation}
\underline{The term $I$:} As in the proof of \Cref{prop:uniform_epsilon_reg}, cf.\ \cite[p.\ 1311]{MazowieckaRodiacSchikorra}, using that $A(x_0,R_1,R_2)^+$ is simply connected, there exists \(\zeta_{ij}\in W^{1,\frac{p}{p-1}}(A(x_0,R_1,R_2)^+;\Lambda^2\R^n)\) such that 
\begin{align*}
	\sqrt{\det g}\, |du|^{p-2}_g A_{ik}\, g^{jk}=d^{*_{\geucl}} \zeta_{ij}\quad  \text{ in } A(x_0,R_1,R_2)^+.
\end{align*}
We extend $v-u$ by $0$ on $\R^n_+\setminus A(x_0,R_1,R_2) $, and then we extend this new function to all \(\R^n\) by even reflection.
 Namely, with \(\sigma:\R^n \rightarrow \R^n\) defined by \(\sigma(x)=(x_1,\dots,-x_n)\) we set
\begin{equation*}
(v-u)(x)=\begin{cases}
(v-u)(x) & \text{ for } x\in \R^n_+ \\
(v-u)(\sigma(x)) & \text{ for } x \in \R^n_.
\end{cases}
\end{equation*}
 Note that we obtain that $v-u\in W^{1,p}(\R^n;\R^d)$ since $A(x_0,R_1,R_2)^+$ is a Lipschitz set. We also extend \(\zeta_{ij}\) to \(\R^n\) by keeping control of the $W^{1,p}$-norm, first by reflection on $A(x_0,R_1,R_2)$, then by the extension on punctured domain on $B(x_0,R_2)$ (see for instance \cite[Theorem 2.2]{HOPKER2014}) and then by Whitney extension to $\R^n$.  We have, without using a different notation for the extensions:
\begin{align*}
|I|&=\left|\int_{A(x_0,R_1,R_2)^+} \scal{ d^{*_{\geucl}} \zeta_{ij} }{  d u^j (v^i-u^i)}_{\geucl} dx  \right| \\
&= \frac{1}{2}\left| \int_{\R^n} \scal{ d^{*_{\geucl}} \zeta_{ij} }{  d u^j (v^i-u^i)}_{\geucl} dx \right| \\
&= \frac{1}{2} \left| \int_{\R^n} \scal{ d^{*_{\geucl}} \zeta_{ij} }{ u^j\, d(v^i-u^i)}_{\geucl} dx \right|  \\
& \leq C \| d^{*_{\geucl}}  \zeta_{ik} \|_{L^{p'}( A(x_0,R_1,R_2)^+,\geucl)} \|d(u-v)\|_{L^p( A(x_0,R_1,R_2)^+,\geucl)}[u]_{BMO}.
\end{align*}
We obtain
\begin{equation}\label{eq:termI}
	\begin{aligned}
		|I| \leq  C \|du\|_{L^p(A(x_0,R_1,R_2)^+,g)}^{p-1}  \|d(u-v)\|_{L^p(A(x_0,R_1,R_2),g)^+} \|du\|_{A(x_0,R_1,R_2)^+,g )}. 
	\end{aligned}
\end{equation}
\smallskip
\underline{The term $II$:} Again, as in the proof of \Cref{prop:uniform_epsilon_reg}, cf.\ \cite[p.\ 1312]{MazowieckaRodiacSchikorra}, using that $A(x_0,R_1,R_2)^+$ is simply connected, there exists \(\zeta_i\in W^{1,p'} (A(x_0,R_1,R_2)^+;\Lambda^2\R^n)\) such that 
\begin{equation}\label{def:zetai}
 \sqrt{\det g}\, |d u|^{p-2}_g (\p_{\lambda} u^i) (g^{\lambda\mu})\dx^{\mu} = d^{*_{\geucl}} \zeta_i \quad  \text{ in } A(x_0,R_1,R_2)^+,
 \end{equation}
 and \(\|\zeta_i\|_{W^{1,p'}(A(x_0,R_1,R_2)^+)}\leq C(p,g,n) \| du\|_{L^p(A(x_0,R_1,R_2)^+,g)}^{p-1}\). We extend \(\zeta_i\) to all of \(\R^n\) so that
 \begin{align*}
 	\|\zeta_i\|_{W^{1,p'}(\R^n,\geucl)}\leq C \| d u\|_{L^p(A(x_0,R_1,R_2)^+,g)}^{p-1}.
 \end{align*}
 We can write, by extending \(|u|^2-1\) by zero in all \(\R^n\):
 \begin{align*}
 |II| &\leq \left| \int_{A(x_0,R_1,R_2)^+} \scal{d^{*_{\geucl}} \zeta_i }{d(|u|^2-1)\, (v^i - u^i)}_{\geucl} \dx  \right| \\
 & \leq \left| \int_{\R^n} \scal{d^{*_{\geucl}} \zeta_i }{d(|u|^2-1)\, (v^i - u^i)}_{\geucl} dx  \right| \\
 & \leq \left| \int_{\R^n} \scal{d^{*_{\geucl}} \zeta_i }{(|u|^2-1)\, d(v^i - u^i)}_{\geucl} dx  \right| \\
 & \leq C \|d^{*_{\geucl}} \zeta \|_{L^{p'}( A(x_0,R_1,R_2)^+,\geucl)} \|d(u-v)\|_{L^p(A(x_0,R_1,R_2)^+,\geucl)}  [|u|^2]_{BMO}.
 \end{align*}
 We obtain
 \begin{equation}\label{eq:termII}
 	 \begin{aligned}
 		|II| \leq  C \|du\|_{L^p( A(x_0,R_1,R_2)^+,g)}^{p-1} \|d(v-u)\|_{L^p( A(x_0,R_1,R_2)^+,g)} \|du\|_{L^n( A(x_0,R_1,R_2)^+,g)}.
 	\end{aligned}
 \end{equation}
 We observe that the dependence in \(p\) of the constant appearing in the previous inequality comes from: the Poincar\'e lemma, the extension theorem for Sobolev maps on punctured domain and the Coifman--Lions--Meyer--Semmes Theorem. It can be checked that these constants are continuous in \(p\) in our case.
 
 \underline{The term $III$:}  
We write the term $III$ as follows:
\begin{align*}
	& \frac12 \int_{A(x_0,R_1,R_2)^+} \scal{|du|^{p-2}d(|u|^2 - 1) }{ u^i\, d(v^i - u^i)}_g\, d\vol_g \\
	= &  \int_{\R^n_+\cap A(x_0,R_1,R_2)} \scal{|du|^{p-2}du^j }{ (u^j)(u^i)\, d(v^i - u^i)}_g\, d\vol_g .
\end{align*} 
We extend $v-u$ by $0$ on $\R^n_+$ and then we use a reflection procedure to obtain a map in \(\R^n\). Using \eqref{def:zetai}, we obtain
 \begin{align*}
 	& \frac12 \int_{\R^n_+\cap A(x_0,R_1,R_2)} \scal{|du|^{p-2}d(|u|^2 - 1) }{ u^i\, d(v^i - u^i)}_g\, d\vol_g \\
 	= &  \int_{\R^n} \scal{d^{*_{\geucl}} \zeta_j }{ (u^j)(u^i)\, d(v^i - u^i)}_{\geucl}\, d\vol_{\geucl}.
 \end{align*} 
 Using the $\Hr^1-\BMO$ duality, we obtain 
 \begin{equation}\label{eq:termIII}
 	 \begin{aligned}
 		| III| \leq   C \|du\|_{L^p( A(x_0,R_1,R_2)^+,g)}^{p-1} \|d(v-u)\|_{L^p(A(x_0,R_1,R_2)^+,g)} \|du\|_{L^n( A(x_0,R_1,R_2)^+,g)}.
 	\end{aligned}
 \end{equation} 
\medskip
Gathering the estimates \eqref{eq:termI}-\eqref{eq:termII}-\eqref{eq:termIII} and coming back to \eqref{eq:decompo_u-v}, we find that 
\begin{align*}
	\|d(u-v)\|_{L^p( A(x_0,R_1,R_2)^+,g)}^p \leq &\ C \|du\|_{L^p(A(x_0,R_1,R_2)^+,g)}^{p-1} \\
	& \times \|d(u-v)\|_{L^p( A(x_0,R_1,R_2)^+,g)} \|du\|_{L^n(A(x_0,R_1,R_2)^+,g )}. 
\end{align*} 
 This entails 
 \begin{align*}
 \|d(u-v)\|_{L^p(A(x_0,R_1,R_2)^+,g)}^{p-1} \leq &  \ C \|du\|_{L^p(A(x_0,R_1,R_2)^+,g)}^{p-1} \|du\|_{L^n( A(x_0,R_1,R_2)^+,g )}. 
 \end{align*}
 By assumption, the last term of the right-hand side is bounded by $\e_*$. Thus, we obtain
 \begin{align*}
 \|du\|_{L^p( A(x_0,R_1,R_2)^+,g)}^{p-1}&\leq C \Big(\|d(u-v)\|_{L^p(A(x_0,R_1,R_2)^+,g)}^{p-1}+\|dv\|_{L^p(A(x_0,R_1,R_2)^+,g)}^{p-1} \Big) \\
 & \leq C \eps_* \|du\|_{L^p( A(x_0,R_1,R_2)^+,g)}^{p-1} + C\|dv\|_{L^p( A(x_0,R_1,R_2)^+,g)}^{p-1}.
 \end{align*}
 Thus if \(\e_*\) is small enough we arrive at the conclusion, using that $p\in[n,n+1]$ and the fact that $C=C(n,p,g,d)$ depends continuously on $p$.
\end{proof}

In the neck regions we will see that $p$-harmonic maps with a free boundary are almost constant on the boundaries of the neck. By \Cref{lem:comparison_2}, we only need to estimate the energy of unconstrained $p$-harmonic maps with constant boundary conditions.

\begin{lemma}\label{lem:comparison_cst}
There exists a constant $C>0$ depending only on $g$,$n$ and $p$, which is bounded as long as $g,n,p$ are bounded, and such that the following holds.

Let \(a,b\) be two constant vectors in \(\R^d\) and $p\geq 2$. Fix a point $x_0\in \p\Sigma^n$. Let $0<R_1<R_2$ be two positive radii smaller than the injectivity radius of $\Sigma^n$. Define the annulus $A(x_0,R_1,R_2)\coloneqq  B_{(\Sigma^n,g)}(x_0,R_2)\setminus B_{(\Sigma^n,g)}(x_0,R_1)$. Let \(v \in W^{1,p}(\Sigma^n \cap A(x_0,R_1,R_2);\R^d)\) satisfy
\begin{equation*}
\left\{
\begin{array}{rcll}
d^{*_g} (|dv|^{p-2}_g dv)&=&0 &\text{ in } A(x_0, R_1,R_2) \\
v&=&a & \text{ on } \Int(\Sigma^n)\cap \p B_{(\Sigma^n,g)}(0,R_1) \\
v&=& b &\text{ on } \Int(\Sigma^n)\cap \p B_{(\Sigma^n,g)}(0,R_2) \\
|dv|^{p-2}_g \p_\nu v&=& 0 & \text{ on } \p \Sigma^n \cap A(x_0,R_1,R_2).
\end{array}
\right.
\end{equation*} 
If $p\neq n$, then we have
\begin{equation*}
\int_{A(x_0,R_1,R_2)}|dv|^p_g\, d\vol_g\leq C\, \frac{|b-a|^p}{\log(R_2/R_1)^p}\,  \frac{R_2^{n-p}-R_1^{n-p}}{n-p}.
\end{equation*}
If $p=n$, then we have
\begin{align*}
	\int_{A(x_0,R_1,R_2)}|dv|^n_g\, d\vol_g\leq C\, \frac{|b-a|^n}{\log(R_2/R_1)^{n-1}}.
\end{align*}
\end{lemma}

\begin{remark}\label{rk:limit}
	We observe that
	\begin{align*}
		\frac{R_2^{n-p}-R_1^{n-p}}{n-p} = \int_{R_1}^{R_2} \frac{dr}{r^{p-n+1}} \xrightarrow[p\to n]{}{\log(R_2/R_1)}.
	\end{align*}
\end{remark}

\begin{proof}
Since $R_2$ is smaller than the injectivity radius of $\Sigma^n$, there exists a diffeomorphism $\psi: B(0,R_2)\cap \{x_n\geq 0\}\to B_{(\Sigma^n,g)}(x_0,R_2)$ such that 
\begin{itemize}
	\item $\psi(0) = x_0$,
	\item $\psi \Big(B(0,R_2)\cap \{x_n= 0\}\Big)\subset \p\Sigma^n$,
	\item $\psi:B(0,R_2)\setminus B(0,R_1)\to A(x_0,R_1,R_2)$ mapping $\p\Big( B(0,R_2)\setminus B(0,R_1)\Big)$ on $\p A(x_0,R_1,R_2)$.
\end{itemize}
 Namely, we take $\psi = \exp^{(\Sigma,g)}_{x_0}$. Hence, we can consider that the domain is contained in $\R^n_+$. Since the map $v$ is obtained by minimizing the $L^n$ norm of the gradient, we can compare \(v\) to the following map:
 \begin{align*}
 	f(x) = a + \frac{\log(|x|/R_1)}{\log(R_2/R_1)}(b-a).
 \end{align*}
 The Euclidean norm of its differential is given by
 \begin{align*}
 	|df|_{\geucl} = \frac{|b-a|}{|x|\, \log(R_2/R_1)}.
 \end{align*}
 Thus, we obtain
 \begin{align*}
 	\int_{A(x_0,R_1,R_2)} |dv|^p_g\, d\vol_g \leq \int_{A(x_0,R_1,R_2)} |df|^p_g\, d\vol_g \leq \frac{C\, |b-a|^p}{\log(R_2/R_1)^p} \int_{B(0,R_2)\setminus B(0,R_1)} \frac{\dx}{|x|^p}.
 \end{align*}
 The last integral of the right-hand side satisfies
 \begin{align*}
 	\int_{B(0,R_2)\setminus B(0,R_1)} \frac{dx}{|x|^p} = |\s^{n-1}|\int_{R_1}^{R_2} r^{n-1-p}\, dr = \left\{
 	\begin{aligned}
 		& |\s^{n-1}| \frac{R_2^{n-p}-R_1^{n-p}}{n-p} &  \text{ if }n\neq p,\\
 		& |\s^{n-1}|\log(R_2/R_1) & \text{ if }n=p.\\
 	\end{aligned}
 	\right.
 \end{align*}
\end{proof}

\section{Bubbling decomposition}\label{sec:Bubbling_decompo}

We start now the proof of \Cref{th:main_bubbling}. Up to a subsequence, we can assume that \( u_k\rightharpoonup u\) weakly in \(W^{1,n}\) and \( u_k \rightarrow u\) strongly in \(L^n\) as \(k\) goes to \(+\infty\).

\subsection{Step 1: Local strong convergence away from a finite set}

We define the concentration set of \( (u_k)_k\) by 
\begin{equation*}
S\coloneqq  \bigcap_{r >0} \left\{ x \in \Sigma^n \colon \limsup_{k \to +\infty} \int_{B_{(\Sigma^n,g)}(x,r)} |du_k|_g^n \, d\vol_g \ge \e_0^n \right\},
\end{equation*}
where \(\e_0\) is the number appearing in Proposition \ref{prop:uniform_epsilon_reg}. Since the \(n\)-energy of \( (u_k)_k\) is uniformly bounded, we can see that \(S\) is a finite subset of \(\Sigma^n\), cf.\ e.g.\ \cite[Proposition 4.3]{sacks1981}. Furthermore, for any \(x_0\notin S\), there exists \(r_0>0\) such that 
\begin{equation*}
\limsup_{k \to +\infty } \int_{B_{(\Sigma^n,g)}(x_0,r_0)} |du_k|_g^n \,  d\vol_g < \e_0^n.
\end{equation*}
Thus the \(\e\)-regularity result of Proposition \ref{pr:Convergence} implies that there exists a subsequence that satisfies
\begin{equation}\nonumber
u_ k\xrightarrow[k \to +\infty]{} u \text{ in }  \C^1_{\text{loc}}(B_{(\Sigma^n,g)}(x_0,r_0);\R^d).
\end{equation}
We deduce that 
\begin{equation}\label{eq:strong_conv_outside_points}
u_k \xrightarrow[k \to +\infty]{} u \text{ in } \C^1_{\text{loc}}(\Sigma^n \setminus S;\R^d).
\end{equation}
We note also that, from the interior regularity of the unconstrained \(n\)-harmonic maps, \cite{Uhlenbeck_1977,Tolksdorf_1984},  \(S \subset \p \Sigma^n\). From \eqref{eq:strong_conv_outside_points}, we obtain that \( u\) satisfies the equation on \(\Sigma^n\setminus S\). Since \(u\) has finite \(n\)-energy from the assumption 3) of Theorem \ref{th:main_bubbling}, thanks to the removability of singularity result of Proposition \ref{prop:sing_remov_result} we actually infer that \(u\) satisfies the equation on all \(\Sigma^n\) and by the results of \cite{MazowieckaRodiacSchikorra} we have that \(u \in \C^{1,\beta}(\Sigma^n;\R^d)\) for some \(\beta>0\). In the following we assume that \(S=\{x_1,\dots,x_{\ell_1}\}\).

\subsection{Step 2: Extraction of the first-generation bubbles}

We focus on the concentration point $x_1$. Let \( 0< \delta_1<\frac{1}{2}\min \{|x_1-x_j|\colon 2\leq j\leq \ell_1 \}\), up to take a smaller \(\delta_1\) we can assume that there exists \(r_1>0\) and a \(\C^\infty\) diffeomorphism \(\psi_1: \R^n_+\cap B(0,r_1) \rightarrow B_{(\Sigma^n,g)}(x_1,\delta_1)\) with \(\psi_1 (\p \R^n_+\cap B(0,r_1))=\p \Sigma^n \cap B_{(\Sigma^n,g)}(x_1,\delta_1)\), \(\psi_1(0)=x_1\) and \(d\psi_1(0)=\text{id}\). We can take \(\psi_1=\exp_{x_1}^{(\Sigma^n,g)}\).
We can write that
%
\begin{equation*}
\int_{B_{(\Sigma^n,g)}(x_1,\delta_1)}|du_k|_g^{p_k}\, d\vol_g =\int_{B(0,r_1)^+}|d(u_k\circ \psi_1)|_{\psi_1^*g}^{p_k}\, \sqrt{\det \psi_1^* g}\, dx.
\end{equation*}
With a slight abuse of notation, we will work in the system of local coordinates given by \(\psi_1\) and we write \(u_k=u_k\circ \psi_1\), \(g=\psi_1^*g\) and \( d\vol_g= \sqrt{\det \psi_1^* g}\, dx\) given by \(\psi_1\). We can proceed analogously for the other points \(x_i\), \(i\in \{2,\dots,\ell_1\}\) and find \(r_i, \psi_i\) for \(i\in \{2,\dots,\ell_1\}\). We define
\begin{equation*}
Q_k^{1,1}(t)=\sup_{x\in B(0,r_1)^+} \int_{B(x,t)\cap B(0,r_1)^+} |du_k|_g^{p_k} \, d \vol_g.
\end{equation*}
We have that \(Q_k^{1,1}\) is continuous with \( Q_k^{1,1} (0)=0\) and for \(k\) large enough,
\begin{align*}
	Q_k^{1,1}(r_1)\geq \vol_g(B(0,r_1)^+)^{1-\frac{p_k}{n}} \left( \int_{B(0,r_1)^+}|du_k|_g^n d\vol_g\right)^{\frac{p_k}{n}} > \frac{ \e_0^{p_k} }{ \vol_g(B(0,r_1)^+)^{\frac{p_k}{n}-1} }.
\end{align*}
Thus, there exist \( \left(a_k^{1,1}\right)_{k\in\N} \subset \overline{B(0,r_1)}^+\) and \( \left(\lambda_k^{1,1}\right)_{k\in\N} \subset (0,r_1)\) such that
\begin{align*}
	Q_k^{1,1}(\lambda_k^{1,1}) =  \int_{B\left(a_k^{1,1},\lambda_k^{1,1}\right)\cap B(0,r_1)^+}|du_k|_g^{p_k} \, d \vol_g =\frac{\eps_{\star}^{p_k}}{2},
\end{align*}
where the constant $\e_{\star}$ is defined using the constants $\e_0$ from \Cref{pr:Convergence}, $\e_*$ from \Cref{lem:comparison_2} and $\e_b$ from \Cref{lem=quantum_energy} by the formula:
\begin{align}\label{eq:def_eps_star}
	\e_{\star} = \min\left( \frac{ \e_0 }{ \vol_g(B(0,r_1)^+)^{\frac{1}{n}-\frac{1}{p_k}} },\e_*,\e_b \right).
\end{align}

\begin{claim}\label{claim_2}
	The following convergence holds $\displaystyle{\lim_{k\to+\infty}} \left( |a_k^{1,1}| + \lambda_k^{1,1} \right) = 0$.
\end{claim}
\begin{proof}
 We can always assume that \(a_k^{1,1} \rightarrow a \in \overline{B(0,r_1)^+}\) and \(\lambda_k^{1,1} \rightarrow \lambda\geq 0\). If \(\lambda \neq 0\) then we have that, for \(k\) large enough,
\begin{equation*}
\frac{\e_{\star}^{p_k}}{2}=Q_k^{1,1}(\lambda_k^{1,1}) \geq \int_{B\left(0,\lambda_k^{1,1}\right)^+}|d u_k|_g^{p_k}\, d \vol_g \geq \int_{B\left(0,\frac{\lambda}{2}\right)^+}|du_k|_g^{p_k} \, d\vol_g> \frac{ \e_0^{p_k} }{ \vol_g(B(0,r_1)^+)^{\frac{p_k}{n}-1} }.
\end{equation*}
This is a contradiction and thus \(\lambda_k^{1,1}\to 0\). Now if \(a\neq 0 \), we can choose \(\rho>0\) small enough so that
\[\int_{B(a,\rho)^+}|du|^n_g\, d\vol_g \leq \frac{\e_{\star}^n}{4}\] and \(0\notin B(a,\rho)\). We can choose \(k\) large enough so that \[2(|a-a_k^{1,1}|+\lambda_k^{1,1})\leq \rho.\]
Then, for large \(k\) we have
\begin{align*}
\frac{\e_{\star}^{p_k}}{2} & =\int_{B\left(a_k^{1,1},\lambda_k^{1,1}\right)\cap B(0,r_1)^+} |du_k|_g^{p_k}d \vol_g  \leq \int_{B\left( a,2(|a-a_k^{1,1}|+\lambda_k^{1,1}) \right)\cap B(0,r_1)^+} |du_k|^{p_k}_g\, d \vol_g \\
& \leq \int_{B(a,\rho)\cap B(0,r_1)^+}|d u_k|^{p_k}_g\, d\vol_g \xrightarrow[k \to +\infty]{} \int_{B(a,\rho )\cap B(0,r_1)^+}|du|^n_g\, d\vol_g <\frac{\e_{\star}^n}{4}.
\end{align*}
The last convergence is true by the strong convergence in \(\C^1_{\text{loc}}(\Sigma^n\setminus \{x_1,\dots,x_{\ell_1}\};\R^d)\) of \( (u_k)_{k\in\N}\) to \(u\), cf.\ \eqref{eq:strong_conv_outside_points}. We obtain a contradiction and this proves that \(a_k^{1,1}\rightarrow 0\).
\end{proof}
We now define the domain $\Omega_k^{1,1}\subset \R^n$, the map $v_k^{1,1}\in W^{1,p_k}(\Omega_k^{1,1};\R^d)$ and the metric \( g_k^{1,1}\in \C^\infty(\Omega_k^{1,1};\R^d \times \R^d)\) as follows:
\begin{equation*}
\forall x\in \Omega_k^{1,1} \coloneqq  \frac{B(0,r_1)^+-a_k^{1,1}}{\lambda_k^{1,1}},\quad  v_k^{1,1} (x)\coloneqq  u_k(a_k^{1,1}+\lambda_k^{1,1} x), \quad g_k^{1,1}(x) \coloneqq  g(a_k^{1,1} + \lambda_k^{1,1}x).
\end{equation*}
The map $v_k$ satisfies the following equation, valid for all \(\Phi \in W^{1,p_k}(\Omega_k^{1,1};\R^d)\) such that \( \Phi\cdot v_k^{1,1}=0\) on \( \frac{\p_0 B(0,r_1)^+ - a_k^{1,1}}{\lambda_k^{1,1}}\), where \(\p_0\) is defined in \eqref{eq:def_partial_D0}, and $\Phi = 0$ on $\frac{(\p B(0,r_1))\cap \R^n_+ - a^{1,1}_k}{\lambda_k^{1,1}}$:
\begin{equation}\label{eq:syst_v11}
\int_{\Omega_k^{1,1} } \scal{|dv_k^{1,1}|^{p_k-2}_g dv_k^{1,1}}{d\Phi}_{g_k^{1,1}}\, d\vol_{g_k^{1,1}} =0.
\end{equation}
We have 
\begin{align*}
	\int_{B(0,1)\cap \Omega_k^{1,1} } |dv_k^{1,1}|^{p_k}_{g_k^{1,1}}\, d\vol_{g_k^{1,1}} & = (\lambda_k^{1,1})^{p_k-n}\int_{B \left(a_k^{1,1},\lambda_k^{1,1} \right)\cap B(0,r_1)^+ }|du_k|^{p_k}_g\, d\vol_g \nonumber \\
	& = (\lambda_k^{1,1})^{p_k-n}\frac{\e_{\star}^{p_k}}{2}. 
\end{align*}
Furthermore, since we know that \(\lambda_k^{1,1} \rightarrow 0\), and \(p_k\geq n\) we  find that, for all \(x\in \R^n\), for \(k\) large enough
\begin{align*}
 \int_{B(x,1)\cap \Omega_k^{1,1} }|dv_k^{1,1}|^{p_k}_{g_k^{1,1}}\, d\vol_{g_k^{1,1}} 
 = & (\lambda_k^{1,1})^{p_k-n}\int_{B \left(a_k^{1,1}+\lambda_k^{1,1} x,\lambda_k^{1,1} \right)\cap B(0,r_1)^+}|du_k|^{p_k}_g\, d\vol_g \\ &\leq (\lambda_k^{1,1})^{p_k-n}Q_k^{1,1}(\lambda_k^{1,1})\leq \frac{\e_{\star}^{p_k}}{2}.
\end{align*}
In particular we have that for all \(x\in \R^n\) and for \(k\) large enough
\begin{equation}\label{eq:nonconcentration_v11}
 \int_{B(x,1)\cap \Omega_k^{1,1} }|dv_k^{1,1}|^{n}_{g_k^{1,1}}\, d\vol_{g_k^{1,1}} <\e_0^n.
\end{equation}

Now we distinguish two cases.
\begin{itemize}
\item[\underline{Case 1)}] Assume that 
\begin{align}\label{eq:full_space}
	\frac{\dist(a_k^{1,1}, \p \R^n_+) }{\lambda_k^{1,1} } \xrightarrow[k\to +\infty]{} +\infty.
\end{align}
\end{itemize}
 In that case \(\Omega_k^{1,1}\to \R^n\) in the sense that for any \(R>0\) there exists \(k\) large enough so that \(B(0,R)\subset \Omega_{k}^{1,1}\). To see that, it suffices to check that for any \(x\in B(0,R)\) and for \(k\) large enough \(a_k^{1,1}+\lambda_k^{1,1} x\in B(0,r_1)^+\). Combining \eqref{eq:syst_v11}-\eqref{eq:nonconcentration_v11}, \Cref{pr:Convergence} and the convergence $g_k^{1,1} \to g(x_1)$ as $k\to \infty$ in $\C^{\infty}_{\loc}$, we obtain that there exists a map \( \omega^{1,1}\in \C^1(\R^n;\R^d)\) such that 
\begin{equation}\label{eq:strong_conv-case_Rn}
v_k^{1,1} \xrightarrow[k \to +\infty]{} \omega^{1,1} \text{ in } \C^1_{\text{loc}}(\R^n;\R^d).
\end{equation}

\begin{claim}\label{claim:finite_energy}
The map \(\omega^{1,1}\) has finite \(n\)-energy, is bounded and is \(n\)-harmonic in \(\R^n\). Hence it is a constant and \eqref{eq:full_space} cannot happen.
\end{claim}
\begin{proof}
 Indeed, for any \(R>0\) we observe that 
\begin{align}
\int_{B(0,R)}|d \omega^{1,1}|^n_{g(x_1)}\, d\vol_{g(x_1)} &= \lim_{k \to +\infty} \int_{\Omega_k^{1,1}\cap B(0,R)}|d v_k^{1,1}|^{p_k}_{g_k^{1,1}}\, d\vol_{g_k^{1,1}} \nonumber \\
&= \lim_{k \to +\infty} (\lambda_k^{1,1})^{p_k-n}\int_{B(0,r_1)^+ \cap B\left(a_k^{1,1},\lambda_k^{1,1} R\right)}|d u_k|_g^{p_k}\, d\vol_g \leq M, \label{eq:finite_energy}
\end{align}
with a constant \(M\) independent of \(R\), since it is the uniform energy bound on the \(p_k\)-energy of the sequence \((u_k)_{k\in\N}\) in \Cref{th:main_bubbling}. Since this is valid for any \(R>0\) we obtain the finiteness of the \(n\)-energy of \(\omega^{1,1}\).

 By using \eqref{eq:nonconcentration_v11}-\eqref{eq:strong_conv-case_Rn} we also find that \(\int_{B(0,1)}|d \omega^{1,1}|^n_{g(x_1)}\, d\vol_{g(x_1)}=\frac{\e_{\star}^n}{2}\). However \eqref{eq:strong_conv-case_Rn} also shows that \(\Delta_{n,g(x_1)} \omega^{1,1}=0\) in \(\R^n\). This implies that \(\omega^{1,1}\) is constant. Indeed, we apply \cite[Corollary 2.2]{martino2024}: for any $R>0$, we have
\begin{align*}
	\|d\omega^{1,1}\|_{L^{\infty}(B(0,R))} \leq C\left(\frac{1}{R^n} \int_{B(0,2R)} |d\omega^{1,1}|^n\right)^{\frac{1}{n}} .
\end{align*}
The right-hand side converges to $0$ as $R\to +\infty$. This shows that \eqref{eq:full_space} cannot happen and we are only in the following  second case.
\end{proof}

\begin{itemize}
\item[\underline{Case 2)}] Up to a subsequence,
\begin{align*}
	\frac{ \dist( a_k^{1,1},\p \R^n_+)}{\lambda_k^{1,1} } \xrightarrow[k\to+\infty]{} A\in [0,+\infty) .
\end{align*}
\end{itemize}
In that case \(\Omega_k^{1,1}\) converges to a half-space. Up to a rotation and a translation we can assume that this is \(\R^n_+\). Again, thanks to \eqref{eq:syst_v11}-\eqref{eq:nonconcentration_v11}, \Cref{pr:Convergence} and the convergence $g_k \to g(x_1)$ as $k\to \infty$ in $\C^{\infty}_{\loc}$, we obtain that there exists \(\omega^{1,1} \in \C^1_{\text{loc}}(\overline{\R^n_+};\R^d)\) such that 
\begin{equation}\label{eq:strong_conv_vn}
v_k^{1,1} \xrightarrow[k \to +\infty]{} \omega^{1,1} \text{ in } \C^1_{\text{loc}}(\overline{\R^n_+};\R^d).
\end{equation} 

\begin{claim}
	The map $\omega^{1,1}\in \C^1_{\text{loc}}(\overline{\R_+^n};\R^d)$ is not constant.
\end{claim}
\begin{proof}
As in Claim \ref{claim:finite_energy}, we can obtain that the map \(\omega^{1,1}\) has finite \(n\)-energy. By using \eqref{eq:strong_conv_vn} we also find that \(\int_{B(0,1)^+}|d \omega^{1,1}|^n_{g(x_1)}\, d\vol_{g(x_1)}=\frac{\e_{\star}^n}{2}>0\) and \(\omega_1\) is a non-trivial bubble in the sense of Definition \ref{def:bubble}. 
\end{proof}
We can repeat this argument near \(x_2,\ldots,x_{\ell_1}\) to obtain bubbles \(\omega^{i,1}\in \C^1(\overline{\R^n_+};\R^d)\) and sequences \( \psi_i(a_k^{i,1})\rightarrow x_i\) and \(\lambda_k^{i,1} \rightarrow 0\) as \(k \to +\infty\) such that 
\begin{equation*}
u_k\circ \psi_i \left(a_k^{i,1}+\lambda_k^{i,1}\cdot\right) \rightarrow \omega^{i,1} \text{ in } \C^1_{\text{loc}}(\overline{\R^n_+};\R^d).
\end{equation*}

%

\subsection{Step 3: No-neck energy property for the first-generation bubble at a given point $x_i\in S$.}

Recall from Lemma \ref{lem:prop_bubbles} that bubbles have a limit at infinity. We can also extend to \(\R^n\) every \(\omega^{i,1}\) for \(1\leq i \leq \ell_1\) by setting \(\omega^{i,1}(x)=\omega^{i,1}(\sigma(x))\) for \(x\in \R^n_-\) with \(\sigma(x)=(x_1,\dots,x_{n-1},-x_n)\). We define
\begin{align*}
	w_k^{i,1}(x) \coloneqq  u_k(x) - \left[\omega^{i,1}\left(\frac{x-a_k^{i,1}}{\lambda_k^{i,1}}\right) - \omega^{i,1}(\infty)\right] \quad \text{ for all }  x \in B(0,r_i)^+.
\end{align*}
  Note that \(w_k^{i,1}\) converges weakly towards \(u\) in \(W^{1,n}(B(0,r_i)^+;\R^d)\). The aim of this step is to prove an energy identity.
\begin{claim}\label{claim:energy_identity_1}
	For all \(1\leq i\leq \ell_1\) let \( \lambda_*^{i,1}\coloneqq  {\displaystyle \liminf_{k \to +\infty}}\  (\lambda_k^{i,1})^{n-p_k}\). Then $\lambda_*^{i,1}$ is finite with the following estimate
\begin{equation}\label{eq:bound_lambda_0}
	\lambda_*^{i,1}\in \left[1,\frac{M}{\e_b^n} \right],
	\end{equation}
	where \(M\) appears in Theorem \ref{th:main_bubbling} and \(\e_b\) appears in \Cref{lem=quantum_energy}. Furthermore the following energy identity holds:
	\begin{align}\label{eq:energy_identity_1}
		\int_{B(0,r_i)^+} |dw_k^{i,1}|^{p_k}_g\, d\vol_g = \int_{B(0,r_i)^+} |du_k|^{p_k}_g\, d\vol_g -\lambda_*^{i,1} \int_{\R^n_+} |d\omega^{i,1}|^n_{g(x_i)}\, d\vol_{g(x_i)} +o_k(1).
	\end{align}
\end{claim}
\begin{proof}
We first prove the estimate \eqref{eq:bound_lambda_0} on the concentration speeds. It is clear that for \(k\) large enough \((\lambda_k^{i,1})^{n-p_k}\geq 1\). We define
\begin{equation*}
	\forall x\in \Omega_k^{i,1}\coloneqq  \frac{B(0,r_i)^+-a_k^{i,1}}{\lambda_k^{i,1}},\quad \quad  g_k^{i,1}(x)\coloneqq  g\left(a_k^{i,1} + \lambda_k^{i,1}x\right).
\end{equation*}

By using a change of variables and the convergence in \eqref{eq:strong_conv_vn}, for all \(R>0\), we can write that 
	\begin{align*}
		& \liminf_{k\to +\infty} \left(\lambda_k^{i,1}\right)^{n-p_k} \\
		= & \liminf_{k\to +\infty} \left(\int_{B(0,R)\cap \Omega_k^{i,1} } \left| d\left[ u_k\left(a_k^{i,1} + \lambda_k^{i,1}\cdot \right) \right] \right|_{g_k^{i,1}}^{p_k}\, d\vol_{g_k^{i,1}} \right)^{-1} \int_{B(a_k^{i,1},R\lambda_k^{i,1})\cap B(0,r_i)^+} |du_k|_g^{p_k} d \vol_g \\
		\leq &  \left(\int_{B(0,R)^+ } \left| d\omega^{i,1} \right|_{g(x_i)}^n \, d \vol_{g(x_i)} \right)^{-1} M.
	\end{align*}
This is valid for any $R$, and letting \(R\to +\infty\) we obtain the result with the help of Lemma \ref{lem=quantum_energy} and \Cref{lem:prop_bubbles}. We now prove the energy identity.	Let \(R_i>1\) to be determined later (it will be sent to $+\infty$). We decompose the energy of $w_k^{i,1}$ as follows,
	\begin{align*}
		& \int_{B(0,r_i)^+}|dw_k^{i,1}|^{p_k}_g \, d\vol_g  \\
		= & \int_{B\left(a_k^{i,1},R_i\lambda_k^{i,1}\right)\cap B(0,r_i)^+ } |dw_k^{i,1}|^{p_k}_g\, d\vol_g + \int_{ B(0,r_i)^+\setminus B\left(a_k^{i,1},R_i\lambda_k^{i,1}\right) } |dw_k^{i,1}|^{p_k}_g\, d\vol_g  \\
		=& \int_{B(0,R_i)\cap \Omega_k^{i,1}} (\lambda_k^{i,1})^{n-p_k}\left| d [u_k\left(a_k^{i,1}+\lambda_k^{i,1}\cdot \right)-\omega^{i,1}] \right|^{p_k}_{g_k^{i,1}}\, d\vol_{g_k^{i,1}}   \\
		& \quad \quad \quad + \int_{ B(0,r_i)^+\setminus B\left(a_k^{i,1},R_i\lambda_k^{i,1}\right) } |dw_k^{i,1}|^{p_k}_g\, d\vol_g  \\
		=&\ I+II.
	\end{align*}
	 We first consider the term \(I\).  First we use the strong convergence of the metric $g_k^{i,1}\to g(x_i)$ in $\C^{\infty}_{\loc}$ as $k\to \infty$ and the bound on \(\lambda_k^{i,1}\) obtained previously to deduce
	\begin{align*}
		I= (\lambda_k^{i,1})^{n-p_k}\int_{B(0,R_i)\cap \Omega_k^{i,1}} \left| d [u_k\left(a_k^{i,1}+\lambda_k^{i,1}\cdot \right)-\omega^{i,1}] \right|^{p_k}_{g(x_i)}\, d\vol_{g(x_i)} + o_k(1).
	\end{align*}
	Then, we use the local strong convergence of \(u_k(a_k^{i,1}+\lambda_k^{i,1}\cdot)\) to \( \omega^{i,1}\) in \(\C^1_{\text{loc}}\left(\overline{\R^n_+};\R^d\right)\) and the convergence of \(\Omega_k^{i,1}\) towards \(\R^n_+\) to obtain
	\begin{multline}\label{eq:I}
		I=  (\lambda_k^{i,1})^{n-p_k}\int_{B(0,R_i)\cap \Omega_k^{i,1} } \left| d \left[ u_k\left(a_k^{i,1}+\lambda_k^{i,1}\cdot \right) \right] \right|^{p_k}_{g(x_i)}\, d\vol_{g(x_i)} \\
		- (\lambda_k^{i,1})^{n-p_k}\int_{B(0,R_i)^+}|d\omega^{i,1}|^{p_k}_{g(x_i)}\, d\vol_{g(x_i)} +f_I(k,R_i),
	\end{multline}
	where $f_I$ is a function satisfying 
	\begin{align*}
		\forall R>0,\ \ \ f_I(k,R)\xrightarrow[k\to+\infty]{} 0.
	\end{align*}
	Now we consider the term \(II\). Thanks to \eqref{eq:bound_lambda_0} and \Cref{lem:prop_bubbles}, it holds
	\begin{align*}
		& \int_{B(0,r_i)^+\setminus B\left(a_k^{i,1},R_i\lambda_k^{i,1}\right)} \left|d \left[ \omega^{i,1} \left(\frac{\cdot-a_k^{i,1}}{\lambda_k^{i,1}} \right) \right] \right|^{p_k}_g\, d\vol_g \\
		\leq &\ C(g)(\lambda_k^{i,1})^{n-p_k} \int_{\R^n_+ \setminus B(0,R_i)} |d\omega^{i,1}|^{p_k}_{g(x_i)}\, d\vol_{g(x_i)} \\
		  \leq &\ C(g)\frac{M}{\e_b^n}\|d\omega^{i,1}\|_{L^\infty(\R^n_+)}^{p_k-n} \int_{\R^n_+ \setminus B(0,R_i)} |d\omega^{i,1}|^n_{g(x_i)}\, d\vol_{g(x_i)}.
	\end{align*}
	The right-hand side is independent of $k$ and converges to 0 as $R_i\to\infty$. Hence we apply \Cref{lem:decoupling} to obtain
	\begin{align}\label{eq:II}
		II =\int_{B(0,r_i)^+ \setminus B\left(a_k^{i,1},R_i\lambda_k^{i,1}\right)} | du_k|^{p_k}_g\, d\vol_g +f_{II}(k,R_i),
	\end{align}
	where, thanks to \Cref{lem:decoupling} the function $f_{II}$ satisfies
	\begin{align*}
		 \sup_{k\in\N} f_{II}(k,R)  \leq C \left( \int_{\R^n_+\setminus B(0,R)} |d\omega^{i,1}|^n_{g(x_i)}\, d\vol_{g(x_i)}\right)^{\frac{1}{n}}.
	\end{align*}
	In particular, it holds \( \displaystyle{\limsup_{R\to +\infty} \left(\sup_{k\in\N} f_{II}(k,R)\right) }= 0.\)
	Thus, by summing \eqref{eq:I} and \eqref{eq:II}, we arrive at 
	\begin{align*}
		& \int_{B(0,r_i)^+}|dw_k^{i,1}|^{p_k}_g\, d\vol_g  \\
		= & \int_{B(0,r_i)^+} \left| d u_k \right|^{p_k}_{g}\, d\vol_{g}- (\lambda_k^{i,1})^{n-p_k}\int_{B(0,R_i)^+}|d\omega^{i,1}|^n_{g(x_i)}\, d\vol_{g(x_i)}+ f_I(k,R_i)+f_{II}(k,R_i) .
	\end{align*}
	We can take a subsequence such that, with \(\lambda_*^{i,1}= \displaystyle{\liminf_{k \to +\infty}} (\lambda_k^{i,1})^{n-p_k}\)
	\begin{align*}
		& \lim_{k\to +\infty}
		\left( \int_{B(0,r_i)^+}|dw_k^{i,1}|^{p_k}_g\, d\vol_g - \int_{B(0,r_i)^+} \left| d u_k \right|^{p_k}_{g}\, d\vol_{g} \right) \\
		= &\ \lambda_*^{i,1}\int_{B(0,R_i)^+}|d\omega^{i,1}|^n_{g(x_i)}\, d\vol_{g(x_i)}
		 + \sup_{k\in\N} f_{II}(k,R_i).
	\end{align*}
	Now taking the limit $R_i\to +\infty$, we obtain \eqref{eq:energy_identity_1}.
\end{proof}

\subsection{Step 4: Extraction of a possible second-generation bubble at a given point $x_i\in S$}\label{sec:Step4}

After extracting the first bubble, we have two possibilities:
\begin{enumerate}[label=Case \arabic*)]
	\item\label{case1} Assume that, up to a subsequence,
	\begin{align*}
		\lim_{k \to +\infty} \int_{B(0,r_i)^+} |dw_k^{i,1}|^{p_k}_g\, d\vol_g =\int_{B(0,r_i)^+} |du|^n_g\, d\vol_g.
	\end{align*} 
	Since $\left(w_k^{i,1} \right)_{k\in\N}$ weakly converges to $u$, we necessarily have, by Hölder inequality, that 
	\begin{align}\label{eq:intermediate}
		\lim_{k \to +\infty} \int_{B(0,r_i)^+} |dw_k^{i,1}|^{n}_g\, d\vol_g =\int_{B(0,r_i)^+} |du|^n_g\, d\vol_g.
	\end{align}
	 We conclude that \(w_k^{i,1} \rightarrow u\) strongly in \(W^{1,n}(B(0,r_i)^+;\R^d)\), we have extracted all the bubbles and Theorem \ref{th:main_bubbling} is proved. 
	 
	
	\item\label{case2} Assume that, up to a subsequence,
	\begin{align}\label{eq:ineq_p_energie}
		\lim_{k \to +\infty} \int_{B(0,r_i)^+} |d w_k^{i,1}|^{p_k}_g\, d\vol_g >\int_{B(0,r_i)^+} |du|^n_g\, d\vol_g.
	\end{align}
	Then some energy is remaining and the goal of Step 4 is to extract a second bubble.

\end{enumerate}

We define
\begin{equation*}
	S_i=\left\{ x \in \overline{B(0,r_i)^+}\colon \quad \limsup_{\rho \to 0} \liminf_{k \to +\infty}\int_{B(x,\rho )\cap B(0,r_i)^+} |dw_k^{i,1}|^{p_k}_g\, d\vol_g >0 \right\}.
\end{equation*}
It holds that \(S_i=\{0\}\) by assumption. Indeed, the sequence $(w_k^{i,1})_{k\in\N}$ cannot concentrate outside of the set of concentration of $(u_k)_{k\in\N}$. We  set 
\begin{equation}\label{eq:def_eps_1}
	\e_1^n \coloneqq  \limsup_{\rho\to 0} \liminf_{k \to +\infty}\int_{B(0,\rho )\cap B(0,r_i)^+}|d w_k^{i,1}|^{p_k}_g\, d\vol_g >0.
\end{equation}
Heuristically, $\eps_1$ is the sum of the energy of all the remaining bubbles generated at the origin. We define 
\begin{align*}
	Q_k^{i,2}(t)=\sup_{x \in B(0,r_i)^+ }\int_{B(x,t)\cap B(0,r_i)^+} |d w_k^{i,1}|^{p_k}_g\, d\vol_g.
\end{align*}
We again have that \(Q_k^{i,2}\) is a continuous function with \(Q_k^{i,2}(0)=0\) and,  for \(k\) large enough, \(Q_k^{i,2}(r_i)\geq \frac{7}{8}\e_1^{p_k}\). Hence there exist \( \left(\lambda_k^{i,2}\right)_{k\in\N}\subset (0,r_i)\) and \( \left(a_k^{i,2}\right)_{k\in\N}\subset \overline{B(0,r_i)^+}\) such that 
\begin{equation}\label{eq:existence_energy}
	\int_{B\left(a_k^{i,2},\lambda_k^{i,2} \right)\cap B(0,r_i)^+} |d w_k^{i,1}|^{p_k}_g\, d\vol_g=\min \left\{ \frac{3\e_1^{p_k}}{4},\frac{3\e_{\star}^{p_k}}{4}\right\}.
\end{equation}
The constant $\e_{\star}$ is defined in \eqref{eq:def_eps_star}. As in Claim \ref{claim_2}, we have \(\lambda_k^{i,2}\to 0\) and \(a_k^{i,2}\to 0\) when \(k \to +\infty\).  We now define 
\begin{equation*}
\forall x \in \Omega_k^{i,2} \coloneqq  \frac{B(0,r_i)^+-a_k^{i,2}}{\lambda_k^{i,2}} ,\quad v_k^{i,2}(x)=w_k^{i,1}(a_k^{i,2}+\lambda_k^{i,2}x)\text{ and } g_k^{i,2} \coloneqq  g(a^{i,2}_k + \lambda_k^{i,2}x).
\end{equation*}
 We have that 
\begin{equation}\nonumber
	\sup_{k\in\N} \int_{ \Omega_k^{i,2} }|d v_k^{i,2}|^{p_k}_{g_k^{i,2}}\, d\vol_{g_k^{i,2}} = \sup_{k\in\N} (\lambda_k^{i,2})^{p_k-n}\int_{B(0,r_i)^+ }|d w_{k}^{i,1}|^{p_k}_g\, d\vol_{g} <+\infty.
\end{equation}
Again we need to distinguish two cases:
\begin{enumerate}[label=Case \upshape(\Roman*)]
\item\label{hyp:weak_conv_whole} We assume that \(\frac{\dist (a_k^{i,2}, \p \R^n_+)}{\lambda_k^{i,2}} \xrightarrow[k\to+\infty]{} +\infty\).
In that case, up to a subsequence,
\begin{equation}\label{eq:weak_conv_step_2whole}
v_k^{i,2} \rightharpoonup \omega^{i,2} \text{ in } W^{1,n}_{\text{loc}}(\R^n;\R^d).
\end{equation}

\item\label{hyp:weak_conv_half} We assume that \(\frac{ \dist (a_k^{i,2}, \p \R^n_+)}{\lambda_k^{i,2}} \xrightarrow[k\to+\infty]{} A\in [0,+\infty).\)
In that case, up to a subsequence, a translation and a rotation, it holds
\begin{equation}\label{eq:weak_conv_step_2half}
v_k^{i,2} \rightharpoonup \omega^{i,2} \text{ in } W^{1,n}_{\text{loc}}(\overline{\R^n_+};\R^d).
\end{equation}
\end{enumerate}

In any case we can prove the following:

\begin{lemma}\label{lem:new_bubble}
	\begin{enumerate}[label=\upshape\roman*)]
\item\label{item:comparison_speed}
If \(\lambda_k^{i,2} =o(\lambda_k^{i,1})\) then \(|a_k^{i,2}-a_k^{i,1}|/\lambda_k^{i,1}\rightarrow +\infty\) as \(k\rightarrow +\infty\).	
	
		\item\label{item:speeds} For all \(1\leq i \leq \ell_1\), it holds that 
		\begin{equation}\label{eq:separation_bubbles}
			\max \left\{ \frac{\lambda_k^{i,2}}{\lambda_k^{i,1}}, \frac{\lambda_k^{i,1}}{\lambda_k^{i,2}}, \frac{|a_k^{i,1}-a_k^{i,2}|}{\lambda_k^{i,1}+\lambda_k^{i,2}} \right\} \xrightarrow[k\to +\infty]{} +\infty.
		\end{equation}
	\end{enumerate}
\end{lemma}
For simplicity, we fix \(i=1\) in the whole proof. 
\begin{proof}[Proof of Lemma \ref{lem:new_bubble}]
Let us start with \Cref{item:comparison_speed}. By contradiction, if,  up to a subsequence, we have that \((a_k^{1,2}-a_k^{1,1})/\lambda_k^{1,1}\xrightarrow[k \rightarrow +\infty]{} a \). Then we can write that
\begin{align}
\min\left\{ \frac{3\e_1^{p_k}}{4}, \frac{3\e_{\star}^{p_k}}{4} \right\}&=\int_{B(a_k^{1,2},\lambda_k^{1,2})\cap B(0,r_1)^+}\left|d \left[u_k-\omega^{1,1}\left(\frac{\cdot-a_k^{1,1}}{\lambda_k^{1,1}} \right)\right]\right|_g^{p_k} \, d \vol_g\nonumber \\
&=(\lambda_k^{1,1})^{n-p_k}\int_{B\left(\frac{a_k^{1,2}-a_k^{1,1}}{\lambda_k^{1,1}},\frac{\lambda_k^{1,2}}{\lambda_k^{1,1}}\right)\cap \Omega_k^{1,1}}\left|d[u_k(a_k^{1,1}+\lambda_k^{1,1}\cdot)-\omega^{1,1}] \right|_{g_k^{1,1}}^{p_k} \,  d \vol_{g_k^{1,1}}. \label{RHS}
\end{align}
But since \(B\left(\frac{a_k^{1,2}-a_k^{1,1}}{\lambda_k^{1,1}},\frac{\lambda_k^{1,2}}{\lambda_k^{1,1}}\right)\rightarrow B(a,0)\) when \(k\to +\infty\) and since, from \eqref{eq:strong_conv_vn}, we know that \(u_k(a_k^{1,1}+\lambda_k^{1,1}\cdot)\rightarrow \omega^{1,1}\) in \(\C^1_{\text{loc}}(\overline{\R^n_+};\R^d)\) we can use \eqref{eq:bound_lambda_0} to deduce that the right-hand side of \eqref{RHS} goes to zero as \(k \to+\infty\). This provides a contradiction.

	We now prove \Cref{item:speeds}. Suppose \eqref{eq:separation_bubbles} were false, then there would exist \(R>0\) such that \(R^{-1}\leq \frac{\lambda_k^{1,2}}{\lambda_k^{1,1}}\leq R\) and \(\frac{|a_k^{1,1}-a_k^{1,2}|}{\lambda_k^{1,1}+\lambda_k^{1,2}}\leq R\). By \eqref{eq:strong_conv_vn} and \eqref{eq:bound_lambda_0}, it holds 
	\begin{align*}
		\min\left \{\frac{3\e_1^n}{4},\frac{3\e_{\star}^n}{4}\right\} &= \int_{B\left(a_k^{1,2},\lambda_k^{1,2} \right)\cap B(0,r_1)^+}|d w_k^{1,1}|^{p_k}_{g}\, d\vol_{g} \\
		& \leq \int_{ B\left(a_k^{1,1}, (R^2+2R)\lambda_k^{1,1}\right)\cap B(0,r_1)^+} |d w_k^{1,1}|^{p_k}_{g}\, d\vol_{g}\\
		&\leq (\lambda_k^{1,1})^{n-p_k}\int_{B(0,R^2+2R)\cap \Omega_k^{1,1} } \left| d \left[ u_k(a_k^{1,1}+\lambda_k^{1,1}\cdot)-\omega^{1,1}\right] \right|^{p_k}_{g_k^{1,1}} \, d\vol_{g_k^{1,1}} \xrightarrow[k \to +\infty]{} 0.
	\end{align*}
	We obtained a contradiction and thus \eqref{eq:separation_bubbles} holds.
	
We skip the rest of the proof of the extraction of the 2nd generation bubbles since it can be proved as in the iteration procedure that we will write in details in the next step. We just state the two remaining main facts of this step:

 \begin{claim}\label{claim:strong_conv_gen_2}
 	The sequences $(v_k^{1,2})_{k\in\N}$ and \((u_k(a_k^{1,2}+\lambda_k^{1,2}\cdot))_{k\in\N}\) converge to $\omega^{1,2}$ in \(\C^1_{\text{loc}}(\overline{\R^n_+} \setminus \mathcal{S}_1;\R^d)\) where \(\mathcal{S}_1\) is a finite set with at most one point. We also have that  $(v_k^{1,2})_{k\in\N}$ converges strongly in $W^{1,n}_{\loc}(\overline{\R^n_+};\R^d)$ and \(\omega^{1,2}\) is a non-trivial bubble. Furthermore \(\e_1\geq \e_b\) where \(\e_1\) is defined in \eqref{eq:def_eps_1} and \(\e_b\) appears in \Cref{lem=quantum_energy}, and
 	\begin{equation*}
 	\int_{B(a_k^{1,2},\lambda_k^{1,2})\cap B(0,r_1)^+}| dw_k^{1,1}|_g^{p_k} d \vol_g =\frac{3\e_{\star}^{p_k}}{4}.
\end{equation*} 	 
 \end{claim}

\begin{claim}\label{claim:energy_id_2nd_gen}
The following energy identity holds
\begin{equation}\label{eq:energy_identity_2}
		\int_{B(0,r_1)^+} |dw^{1,2}_k|_g^{p_k} d \vol_g =  \int_{B(0,r_1)^+} |du_k|^{p_k}_g d\vol_g- \sum_{j=1}^2 \lambda_*^{j,1}\int_{\R^n_+} |d\omega^{1,j}|_{g(x_1)}^n d \vol_{g(x_1)} +o_k(1),
	\end{equation}
	where \(\lambda_*^{j,1}\coloneqq  \displaystyle{\liminf_{k \to +\infty}} (\lambda_k^{j,1})^{n-p_k}\in \left[ 1,\frac{M}{\e_b^n} \right]\) for \(j=1,2\).
\end{claim}

\subsection{Step 5: Induction for the extraction of all the bubbles at a given point of concentration}
 We focus on the case $i=1$, i.e., we look at the possible next-generation bubbles only at the concentration point \(x_1\). Everything that we will say can be repeated near the other points of concentrations. We show that we can iterate the arguments of Step 4 (second-generation bubble) to obtain the next generation bubbles. Since at each generation, the energy decreases by at least \(\e_{\star}^n\), this process stops after a finite number of iterations. Let \(\pp\geq 2\) be an integer and assume that we have
\begin{enumerate}
	\item Bubbles (in the sense of Definition \ref{def:bubble}) $\omega^{1,1},\ldots,\omega^{1,\pp}$,
	\item Sequences of positive numbers converging to $0$: $(\lambda^{1,1}_k)_{k\in\N},\ldots,(\lambda^{1,\pp}_k)_{k\in\N} \subset (0,r_1)$,
	\item Sequences of points in $\overline{B(0,r_1)^+}$ converging to $0$: $(a^{1,1}_k)_{k\in\N},\ldots, (a^{1,\pp}_k)_{k\in\N}$,
	\item Finite sets with at most \(\pp\) points: \(\mathcal{S}_1,\dots,\mathcal{S}_{\pp-1}\)
	\item\label{asump:tj} A sequence of increasing numbers $(t_j)_{j\in\N} \subset \left(\frac{1}{2},1\right)$,
\end{enumerate}
such that the following properties are satisfied:
\begin{enumerate}[label=HR-\arabic*$)_\pp$]
	\item\label{HR1p} For all $1\leq \qq\leq \pp$ and \(x\in B(0,r_1)^+\) we define\footnote{Note that to be able to define this function on all \(B(0,r_1)^+\) we can extend the bubbles to all \(\R^n\) by parity.}
	\[w^{1,\qq}_k(x) \coloneqq  u_k(x) - \sum_{j=1}^\qq \omega^{1,j}\left(\frac{x - a^{1,j}_k }{\lambda^{1,j}_k }\right)-\sum_{j=1}^\qq \omega^{1,j}(\infty).\]
	 If \(\qq\in \{1,\dots,\pp-1\}\) then we have the following convergences:
	\begin{align}\label{eq:HR1p_1bis}
		w^{1,\qq}_k (a^{1,\qq+1}_k + \lambda_k^{1,\qq+1}\cdot ) \text{ and } u_k(a^{1,\qq+1}_k + \lambda_k^{1,\qq+1}\cdot )\xrightarrow[k\to +\infty]{\text{strongly in } \C^1_{\loc}(\overline{\R^n_+}\setminus \mathcal{S}_q;\R^d)} \omega^{1,\qq+1}
	\end{align}	
	and, 
	\begin{align}\label{eq:HR1p_1}
		w^{1,\qq}_k (a^{1,\qq+1}_k + \lambda_k^{1,\qq+1}\cdot ) \xrightarrow[k\to +\infty]{\text{strongly in } W^{1,n}_{\loc}(\overline{\R^n_+};\R^d)} \omega^{1,\qq+1}.
	\end{align}
	

	If \(\qq\in \{1,\dots,\pp \}\), we denote \( \Omega_k^{1,\qq}\coloneqq (B(0,r_1)^+-a_k^{1,\qq})/\lambda_k^{1,\qq}\) and $g_k^{1,\qq} \coloneqq  g(a_k^{1,\qq} + \lambda_k^{1,\qq}\cdot)$, then it holds
	\begin{align}\label{eq:HR1p_2}
		\forall K>0,\quad \int_{B(0,K)\cap \Omega_k^{1,\qq}} \left| d\left[ w^{1,\qq}_k (a^{1,\qq}_k + \lambda_k^{1,\qq}\cdot )\right] \right|^{p_k}_{g_k^{1,\qq}}\, d\vol_{g_k^{1,\qq}} \xrightarrow[k\to +\infty]{} 0.
	\end{align}
	We also have the following property, where the constant $\e_{\star}$ is defined in \eqref{eq:def_eps_star}:
	\begin{align}\label{eq:HR1p_3}
		\forall k\in\N, \quad \int_{B\left(a_k^{1,\qq+1},\lambda_k^{1,\qq+1}\right)\cap B(0,r_1)^+} |dw^{1,\qq}_k|_g^{p_k}\, d \vol_g = t_\qq \e_{\star}^{p_k}.
	\end{align}
	\item\label{HR2p} We have the following relations:
	\begin{align*}
		\forall j\neq \qq\in \{1,\ldots,\pp\}, \quad \frac{\lambda_k^{1,j}}{\lambda_k^{1,\qq}} + \frac{ \lambda_k^{1,\qq} }{ \lambda_k^{1,j} } + \frac{ |a^{1,j}_k - a^{1,\qq}_k | }{\lambda^{1,j}_k + \lambda^{1,\qq}_k } \xrightarrow[k\to +\infty]{} +\infty.
	\end{align*}
	\item\label{HR3p} If $j<\qq$ then 
	\begin{align*}
		\Big( \lambda^{1,\qq}_k =o_k(\lambda_k^{1,j} ) \Big) \Rightarrow \left( \frac{ |a^{1,j}_k - a^{1,\qq}_k| }{ \lambda_k^{1,j} } \xrightarrow[k\to +\infty]{} +\infty \right).
	\end{align*}
	\item\label{HR4p} For any  \(j\in \{1,\dots,\pp\}\) we denote $\lambda_*^{1,j} \coloneqq  \displaystyle{\liminf_{k \to +\infty}} (\lambda_k^{1,j})^{n-p_k}$ then, with \(M\) the energy bound appearing in Theorem \ref{th:main_bubbling} and \(\e_b\) is defined in \Cref{lem=quantum_energy}, we have
	\begin{equation}\label{eq:boundedness_lambda}	
		 \lambda_*^{1,j}\in \left[1,\frac{M}{\e_b^n} \right].
	 \end{equation}  
	 For any \(\qq\in \{1,\dots,\pp\}\) we have the no-neck energy property:
	\begin{align*}
		\int_{B(0,r_1)^+} |dw^{1,\qq}_k|_g^{p_k} d \vol_g = \int_{B(0,r_1)^+} |du_k|_g^{p_k} d \vol_g- \sum_{j=1}^\qq \lambda_*^{1,j}\int_{\R^n_+} |d\omega^{1,j}|_g^n \, d \vol_g +o_k(1).
	\end{align*}
\end{enumerate}


We also assume that 
\begin{align}\label{eq:gap_concentration}
	\liminf_{k\to +\infty} \int_{B(0,r_1)^+} |dw^{1,\pp}_k|_g^{p_k}\,  d \vol_g  > \int_{B(0,r_1)^+} |du|_g^n \, d \vol_g .
\end{align}
Otherwise, we have extracted all the bubbles and the sequence $\left(w^{1,\pp}_k\right)_{k\in\N}$ converges strongly to $u$ in $W^{1,n}_{\loc}(\overline{B(0,r_1)^+};\R^d)$, cf.\ \ref{case1} and \ref{case2} at the beginning of \Cref{sec:Step4}. Under the assumption \eqref{eq:gap_concentration}, we call
\begin{equation}\label{def:epsp}
\e_\pp^n\coloneqq \limsup_{\rho \to 0} \lim_{k \to+\infty} \int_{B(0,\rho)^+}|d w_k^{1,\pp}|_g^{p_k} \,  d \vol_g>0.
\end{equation}
 We consider the concentration function
\begin{align*}
	\forall \lambda>0,\quad Q_k^{1,\pp}(\lambda) \coloneqq  \sup_{ x\in B(0,r_1)^+} \int_{ B(x,\lambda)\cap B(0,r_1)^+} |dw^{1,\pp}_k|_g^{p_k} \,  d \vol_g.
\end{align*}
Since $Q_k^{1,\pp}$ is continuous with $Q_k^{1,\pp}(0)=0$ and $\displaystyle{\limsup_{k\to +\infty}}\, Q_k^{1,\pp}(r_1)\geq \eps_\pp^n$, there exists $\lambda_k^{1,\pp+1}\in (0,r_1)$ and $a_k^{1,\pp+1}\in \overline{B(0,r_1)^+}$ such that 
\begin{align}\label{eq:quanta_energy}
	t_\pp \min\left( \eps_{\star}^{p_k},\eps_\pp^{p_k} \right) = Q_k^{1,\pp}(\lambda_k^{1,\pp+1}) = \int_{ B\left(a_k^{1,\pp+1},\lambda_k^{1,\pp+1} \right)\cap B(0,r_1)^+} |dw^{1,\pp}_k|^{p_k} \, d \vol_g.
\end{align}
The number $\e_{\star}$ was defined in \eqref{eq:def_eps_star}. Up to reducing $\lambda_k^{1,\pp+1}$, we can assume that $\lambda_k^{1,\pp+1}$ is the smallest radius satisfying the above condition. Given $\qq\in\{1,\ldots, \pp\}$, we define
\begin{align*}
	\forall x \in \Omega_k^{1,\qq+1} \coloneqq  \frac{ B(0,r_1)^+ - a_k^{1,\qq+1} }{ \lambda_k^{1,\qq+1} },\quad v^{1,\qq+1}_k(x) \coloneqq  w^{1,\qq}_k\left( a_k^{1,\qq+1} + \lambda_k^{1,\qq+1} x \right).
\end{align*}
We have the following global estimate:
\begin{equation*}
	\sup_{k\in\N} \int_{ \Omega_k^{1,\pp+1} }|d v_k^{i,\pp+1}|^{p_k}_{g_k^{1,\pp+1}}\, d\vol_{g_k^{1,\pp+1}} = \sup_{k\in\N} (\lambda_k^{1,\pp+1})^{p_k-n}\int_{B(0,r_1)^+ }|d w_{k}^{1,\pp}|^{p_k}_g\, d\vol_{g} <+\infty.
\end{equation*}
Thus $\left(v_k^{1,\pp+1}\right)_{k\in\N}$ has locally bounded energy and we deduce that 
either
\begin{enumerate}[label=Case (\arabic*)]
\item \label{Case_1_p+1} \( \dist(a_k^{1,\pp+1},\p \R^n_+)/\lambda_k^{1,\pp+1}\to +\infty\) and \(v_k^{1,\pp+1}\) converges weakly to \(\omega^{1,\pp+1}\) in \(W^{1,n}_{\text{loc}}(\R^n;\R^d),\) or
\item \label{Case_2_p+1} \( \dist(a_k^{1,\pp+1},\p \R^n_+)/\lambda_k^{1,\pp+1}\to A\) and, up to a translation and a rotation \(v_k^{1,\pp+1}\) converges weakly to \(\omega^{1,\pp+1}\) in \(W^{1,n}_{\text{loc}}(\overline{\R^n_+};\R^d),\).
\end{enumerate}

\subsubsection*{Let us show HR-\(3)_{p+1}\).}
We assume that there exists \(\rr \in \{1,\dots,\pp\}\) such that \(\lambda_k^{1,\pp+1}=o_k(\lambda_k^{1,\rr})\). By contradiction, we assume that $\rr$ is maximal for the condition
\begin{equation}\label{cond:rmaximal}
	\limsup_{k\to +\infty} \frac{|a_k^{1,\pp+1}-a_k^{1,\rr}|}{\lambda_k^{1,\rr}} <+\infty.
\end{equation}
Up to a subsequence, we can assume that \( \frac{a_k^{1,\pp+1}-a_k^{1,\rr}}{\lambda_k^{1,\rr}}\rightarrow a\) when \(k \to +\infty\). We have that 
\begin{align}
  t_\pp\min\{\e_\pp^{p_k},\e_{\star}^{p_k}\}
 = & \int_{B\left(a_k^{1,\pp+1},\lambda_k^{1,\pp+1}\right)\cap B(0,r_1)^+}|dw_k^{1,\pp}|_g^{p_k} d \vol_g \nonumber \\
 \leq &\ C \int_{B\left(a_k^{1,\pp+1},\lambda_k^{1,\pp+1}\right)\cap B(0,r_1)^+} |dw_k^{1,\rr}|_g^{p_k}\,  d \vol_g \nonumber \\
 & +C\sum_{\qq=\rr+1}^\pp \|d\omega^{1,\qq}\|_{L^{\infty}}^{p_k-n} \int_{B\left(\frac{a_k^{1,\pp+1}-a_k^{1,\qq}}{\lambda_k^{1,\qq}},\frac{\lambda_k^{1,\pp+1}}{\lambda_k^{1,\qq}} \right)}|d\omega^{1,\qq}|_g^{n} \, d \vol_g. \label{eq:HR3p1}
\end{align}
For the first term in the right-hand side we observe that
\begin{multline*}
	\int_{B\left(a_k^{1,\pp+1},\lambda_k^{1,\pp+1}\right)\cap B(0,r_1)^+} |dw_k^{1,\rr}|^{p_k}_g d\vol_g \\
	=(\lambda_k^{1,\rr})^{n-p_k}\int_{B\left(\frac{a_k^{1,\pp+1}-a_k^{1,\rr}}{\lambda_k^{1,\rr}},\frac{\lambda_k^{1,\pp+1}}{\lambda_k^{1,\rr}} \right)\cap \Omega_k^{1,\rr} }|d [w_k^{1,\rr}(a_k^{1,\rr}+\lambda_k^{1,\rr}\cdot)]|_{g_k^{1,\rr}}^{p_k} d\vol_{g_k^{1,\rr}},
\end{multline*} 
 where \( g_k^{1,\rr}(x)=g_k(a_k^{1,\rr+1}+\lambda_k^{1,\rr+1} x)\). Under our assumptions \(B\left(\frac{a_k^{1,\pp+1}-a_k^{1,\rr}}{\lambda_k^{1,\rr}},\frac{\lambda_k^{1,\pp+1}}{\lambda_k^{1,\rr}} \right) \) shrinks to the point \(a\). From \ref{HR1p}-\eqref{eq:HR1p_2}, we know that \(w_k^{1,\rr}(a_k^{1,\rr}+\lambda_k^{1,\rr}\cdot)\) converges strongly in \(W^{1,n}_{\text{loc}}(\overline{\R^n_+};\R^d)\) towards \(0\). Hence, by using \eqref{eq:boundedness_lambda} we find that
\begin{equation*}
\lim_{k \to +\infty }\int_{B\left(a_k^{1,\pp+1},\lambda_k^{1,\pp+1}\right)\cap B(0,r_1)^+ } |dw_k^{1,\rr}|_g^{p_k} \, d \vol_g=0.
\end{equation*} 
 Now let \(\qq\in \{\rr+1,\dots,\pp\}\). In order to show that each term of \eqref{eq:HR3p1} converges to zero, we distinguish different cases.
 \begin{itemize}
 \item[1)] If \(\lambda_k^{1,\pp+1}/\lambda_k^{1,\qq}\to 0\) as $k\to +\infty$.
 \end{itemize}
  Then, since  the radius of the ball \(B \left( \frac{a_k^{1,\pp+1}-a_k^{1,\qq}}{\lambda_k^{1,\qq}} ,\frac{\lambda_k^{1,\pp+1}}{\lambda_k^{1,\qq}} \right)\) shrinks to $0$, we obtain
  \begin{equation}\label{eq:conv_to_zero}
  \int_{B \left( \frac{a_k^{1,\pp+1}-a_k^{1,\qq}}{\lambda_k^{1,\qq}} ,\frac{\lambda_k^{1,\pp+1}}{\lambda_k^{1,\qq}} \right)}|d\omega^{1,\qq}|^n_g d \vol_g \xrightarrow[k\to+\infty]{} 0.
  \end{equation}
  
  \begin{itemize}
  \item[2)] If there exists \(R>0\) such that for all $k\in\N$, we have \( R^{-1} \leq \frac{\lambda_k^{1,\pp+1}}{\lambda_k^{1,\qq}}\leq R\).
  \end{itemize}
 Then we necessarily have that \(\frac{|a_k^{1,\pp+1}-a_k^{1,\qq}|}{\lambda_k^{1,\qq}}\to +\infty\) as $k\to +\infty$, and we conclude by dominated convergence that \eqref{eq:conv_to_zero} also holds true in that case. Indeed if \(\displaystyle{ \limsup_{k \to +\infty}} \frac{|a_k^{1,\pp+1}-a_k^{1,\qq}|}{\lambda_k^{1,\qq}}\leq C\) for some \(C\), then $\qq>\rr$ and also satisfies the condition \eqref{cond:rmaximal}. But $\rr$ has been chosen as the maximal integer satisfying \eqref{cond:rmaximal}, this is a contradiction.
  \begin{itemize}
  \item[3)] If there exists \(\qq\in \{\rr+1,\dots,\pp\}\) such that \(\frac{\lambda_k^{1,\pp+1}}{\lambda_k^{1,\qq}}\to +\infty\) as $k\to +\infty$.
  \end{itemize}
  Then, we call \(\qq_0\) the smallest such \(\qq\). Since \(\frac{\lambda_k^{1,\pp+1}}{\lambda_k^{1,\rr}}\to 0\) by assumption, we also have \(\frac{\lambda_k^{1,\rr}}{\lambda_k^{1,\qq_0}}\to +\infty\) as $k\to +\infty$. Thanks to \ref{HR3p},  since \(\qq_0>\rr\), we have that \( \frac{|a_k^{1,\qq_0}-a_k^{1,\rr}|}{\lambda_k^{1,\rr}}\to +\infty\) as $k\to +\infty$. We observe that, by \ref{HR1p}-\eqref{eq:HR1p_3},
  \begin{align*}
  \frac{t_{\qq_0-1}}{t_{\rr-1}}Q_k^{\rr-1}(\lambda_k^{1,\rr})&=Q_k^{\qq_0-1}(\lambda_k^{1,\qq_0}) \\
  &=\int_{B\left(a_k^{1,\qq_0},\lambda_k^{1,\qq_0} \right)\cap B(0,r_1)^+} |dw_k^{1,\qq_0-1}|_g^{p_k} \, d \vol_g \\
  &=\int_{B\left(  a_k^{1,\qq_0},\lambda_k^{1,\qq_0} \right)\cap B(0,r_1)^+} \left|d\left[w_k^{\rr-1}+\sum_{\ss=\rr}^{\qq_0-1} \omega^{1,\ss}\left(\frac{\cdot-a_k^{1,\ss}}{\lambda_k^{1,\ss}} \right)\right] \right|_g^{p_k} \, d \vol_g.
  \end{align*}

  \begin{claim}\label{cl:conv_zero_minor_bulles}
  	For all \(\ss \in \{\rr,\dots,\qq_0-1\}\), we have that
  	\begin{align} \label{eq:conv_zero_minor_bulles}
  		& \int_{B\left(a_k^{1,\qq_0},\lambda_k^{1,\qq_0}\right)}\left|d \left[ \omega^{1,\ss}\left(\frac{\cdot-a_k^{1,\ss}}{\lambda_k^{1,\ss}} \right) \right] \right|^{p_k}_g \, d\vol_g \\
  		 \leq &\ C \int_{B\left(a_k^{1,\qq_0},\lambda_k^{1,\qq_0}\right)}\left|d \left[ \omega^{1,\ss}\left(\frac{\cdot-a_k^{1,\ss}}{\lambda_k^{1,\ss}} \right) \right] \right|^n_g \, d\vol_g \xrightarrow[ k\to +\infty]{} 0. \nonumber
  	\end{align}
  \end{claim}
  Let us assume for a moment that \Cref{cl:conv_zero_minor_bulles} holds. Since \(\frac{t_{\qq_0-1}}{t_{\rr-1}}>1\) by definition of the sequence $(t_j)_{j\in\N}$ in \eqref{asump:tj}, we obtain that 
  \begin{equation*}
  	Q_k^{1,\rr-1}(\lambda_k^{1,\rr})<Q_k^{1,\qq_0-1}(\lambda_k^{1,\qq_0}).
  \end{equation*}
  From Lemma \ref{lem:decoupling}, for \(k\) large enough, we deduce that 
  \begin{align*}
  	Q_k^{1,\rr-1}(\lambda_k^{1,\rr})<\int_{B\left(a_k^{1,\qq_0}, \lambda_k^{1,\qq_0}\right)\cap B(0,r_1)^+ } |dw_k^{1,\rr-1}|_g^{p_k} \, d \vol_g \leq Q_k^{1,\rr-1}(\lambda_k^{1,\qq_0}).
  \end{align*}
  This is a contradiction, since \(\lambda_k^{1,\qq_0}=o_k( \lambda_k^{1,\rr})\) and \(Q_k^{1,\rr-1}\) is a nondecreasing function. Hence, with the help of Claim \ref{cl:conv_zero_minor_bulles}, the case $3)$ never occurs. We now prove \Cref{cl:conv_zero_minor_bulles}. 
  
  \begin{proof}[Proof of \Cref{cl:conv_zero_minor_bulles}]
  	The first inequality follows from the boundedness  of the derivatives of the bubbles, \Cref{lem:prop_bubbles}, and \eqref{eq:boundedness_lambda}. Let \(\ss \in \{\rr,\dots,\qq_0-1\}\). By a change of variables, we obtain
  	\begin{equation*}
  		\int_{B\left(a_k^{1,\qq_0},\lambda_k^{1,\qq_0} \right)} \left| d\left[\omega^{1,\ss}\left(\frac{\cdot-a_k^{1,\ss}}{\lambda_k^{1,\ss}} \right) \right] \right|_g^n d \vol_g = \int_{ B\left(\frac{a_k^{1,\qq_0}-a_k^{1,\ss}}{\lambda_k^{1,\ss}},\frac{\lambda_k^{1,\qq_0}}{\lambda_k^{1,\ss}} \right)} |d \omega^{1,\ss}|_{g_k^{1,\ss}}^n d \vol_{g_k^{1,\ss}}.
  	\end{equation*}
  	
  	We again distinguish cases
  	\begin{itemize}
  		\item[a)] If \(\frac{\lambda_k^{1,\qq_0}}{\lambda_k^{1,\ss}}\to 0\) as $k\to +\infty$. By dominated convergence, we obtain
  		\begin{equation}\label{eq:shrink_bubble}
  			\int_{B \left(\frac{a_k^{1,\qq_0}-a_k^{1,\ss}}{\lambda_k^{1,\ss}},\frac{\lambda_k^{1,\qq_0}}{\lambda_k^{1,\ss}} \right)}|d \omega^{1,\ss}|_{g_k^{1,\ss}}^n d \vol_{g_k^{1,\ss}}	\to 0.
  		\end{equation}
  		\item[b)] If there exists \(R>0\) such that for all $k\in\N$, it holds \(R^{-1}\leq \frac{\lambda_k^{1,\qq_0}}{\lambda_k^{1,\ss}}\leq R\). \\
  		By \ref{HR2p} we have that \(\frac{|a_k^{1,\ss}-a_k^{1,\qq_0}|}{\lambda_k^{1,\ss}}\to +\infty\). In that case, \eqref{eq:shrink_bubble} also follows from dominated convergence.
  		\item[c)] If \( \frac{\lambda_k^{1,\qq_0}}{\lambda_k^{1,\ss}}\to +\infty\) as $k\to +\infty$.
  	 	Then we have \(\ss<\qq_0\) and $
  		\frac{\lambda_k^{1,\pp+1}}{\lambda_k^{1,\ss}} = \frac{ \lambda_k^{1,\pp+1} }{ \lambda_k^{1,\qq_0} }\cdot \frac{ \lambda_k^{1,\qq_0} }{ \lambda_k^{1,\ss} }\xrightarrow[k\to+\infty]{} +\infty$. This is a contradiction with our choice of \(\qq_0\) and this concludes the proof of \eqref{eq:conv_zero_minor_bulles}.
  	\end{itemize}
  \end{proof}
  To finish the proof of HR-\(3)_{\pp+1}\) we notice that we have proved that all the terms of the right-hand side of \eqref{eq:HR3p1} converge to $0$. This is impossible. Hence, such an integer $\rr$ cannot exist.
	\subsubsection*{Let us show HR-$2)_{\pp+1}$.}
	By contradiction, assume that there exists $\qq\in\{1,\ldots,\pp\}$ and $R>0$ such that 
	\begin{align*}
		\frac{\lambda_k^{1,\pp+1}}{\lambda_k^{1,\qq}} + \frac{ \lambda_k^{1,\qq} }{ \lambda_k^{1,\pp+1} } + \frac{ |a^{1,\pp+1}_k - a^{1,\qq}_k | }{\lambda^{1,\pp+1}_k + \lambda^{1,\qq}_k } \leq R.
	\end{align*}
	Then we have the following inclusion: $B\left(a_k^{1,\pp+1},\lambda_k^{1,\pp+1} \right) \subset B\left(a_k^{1,\qq},(R^2+2R)\lambda_k^{1,\qq} \right)$. Up to a subsequence, we can assume that \((a_k^{1,\pp+1}-a_k^{1,\qq})/\lambda_k^\qq\xrightarrow[k \rightarrow +\infty]{} a \). Then it holds
	\begin{align}
		t_{\pp}\min\left\{ \e_{\pp}^{p_k}, \e_{\star}^{p_k} \right\}&=\int_{ B\left(a_k^{1,\pp+1},\lambda_k^{1,\pp+1} \right)\cap B(0,r_1)^+} |dw^{1,\pp}_k|_g^{p_k} \, d \vol_g. \nonumber \\
		& \leq \int_{ B\left(a_k^{1,\qq},(R^2+2R)\lambda_k^{1,\qq} \right)\cap B(0,r_1)^+} |dw^{1,\pp}_k|_g^{p_k} \, d \vol_g. \nonumber \\
		& \leq C(n,p) \int_{ B\left(a_k^{1,\qq},(R^2+2R) \lambda_k^{1,\qq} \right)\cap B(0,r_1)^+} |dw^{1,\qq}_k|^{p_k}_g \, d \vol_g \nonumber\\
		& \quad + C(n,p)\sum_{\rr=\qq+1}^\pp \int_{ B\left(a_k^{1,\qq} , (R^2+2R) \lambda_k^{1,\qq} \right)^+} \left|d\left[ \omega^{1,\rr}\left( \frac{ \cdot - a_k^{1,\rr} }{ \lambda^{1,\rr} } \right) \right]\right|^{p_k}_g \, d \vol_g \nonumber \\
		&\leq C(n,p) (\lambda_k^{1,\qq})^{n-p_k}\int_{B\left(0, (R^2+2R) \right)\cap \Omega_k^{1,\qq}}\left|d[w^{1,\qq}_k(a_k^{1,\qq} + \lambda_k^{1,\qq}\cdot )] \right|_{g_k^{1,\qq}}^{p_k} \, d \vol_{g_k^{1,\qq}}. \label{RHS_0}\\
		& \quad + C(n,p)\sum_{\rr=\qq+1}^\pp \int_{ B\left( \frac{ a_k^{1,\qq} - a_k^{1,\rr} }{ \lambda_k^{1,\rr} }, (R^2+2R) \frac{ \lambda_k^{1,\qq} }{ \lambda_k^{1,\rr} } \right)} \left|d\omega^{1,\rr}\right|_{g_k^{1,\rr}}^{n} \, d \vol_{g_k^{1,\rr}}. \label{RHS_1}
	\end{align}
	 By \ref{HR1p} and \ref{HR4p}, more precisely \eqref{eq:HR1p_2} and \eqref{eq:boundedness_lambda}, we deduce that the term in \eqref{RHS_0} goes to zero as \(k \to+\infty\). We now focus on the term in \eqref{RHS_1}. Let $\rr\in\{\qq+1,\ldots,\pp\}$. We distinguish three cases.
    \begin{enumerate}
    	\item If \(\frac{\lambda_k^{1,\qq}}{\lambda_k^{1,\rr}}\to 0\) as $k\to +\infty$.
    By dominated convergence, we obtain
    \begin{equation}\label{eq:vanish}
    	\int_{ B\left( \frac{ a_k^{1,\qq} - a_k^{1,\rr} }{ \lambda_k^{1,\rr} }, (R^2+2R) \frac{ \lambda_k^{1,\qq} }{ \lambda_k^{1,\rr} } \right)} \left|d\omega^{1,\rr}\right|_{g_k^{1,\rr}}^n  d \vol_{g_k^{1,\rr}}
    	\xrightarrow[k\to+\infty]{} 0.
    \end{equation}
    	\item If there exists \(R>0\) such that for all $k\in\N$ we have \(R^{-1}\leq \frac{\lambda_k^{1,\qq}}{\lambda_k^{1,\rr}}\leq R\).\\
    	By \ref{HR2p}, it holds that
    	\begin{align*}
    		\frac{ |a_k^{1,\qq} - a_k^{1,\rr}| }{\lambda_k^{1,\rr}} \xrightarrow[k\to +\infty]{} +\infty.
    	\end{align*}
    	Hence, again by dominated convergence, we still obtain \eqref{eq:vanish}

    	\item If \( \frac{\lambda_k^{1,\qq}}{\lambda_k^{1,\rr}}\to +\infty\) as $k\to +\infty$. Since $\qq<\rr$, we deduce from \ref{HR3p} that 
    	\begin{align*}
    		\frac{ |a_k^{1,\qq} - a_k^{1,\rr}| }{\lambda_k^{1,\qq}} \xrightarrow[k\to +\infty]{} +\infty.
    	\end{align*}
    	Hence, the characteristic functions of the balls $B\left( \frac{ a_k^{1,\qq} - a_k^{1,\rr} }{ \lambda_k^{1,\rr} }, \frac{ \lambda_k^{1,\qq} }{ \lambda_k^{1,\rr} } \right)$ converges almost everywhere to $0$. Therefore, we again obtain \eqref{eq:vanish} from dominated convergence.
\end{enumerate}
Hence, the term in \eqref{RHS_1} converges to $0$ as $k\to +\infty$. This is a contradiction and we obtain HR-$2)_{\pp+1}$.

\subsubsection*{Proof of HR-$1)_{\pp+1}$}

We consider the following three sets of indices:
\begin{equation*}\
	\begin{aligned}
		& \A \coloneqq  \left\{ j \in\{ 1,\ldots,\pp\} \colon \limsup_{k\to +\infty} \left( \frac{ \lambda_k^{1,j} }{ \lambda_k^{1,\pp+1} } +\frac{ \lambda_k^{1,\pp+1} }{ \lambda_k^{1,j} } \right) <+\infty \quad \text{and} \quad \frac{ |a_k^{1,j} - a_k^{1,\pp+1}| }{ \lambda_k^{1,j} + \lambda_k^{1,\pp+1} } \xrightarrow[k\to+\infty]{} +\infty \right\},\\
		& \B \coloneqq  \left\{ j \in\{ 1,\ldots,\pp\} \colon  \left( \frac{ \lambda_k^{1,j} }{ \lambda_k^{1,\pp+1} } +\frac{ \lambda_k^{1,\pp+1} }{ \lambda_k^{1,j} } \right) \xrightarrow[k\to+\infty]{} +\infty \quad \text{and} \quad \frac{ |a_k^{1,j} - a_k^{1,\pp+1}| }{ \lambda_k^{1,j} + \lambda_k^{1,\pp+1} } \xrightarrow[k\to+\infty]{} +\infty \right\},\\
		& \C \coloneqq  \left\{ j \in\{ 1,\ldots,\pp\} \colon  \frac{ \lambda_k^{1,\pp+1} }{ \lambda_k^{1,j} } \xrightarrow[k\to+\infty]{} +\infty \quad \text{and} \quad \limsup_{k\to +\infty} \frac{ |a_k^{1,j} - a_k^{1,\pp+1}| }{ \lambda_k^{1,j} + \lambda_k^{1,\pp+1} } < +\infty \right\}.
	\end{aligned}
\end{equation*}

We decompose $w_k^{1,\pp}$ as a sum of four terms: 
\begin{align}\label{eq:decompo_wk}
	w_k^{1,\pp} = u_k - \sum_{\alpha\in \A} \omega^{1,\alpha} \left( \frac{ \cdot - a_k^{1,\alpha}}{ \lambda_k^{1,\alpha} } \right) - \sum_{\beta\in \B} \omega^{1,\beta} \left( \frac{ \cdot - a_k^{1,\beta}}{ \lambda_k^{1,\beta} } \right) - \sum_{\gamma \in \C} \omega^{1,\gamma} \left( \frac{ \cdot - a_k^{1,\gamma}}{ \lambda_k^{1,\gamma} } \right)-\sum_{j=1}^\pp \omega^{1,j}(\infty).
\end{align}
Note that \textit{a priori} \(w_k^{1,\pp}\) is well-defined on the domain
\begin{align*}
B(0,r_1)^+\cap \left( \bigcap_{\alpha \in \A} \{ (x-a_k^{1,\alpha})_n>0\}\right) \cap \left(\bigcap_{\beta  \in \B} \{ (x-a_k^{1,\beta})_n>0\} \right) \cap \left( \bigcap_{\gamma\in \C} \{ (x-a_k^{1,\gamma})_n>0\}\right).
\end{align*}
However we will consider it as well-defined on \(B(0,r_1)^+\) by extending all the bubbles as functions defined on \(\R^n\) by a reflection procedure.
We first show that the contribution coming from the groups $\A$ and $\B$ is negligible.

\begin{claim}\label{cl:AB_neg}
Define \(\Omega_k^{1,\pp+1}\coloneqq (B(0,r_1)^+-a_k^{1,\pp+1})/\lambda_k^{1,\pp+1}\). If $\alpha\in \A\cup \B$, then it holds for any \(L>0\),
	\begin{align*}
		\int_{B(0,L)\cap \Omega_k^{1,\pp+1}} \left| d \left[ \omega^{1,\alpha} \left( \frac{ a_k^{1,\pp+1} - a_k^{1,\alpha} + \lambda_k^{1,\pp+1} \cdot }{ \lambda_k^{1,\alpha} } \right) \right] \right|^{p_k}_{g_k^{1,\pp+1}} d \vol_{g_k^{1,\pp+1}}
		\xrightarrow[k \to +\infty]{} 0. 
	\end{align*}
\end{claim}

\begin{proof}
If $\alpha\in \A$, we set \(g_k^{1,\alpha}=g(a_k^{1,\alpha}+\lambda_k^{1,\alpha}\cdot)\). Then, a change of variables and the use of the dominated convergence theorem show that  
\begin{align}
	\int_{B(0,L)\cap \Omega_k^{1,\pp+1}} \left| d \left[ \omega^{1,\alpha} \left( \frac{ a_k^{1,\pp+1} - a_k^{1,\alpha} + \lambda_k^{1,\pp+1} \cdot }{ \lambda_k^{1,\alpha} } \right) \right] \right|^{p_k}_{g_k^{1,\pp+1}} d \vol_{g_k^{1,\pp+1}} \nonumber \\
	 \leq  \left( \frac{\lambda_k^{\pp+1}}{\lambda_k^{1,\alpha}}\right)^{p_k-n}\int_{ B\left( \frac{ a_k^{1,\pp+1} - a_k^{1,\alpha} }{ \lambda_k^{1,\alpha} }, \frac{ L\lambda_k^{1,\pp+1} }{ \lambda_k^{1,\alpha} } \right)} |d\omega^{1,\alpha} |_{g_k^{1,\alpha}}^{p_k} d \vol_{g_k^{1,\alpha}} 
	 \xrightarrow[k \to +\infty]{} 0. \label{eq:*}
\end{align}
If $\beta\in \B$, then a change of variables and  \eqref{eq:boundedness_lambda} show that  
\begin{align}
	& \int_{B(0,L)\cap \Omega_k^{1,\pp+1}} \left| d \left[ \omega^{1,\beta} \left( \frac{ a_k^{1,\pp+1} - a_k^{1,\beta} + \lambda_k^{1,\pp+1} \cdot }{ \lambda_k^{1,\beta} } \right) \right] \right|^{p_k}_{g_k^{1,\pp+1}} d \vol_{{g_k^{1,\pp+1}}} \nonumber \\
	 \leq  &\ \left(\frac{\lambda_k^{1,\pp+1}}{\lambda_k^{1,\beta}} \right)^{p_k-n}\int_{ B\left( \frac{ a_k^{1,\pp+1} - a_k^{1,\beta} }{ \lambda_k^{1,\beta} }, \frac{ L\lambda_k^{1,\pp+1} }{ \lambda_k^{1,\beta} } \right)} |d\omega^{1,\beta} |^{p_k}_{g_k^{1,\beta}} d \vol_{g_k^{1,\beta}} \nonumber \\
	 \leq &\ \frac{M}{\e_b^n} \|d\omega^{1,\beta}\|_{L^{\infty}(\R^n_+)}^{p_k-n} \int_{ B\left( \frac{ a_k^{1,\pp+1} - a_k^{1,\beta} }{ \lambda_k^{1,\beta} }, \frac{ L\lambda_k^{1,\pp+1} }{ \lambda_k^{1,\beta} } \right)} |d\omega^{1,\beta} |^{n}_{g_k^{1,\beta}} d \vol_{g_k^{1,\beta}}  \xrightarrow[k\to +\infty]{} 0. \label{eq:**}
\end{align}
In the last inequality, we used $\lambda_k^{1,\pp+1}\leq 1$. To show the last convergence we distinguish two cases. If $\frac{ \lambda_k^{1,\pp+1} }{ \lambda_k^{1,\beta} } \to 0$ as $k\to +\infty$, then the last integral also converges to $0$ as $k\to +\infty$. If $\frac{ \lambda_k^{1,\pp+1} }{ \lambda_k^{1,\beta} } \to +\infty$ as $k\to +\infty$, then by HR-$3)_{p+1}$, we deduce that \(\frac{ |a_k^{1,\pp+1} - a_k^{1,\beta}| }{ \lambda_k^{1,\pp+1} } \xrightarrow[k\to +\infty]{} +\infty\). And thus the characteristic function of the ball \(B\left( \frac{ a_k^{1,\pp+1} - a_k^{1,\beta} }{ \lambda_k^{1,\beta} }, \frac{ L\lambda_k^{1,\pp+1} }{ \lambda_k^{1,\beta} } \right)\) goes to zero almost everywhere. We conclude as well that the contribution of $\omega^{1,\beta}$ converges to $0$. We conclude the proof of the claim thanks to \Cref{lem:decoupling}.
\end{proof} 


The purpose of the two next claims is to show that we have strong convergence of the blow-up map \(w_k(a_k^{1,\pp+1}+\lambda_k^{1,\pp+1}\cdot)\) towards \(\omega^{1,\pp+1}\) outside a set with a finite number of points \(\mathcal{S}_\pp\). In the following claim, we show that the energy of each bubble of the group $\C$ is concentrated at the points $\displaystyle{\lim_{k\to+\infty}} \frac{ a_k^{1,\gamma}-a_k^{1,\pp+1}}{\lambda_k^{1,p+1}}$.

\begin{claim}\label{claim:C_outside_points}
	Up to a subsequence, we assume that for any $\gamma\in \C$, there exists \(a_{\gamma}\in \R^n\) such that \( \displaystyle{\lim_{k\to+\infty}} \frac{ a_k^{1,\gamma}-  a_k^{1,\pp+1}}{\lambda_k^{1,p+1}}=a_{\gamma}\). Then for all \(\gamma\) in \(\C\) and for all \(L,\eta>0\) we have
	\begin{equation}\label{eq:****}
	\int_{B(0,L)\setminus B(a_\gamma,\eta)} \left|d \Big[\omega^{1,\gamma}\left( \frac{a_k^{1,\pp+1}-a_k^{1,\gamma}+\lambda_k^{1,\pp+1}\cdot}{\lambda_k^{1,\gamma}} \right) \Big] \right|_{g_k^{1,\pp+1}}^{p_k} \, d \vol_{g_k^{1,\pp+1}} \xrightarrow[k\to +\infty]{} 0,
	\end{equation}
	\begin{align}\label{eq:*****}
		\limsup_{K\to +\infty} \limsup_{k \to +\infty} \int_{B(0,L)\setminus B\left(\frac{a_k^{1,\gamma}-a_k^{1,\pp+1}}{\lambda_k^{1,\pp+1}},K\frac{\lambda_k^{1,\gamma}}{\lambda_k^{1,\pp+1}} \right)} \left|d \Big[\omega^{1,\gamma}\left( \frac{a_k^{1,\pp+1}-a_k^{1,\gamma}+\lambda_k^{1,\pp+1}\cdot}{\lambda_k^{1,\gamma}} \right) \Big] \right|_{g_k^{1,\pp+1}}^{p_k} \, d \vol_{g_k^{1,\pp+1}} \nonumber \\
		= 0,
	\end{align}
\end{claim}

\begin{proof}
Indeed, for \(k\) large enough we have that \( B\left(\frac{a_k^{1,\gamma}-a_k^{1,\pp+1}}{\lambda_k^{1,\pp+1}},\frac{\eta}{2}\right)\subset B(a_\gamma,\eta)\) and thus \(B(a_\gamma,\eta)^c \subset  B\left(\frac{a_k^{1,\gamma}-a_k^{1,\pp+1}}{\lambda_k^{1,\pp+1}},\frac{\eta}{2} \right)^c\).  Thus
\begin{align*}
& \int_{B(0,L)\setminus B\left(\frac{a_k^{1,\gamma} - a_k^{1,\pp+1}}{\lambda_k^{1,\pp+1}},\frac{\eta}{2} \right)}\left|d \Big[\omega^{1,\gamma}\left( \frac{a_k^{1,\pp+1}-a_k^{1,\gamma}+\lambda_k^{1,\pp+1}\cdot}{\lambda_k^{1,\gamma}} \right)\Big] \right|_{g_k^{1,\pp+1}}^{p_k} \, d \vol_{g_k^{1,\pp+1}} \\
=& \left( \frac{\lambda_k^{1,\pp+1}}{\lambda_k^{1,\gamma}}\right)^{p_k-n} \int_{B\left( \frac{a_k^{1,\gamma}-a_k^{1,\pp+1} }{\lambda_k^{1,\pp+1}}, L\frac{\lambda_k^{1,\pp+1}}{\lambda_k^{1,\gamma}} \right) \setminus B\left(0,\frac{\eta \lambda_k^{1,\pp+1}}{2\lambda_k^{1,\gamma}} \right)}|d \omega^{1,\gamma}|^{p_k}_{g_k^{\gamma,\pp}} d \vol_{g_k^{\gamma,\pp}} \\
\leq &\ \frac{ M}{\e_b^n}\left\| d\omega^{1,\gamma}\right\|_{L^{\infty}(\R^n_+)}^{p_k-n} \int_{B\left(a_{\gamma}, 2L\frac{\lambda_k^{1,\pp+1}}{\lambda_k^{1,\gamma}} \right) \setminus B\left(0,\frac{\eta \lambda_k^{1,\pp+1}}{2\lambda_k^{1,\gamma}} \right)}|d \omega^{1,\gamma}|^n_{g_k^{\gamma,\pp}} d \vol_{g_k^{\gamma,\pp}}.
\end{align*}
Since the term $\left\| d\omega^{1,\gamma}\right\|_{L^{\infty}(\R^n_+)}^{p_k-n}$ is bounded and $\frac{\lambda_k^{1,\pp+1}}{\lambda_k^{1,\gamma}}$ goes to $+\infty$ as $k\to +\infty$, we deduce that the above expression converges to $0$ as $k\to +\infty$ by dominated convergence. Therefore, we obtain \eqref{eq:****}. 

To prove \eqref{eq:*****} we proceed in the same way: a change of variables shows that 
\begin{align*}
& \int_{B(0,L)\setminus B\left(\frac{a_k^{1,\gamma} - a_k^{1,\pp+1}}{\lambda_k^{1,\pp+1}},K\frac{\lambda_k^{1,\gamma}}{\lambda_k^{1,\pp+1}} \right)}\left|d \Big[\omega^{1,\gamma}\left( \frac{a_k^{1,\pp+1}-a_k^{1,\gamma}+\lambda_k^{1,\pp+1}\cdot}{\lambda_k^{1,\gamma}} \right)\Big] \right|_{g_k^{1,\pp+1}}^{p_k} \, d \vol_{g_k^{1,\pp+1}} \\
=& \left( \frac{\lambda_k^{1,\pp+1}}{\lambda_k^{1,\gamma}}\right)^{p_k-n} \int_{B\left( \frac{a_k^{1,\gamma}-a_k^{1,\pp+1} }{\lambda_k^{1,\pp+1}}, L\frac{\lambda_k^{1,\pp+1}}{\lambda_k^{1,\gamma}} \right) \setminus B\left(0,K \right)}|d \omega^{1,\gamma}|^{p_k}_{g_k^{\gamma,\pp}} d \vol_{g_k^{\gamma,\pp}} \\
\leq &\ \frac{ M}{\e_b^n}\left\| d\omega^{1,\gamma}\right\|_{L^{\infty}(\R^n_+)}^{p_k-n} \int_{\R^n_+ \setminus B\left(0,K\right)}|d \omega^{1,\gamma}|^n_{g_k^{\gamma,\pp}} d \vol_{g_k^{\gamma,\pp}}.
\end{align*}
\end{proof}

Away from these points, we show that the sequence $\left(u_k(a_k^{1,\pp+1}+\lambda_k^{1,\pp+1}\cdot) \right)_{k\in\N}$ converges strongly to $\omega^{1,\pp+1}$.

\begin{claim}\label{claim:strong_conv_u_kp+1}
With the notations of Claim \ref{claim:C_outside_points} we define \(\mathcal{S}_\pp\coloneqq \cup_{\gamma\in \C}\{a_\gamma\}\). We have that
\begin{itemize}
\item[1)] \( u_k(a_k^{1,\pp+1}+\lambda_k^{1,\pp+1}\cdot) \xrightarrow[]{k \to +\infty} \omega^{1,\pp+1}\) in \(\C^1_{\text{loc}}(\R^n\setminus \mathcal{S}_\pp;\R^d)\) or \(\C^1_{\text{loc}}(\overline{\R^n_+}\setminus \mathcal{S}_\pp;\R^d)\),

\item[2)] \( w_k(a_k^{1,\pp+1}+\lambda_k^{1,\pp+1}\cdot) \xrightarrow[]{k \to +\infty} \omega^{1,\pp+1}\) in \(\C^1_{\text{loc}}(\R^n\setminus \mathcal{S}_\pp;\R^d)\) or \(\C^1_{\text{loc}}(\overline{\R^n_+}\setminus \mathcal{S}_\pp;\R^d)\),


\item[3)] \((\lambda_k^{1,\pp+1})^{n-p_k}\leq M/\e_b^n\).
\end{itemize}
\end{claim}

\begin{proof}
Thanks to \Cref{cl:AB_neg}, \Cref{claim:C_outside_points} and \Cref{lem:decoupling}, for all \(x\) in \(\R^n\) and for all \(\eta>0\) 
\begin{multline*}
	\int_{[B(x,1)\setminus \cup_{\gamma\in \C}B(a_\gamma,\eta)]\cap \Omega_k^{1,\pp+1}} \left| d \left[ w_k^{1,\pp}(a_k^{1,\pp+1}+\lambda_k^{1,\pp+1}\cdot) \right] \right|^{p_k}_{g_k^{1,\pp+1}} \, d \vol_{g_k^{1,\pp+1}} \\
	= \int_{[B(x,1)\setminus \cup_{\gamma\in \C}B(a_\gamma,\eta)]\cap \Omega_k^{1,\pp+1}} \left|d \left[ u_k(a_k^{1,\pp+1}+\lambda_k^{1,\pp+1}\cdot) \right] \right|^{p_k}_{g_k^{1,\pp+1}} \, d \vol_{g_k^{1,\pp+1}} +f_1(k,x,\eta),
\end{multline*}
where for all \(x\in \R^n\) and for all \( \eta>0\) we have that \(f_1(k,x,\eta)\to 0\) when \(k \to +\infty\).

In particular, by using \eqref{eq:quanta_energy} and a change of variables we find that, for all \(x\) in \(\R^n\) and all \(\eta>0\), for \(k\) large enough,
\begin{equation*}
\int_{[B(x,1)\setminus \cup_{\gamma\in \C}B(a_\gamma,\eta)]\cap \Omega_k^{1,\pp+1}} \left|d \left[ u_k(a_k^{1,\pp+1}+\lambda_k^{1,\pp+1}\cdot) \right] \right|^{n}_{g_k^{1,\pp+1}} \, d \vol_{g_k^{1,\pp+1}}<\e_{\star}^{n}.
\end{equation*}
Since \( u_k(a_k^{1,\pp+1}+\lambda_k^{1,\pp+1}\cdot)\) is a free-boundary $p_k$-harmonic map in \(\Omega_k^{1,\pp+1}\), we can apply the \(\e\)-regularity of \Cref{pr:Convergence} to deduce that \( u_k(a_k^{1,\pp+1}+\lambda_k^{1,\pp+1}\cdot)\) converges to some limit in \(\C^1_{\text{loc}}(\R^n\setminus \mathcal{S}_\pp;\R^d)\) or in \(\C^1_{\text{loc}}(\overline{\R^n_+}\setminus \mathcal{S}_\pp;\R^d)\). We can use \eqref{eq:*}-\eqref{eq:**}-\eqref{eq:****} again and the fact that \(w_k^{1,\pp}(a_k^{1,\pp+1}+\lambda_k^{1,\pp+1}\cdot)\) converges weakly in \(W^{1,n}_{\text{loc}}(\R^n;\R^d)\) or in \(W^{1,n}_{\text{loc}}(\overline{\R^n_+};\R^d)\) towards \(\omega^{1,\pp+1}\) to deduce that this limit is \(\omega^{1,\pp+1}\). The second item is now a consequence of the definition of \(w_k^{1,\pp}\).
The third item can be shown exactly as in \eqref{eq:bound_lambda_0}.
\end{proof}

In the following claim, we prove that the energy of $\left( w_k^{1,\pp+1}(a_k^{1,\pp+1}+\lambda_k^{1,\pp+1}\cdot)\right)_{k\in\N}$ is concentrated in a particular region, where the bubbles $\omega^{1,1},\ldots,\omega^{1,\pp+1}$ have no remaining energy in the limit. 

\begin{claim}\label{cl:estimate_strong_conv}
	Let $\mathcal{S}_\pp \coloneqq  \{a_{\gamma}\}_{\gamma\in \C}$. Given $b\in \mathcal{S}_{\pp}$, $\eta\in(0,1)$ and $K>1$, we denote
	\begin{align}
		\C_b  &\coloneqq \{\gamma\in \C : a_\gamma=b\}, \nonumber \\
		\U_{k,\eta,K} &\coloneqq  \bigcup_{b\in\mathcal{S}_\pp} \left( B(b,\eta)\setminus \bigcup_{\gamma\in\C_b} B\left( \frac{a_k^{1,\gamma}-a_k^{1,\pp+1}}{\lambda_k^{1,\pp+1}},K\frac{\lambda_k^{1,\gamma}}{\lambda_k^{1,\pp+1}} \right)  \right) \cap \Omega_k^{1,\pp+1}.\label{eq:U_ketaK}
	\end{align}
	
	We define \(v_k^{1,\pp+1}(x) \coloneqq w_k^{1,\pp}\left(a_k^{1,\pp+1}+\lambda_k^{1,\pp+1}x \right)\) for all \(x\in \Omega_k^{1,\pp+1}\). For any $L>1$ such that $\mathcal{S}_\pp\subset B(0,L)$, we have that
	\begin{align}
		 & \limsup_{k\to+\infty} \int_{B(0,L)\cap \Omega_k^{1,\pp+1}}|dv_k^{1,\pp+1}-d\omega^{1,\pp+1}|^{p_k}_{g_k^{1,\pp+1}} d \vol_{g_k^{1,\pp+1}}  \nonumber \\
		= &\ \limsup_{K\to +\infty}\  \limsup_{\eta\to 0}\ \limsup_{k\to +\infty}  \int_{ \U_{k,\eta,K} } \left|d\left[ u_k\left(a_k^{1,\pp+1} + \lambda_k^{1,\pp+1}\cdot\right) \right] \right|^{p_k}_{g_k^{1,\pp+1}} d \vol_{g_k^{1,\pp+1}}  . \label{eq:estimate_strong_conv}
	\end{align}
	Moreover, for any $1\leq \ss\leq \pp+1$, it holds
	\begin{align}\label{eq:no_energy_bubbles}
		0 = \limsup_{K\to +\infty}\ \limsup_{\eta\to 0}\ \limsup_{k\to+\infty} \int_{ \U_{k,\eta,K} } 
		\left| d\left[ \omega^{1,\ss}\left(\frac{a_k^{1,\pp+1}-a_k^{1,\ss} + \lambda_k^{1,\pp+1}\cdot}{\lambda_k^{1,\ss}}\right)\right] \right|^{p_k}_{g_k^{1,\pp+1}} d \vol_{g_k^{1,\pp+1}}.
	\end{align}
\end{claim}

\begin{proof}
Combining \Cref{cl:AB_neg} with \eqref{eq:****}, we find that for any $\eta\in(0,1)$ (we will send $\eta\to 0$ in the end) and for \(k\) large enough 
	\begin{align}
		& \int_{B(0,L)\cap \Omega_k^{1,\pp+1}}|dv_k^{1,p+1}-d\omega^{1,p+1}|^{p_k}_{g_k^{1,\pp+1}}d \vol_{g_k^{1,\pp+1}} \nonumber \\
		= & \int_{\left( \bigcup_{\gamma\in \C} B(a_{\gamma},\eta)\right)\cap \Omega_k^{1,\pp+1} } |dv_k^{1,p+1}-d\omega^{1,p+1}|^{p_k}_{g_k^{1,\pp+1}} d \vol_{g_k^{1,\pp+1}}  +e_1(k,L,\eta). \label{eq:excludeAB}
	\end{align}
	Here \(e_1(k,L,\eta)\) satisfies 
	\begin{align*}
		\forall L,\eta>0, \quad e_1(k,L,\eta) \xrightarrow[k \to +\infty]{} 0.
	\end{align*}
	For $\eta>0$ small enough, the balls $(B(a_{\gamma},\eta))_{\gamma\in\C}$ are either disjoint or some of them are equal. 
%
	With this notations, we now write \eqref{eq:excludeAB} as follows:
	\begin{equation}\label{eq:decompo1} 
	\begin{aligned}
		&\int_{B(0,L)\cap \Omega_k^{1,\pp+1}}|dv_k^{1,p+1}-d\omega^{1,p+1}|^{p_k}_{g_k^{1,\pp+1}}d \vol_{g_k^{1,\pp+1}}  \\
		= & \sum_{b\in\mathcal{S}_\pp}\int_{B(b,\eta)\cap \Omega_k^{1,\pp+1} }  |dv_k^{1,p+1}-d\omega^{1,p+1}|^{p_k}_{g_k^{1,\pp+1}} d \vol_{g_k^{1,\pp+1}}  +e_1(k,L,\eta).
	\end{aligned}
	\end{equation}
	We now fix $b\in\mathcal{S}_\pp$ and restrict ourselves to a given ball $B(b,\eta)$. For any $K>0$, we consider the following decomposition of the domain of integration of the right-hand side of \eqref{eq:decompo1}:
	\begin{align}
		B(b,\eta) = &\ \left( B(b,\eta)\setminus \bigcup_{\gamma\in\C_b} B\left( \frac{a_k^{1,\gamma}-a_k^{1,\pp+1}}{\lambda_k^{1,\pp+1}},K\frac{\lambda_k^{1,\gamma}}{\lambda_k^{1,\pp+1}} \right)  \right) \label{eq:domain1} \\
		 & \cup \bigcup_{\gamma\in\C_b} B\left( \frac{a_k^{1,\gamma}-a_k^{1,\pp+1}}{\lambda_k^{1,\pp+1}},K\frac{\lambda_k^{1,\gamma}}{\lambda_k^{1,\pp+1}} \right)  \label{eq:domain2}.
	\end{align}
	We first show that the integral on the domain \eqref{eq:domain2} converges to zero. Indeed, for each $\gamma\in \C_b$, we have by change of variables
	\begin{align*}
		& \int_{  B\left( \frac{a_k^{1,\gamma}-a_k^{1,\pp+1}}{\lambda_k^{1,\pp+1}},K\frac{\lambda_k^{1,\gamma}}{\lambda_k^{1,\pp+1}} \right)  \cap \Omega_k^{1,\pp+1} }|dv_k^{1,p+1}-d\omega^{1,p+1}|^{p_k}_{g_k^{1,\pp+1}} d \vol_{g_k^{1,\pp+1}} \\
		= & \left(\frac{\lambda_k^{1,\pp+1}}{\lambda_k^{1,\gamma}}\right)^{p_k-n} \int_{  B\left(0,K \right)  \cap \Omega_k^{1,\gamma} }\left|d \left[w_k^{1,\pp+1}(a_k^{1,\gamma} + \lambda_k^{1,\gamma}\cdot)\right] \right|^{p_k}_{g_k^{1,\gamma}} d \vol_{g_k^{1,\gamma}}.
	\end{align*}
	Since $1\leq \gamma\leq \pp$, we deduce from \eqref{eq:HR1p_2}, \ref{HR4p} and \Cref{lem:decoupling} that 
	\begin{align*}
		& \int_{  B\left( \frac{a_k^{1,\gamma}-a_k^{1,\pp+1}}{\lambda_k^{1,\pp+1}},K\frac{\lambda_k^{1,\gamma}}{\lambda_k^{1,\pp+1}} \right)  \cap \Omega_k^{1,\pp+1} }|dv_k^{1,p+1}-d\omega^{1,p+1}|^{p_k}_{g_k^{1,\pp+1}} d \vol_{g_k^{1,\pp+1}} \\
		= & \left(\frac{\lambda_k^{1,\pp+1}}{\lambda_k^{1,\gamma}}\right)^{p_k-n} \int_{  B\left(0,K \right)  \cap \Omega_k^{1,\gamma} }\left|d \left[\sum_{\ss=\gamma+1}^{\pp+1} \omega^{1,\ss}\left(\frac{a_k^{1,\gamma}-a_k^{1,\ss} + \lambda_k^{1,\gamma}\cdot}{\lambda_k^{1,\ss}}\right) \right] \right|^{p_k}_{g_k^{1,\gamma}} d \vol_{g_k^{1,\gamma}} +o_k(1)\\
		\leq &\ C \sum_{\ss=\gamma+1}^{\pp+1}\int_{  B\left( \frac{a_k^{1,\gamma} - a_k^{1,\ss}}{\lambda_k^{1,\ss}} ,K \frac{\lambda_k^{1,\gamma}}{\lambda_k^{1,\ss}} \right)  \cap \Omega_k^{1,\ss} }\left|d  \omega^{1,\ss}\right|^{p_k}_{g_k^{1,\ss}} d \vol_{g_k^{1,\ss}} + o_k(1).
	\end{align*}
	Given $\ss\in\{\gamma+1,\ldots,\pp+1\}$, we distinguish the following cases:
	\begin{enumerate}[label=\alph*)]
		\item If \(\lambda_k^{1,\ss}=o_k(\lambda_k^{1,\gamma})\), then by HR-\(2)_{\pp+1}\) we find that \( \frac{|a_k^{1,\gamma}-a_k^{1,\ss}|}{\lambda_k^{1,\gamma}}\to +\infty\) and hence the characteristic function of the ball \(B\left( \frac{ a_k^{1,\gamma}-a_k^{1,\ss} }{\lambda_k^{1,\ss}}, K \frac{ \lambda_k^{1,\gamma} }{ \lambda_k^{1,\ss} } \right)\) goes to zero.
		
		\item If \(\lambda_k^{1,\gamma}= o_k(\lambda_k^{1,\ss})\), then again the characteristic function of the ball \(B\left( \frac{ a_k^{1,\gamma}-a_k^{1,\ss} }{\lambda_k^{1,\ss}}, K \frac{ \lambda_k^{1,\gamma} }{ \lambda_k^{1,\ss} } \right)\) goes to zero.
		
		\item If $\frac{\lambda_k^{1,\gamma}}{\lambda_k^{1,\ss}}+\frac{\lambda_k^{1,\ss}}{\lambda_k^{1,\gamma}}\leq R$ for some $R>0$, then it holds \(\frac{|a_k^{1,\gamma}-a_k^{1,\ss}|}{\lambda_k^{1,\gamma}+\lambda_k^{1,\ss}}\xrightarrow[k\to+\infty]{} +\infty\) by HR-$2)_{p+1}$. Hence the characteristic function of the ball \(B\left( \frac{ a_k^{1,\gamma}-a_k^{1,\ss} }{\lambda_k^{1,\ss}}, K \frac{ \lambda_k^{1,\gamma} }{ \lambda_k^{1,\ss} } \right)\) goes to zero.
	\end{enumerate}
	We obtain for any $K>0$, 
	\begin{align*}
		\int_{ \left(  \bigcup_{\gamma\in\C_b} B\left( \frac{a_k^{1,\gamma}-a_k^{1,\pp+1}}{\lambda_k^{1,\pp+1}},K\frac{\lambda_k^{1,\gamma}}{\lambda_k^{1,\pp+1}} \right)  \right) \cap \Omega_k^{1,\pp+1} }|dv_k^{1,p+1}-d\omega^{1,p+1}|^{p_k}_{g_k^{1,\pp+1}} d \vol_{g_k^{1,\pp+1}} \xrightarrow[k\to+\infty]{} 0.
	\end{align*}
	Therefore, the main term in \eqref{eq:decompo1} comes from the domain \eqref{eq:domain1}:
	\begin{align*}
		& \int_{ B(b,\eta)\cap \Omega_k^{1,\pp+1} }|dv_k^{1,p+1}-d\omega^{1,p+1}|^{p_k}_{g_k^{1,\pp+1}} d \vol_{g_k^{1,\pp+1}}  \\
		= & \int_{ \left( B(b,\eta)\setminus \bigcup_{\gamma\in\C_b} B\left( \frac{a_k^{1,\gamma}-a_k^{1,\pp+1}}{\lambda_k^{1,\pp+1}},K\frac{\lambda_k^{1,\gamma}}{\lambda_k^{1,\pp+1}} \right)  \right) \cap \Omega_k^{1,\pp+1} }|dv_k^{1,p+1}-d\omega^{1,p+1}|^{p_k}_{g_k^{1,\pp+1}} d \vol_{g_k^{1,\pp+1}} \\
		& \quad + e_2(k,\eta,L,K),
	\end{align*}
	where the remainder $e_2(k,\eta,L,K)$ satisfies
	\begin{align*}
		\forall \eta,L,K,\quad e_2(k,\eta,L,K)\xrightarrow[k\to+\infty]{} 0.
	\end{align*}
	We now combine \Cref{cl:AB_neg} and \Cref{claim:C_outside_points} to obtain
	\begin{align*}
		& \int_{ B(b,\eta)\cap \Omega_k^{1,\pp+1} }|dv_k^{1,p+1}-d\omega^{1,p+1}|^{p_k}_{g_k^{1,\pp+1}} d \vol_{g_k^{1,\pp+1}}  \\
		= & \int_{ \left( B(b,\eta)\setminus \bigcup_{\gamma\in\C_b} B\left( \frac{a_k^{1,\gamma}-a_k^{1,\pp+1}}{\lambda_k^{1,\pp+1}},K\frac{\lambda_k^{1,\gamma}}{\lambda_k^{1,\pp+1}} \right)  \right) \cap \Omega_k^{1,\pp+1} }\left| d\left[ u_k(a_k^{1,\pp+1}+\lambda_k^{1,\pp+1}\cdot)-\omega^{1,p+1}\right] \right|^{p_k}_{g_k^{1,\pp+1}} d \vol_{g_k^{1,\pp+1}} \\
		& \quad + e_2(k,\eta,L,K).
	\end{align*}
	Recall that \(	\U_{k,\eta,K}\) is defined in \eqref{eq:U_ketaK}.
	Using \Cref{lem:decoupling} for the limit $\eta\to 0$, we obtain
	\begin{align*}
		& \limsup_{k\to+\infty} \int_{B(0,L)\cap \Omega_k^{1,\pp+1}} \left| dv_k^{1,\pp+1} - d\omega^{1,\pp+1} \right|^{p_k}_{g_k^{1,\pp+1}}\, d\vol_{g_k^{1,\pp+1}}  \\
		= & \limsup_{\eta\to 0} \limsup_{k\to+\infty} \int_{ \U_{k,\eta,K} }\left| d\left[ u_k(a_k^{1,\pp+1}+\lambda_k^{1,\pp+1}\cdot)\right] \right|^{p_k}_{g_k^{1,\pp+1}} d \vol_{g_k^{1,\pp+1}}.  
	\end{align*}
	By construction of $\U_{k,\eta,K}$, we have that for any $1\leq \ss\leq \pp+1$, it holds
	\begin{align*}
		 0 = \limsup_{K\to +\infty} \limsup_{\eta\to 0} \limsup_{k\to+\infty} \int_{ \U_{k,\eta,K} } 
		 \left| d\left[ \omega^{1,\ss}\left(\frac{a_k^{1,\pp+1}-a_k^{1,\ss} + \lambda_k^{1,\pp+1}\cdot}{\lambda_k^{1,\ss}}\right)\right] \right|^{p_k}_{g_k^{1,\pp+1}} d \vol_{g_k^{1,\pp+1}}.
	\end{align*}
	This is a consequence of \Cref{cl:AB_neg} and \Cref{claim:C_outside_points}.
\end{proof}

We now prove that we can restrict the domain of integration in \eqref{eq:estimate_strong_conv} to be a finite union of degenerating annuli. We point out here that a crucial ingredient is that we detect bubbles in a precise order thanks to the concentration function, namely we extract the most concentrated bubbles first. Another important idea is that the energy on annuli of fixed conformal class always vanishes in the limit. Thus, only degenerating annuli might contribute to the loss of energy. 

\begin{claim}\label{cl:estimate_strong_conv2}
	Let $\mathcal{S}_\pp \coloneqq  \{a_{\gamma}\}_{\gamma\in \C}$, $\eta\in(0,1)$ and $K>1$. Given $b\in\mathcal{S}_\pp$, there exist $J_b\in\N$ and a finite number of disjoint annuli $A_{k,\eta,K}^{b,j} = B(b,r_{k,\eta,K}^{b,j})\setminus B(b,s_{k,\eta,K}^{b,j})$ for $j\in\{1,\ldots,J_b\}$ such that
	\begin{itemize}
		\item for each $j\in\{1,\ldots,J_b\}$, it holds $s_{k,\eta,K}^{b,j}>0$, $s_{k,\eta,K}^{b,j} = o_k(r_{k,\eta,K}^{b,j})$ and $\displaystyle{\limsup_{k\to +\infty}}\ (s_{k,\eta,K}^{b,j})^{n-p_k} <+\infty$,
		\item for each $j\in\{1,\ldots,J_b\}$, it holds 
		\begin{align}\label{eq:check_smallness}
			\int_{A^{b,j}_{k,\eta,K}\cap \Omega_k^{1,\pp+1}} \left| d\left[u_k\left(a_k^{1,\pp+1} + \lambda_k^{1,\pp+1}\cdot\right)\right] \right|^n_{g_k^{1,\pp+1}}\, d\vol_{g_k^{1,\pp+1}} \leq \eps_{\star}^n .
		\end{align}
		\item The following boundary conditions are satisfied: 
		\begin{equation}\label{eq:boundary_annulus}
		\begin{aligned}
			& \limsup_{K\to+\infty}\ \limsup_{\eta\to 0}\ \limsup_{k\to +\infty} \Bigg( \left\|d[u_k(a_k^{1,\pp+1} + \lambda_k^{1,\pp+1}\cdot)] \right\|_{L^{p_k}\big( B(b,r_{k,\eta,K}^{b,j})\setminus B(b,r_{k,\eta,K}^{b,j}/2) \cap \Omega_k^{1,\pp+1}\big)} \\
			&\quad  + \left\|d[u_k(a_k^{1,\pp+1} + \lambda_k^{1,\pp+1}\cdot)] \right\|_{L^{p_k}\big( B(b,2s_{k,\eta,K}^{b,j})\setminus B(b,s_{k,\eta,K}^{b,j})\cap \Omega_k^{1,\pp+1} \big)} \Bigg) = 0.
		\end{aligned}
		\end{equation}
		\item the following equality is valid for any $L>1$ such that $\mathcal{S}_\pp\subset B(0,L)$:
		\begin{align}
			& \limsup_{k\to+\infty} \int_{B(0,L)\cap \Omega_k^{1,\pp+1}}|dv_k^{1,\pp+1}-d\omega^{1,\pp+1}|^{p_k}_{g_k^{1,\pp+1}} d \vol_{g_k^{1,\pp+1}} \label{eq:estimate_strong_conv2}  \\
			= &\  \limsup_{\eta\to 0}\limsup_{k\to +\infty}  \int_{ \Omega_k^{1,\pp+1} \cap \bigcup_{b\in\mathcal{S}_\pp}\bigcup_{j=1}^{J_b} A_{k,\eta,K}^{b,j}} \left|d\left[ u_k\left(a_k^{1,\pp+1} + \lambda_k^{1,\pp+1}\cdot\right) \right] \right|^{p_k}_{g_k^{1,\pp+1}} d \vol_{g_k^{1,\pp+1}}  . \nonumber
		\end{align}
	\end{itemize}
\end{claim}

\begin{proof}
	We fix $b\in\mathcal{S}_\pp$ and $\eta\in(0,1), K>0$. 
	For each $\gamma\in\C_b$, the ball $B\left(\frac{a_k^{1,\gamma}-a_k^{1,\pp+1}}{\lambda_k^{1,\pp+1}},K\frac{\lambda_k^{1,\gamma}}{\lambda_k^{1,\pp+1}}\right)$ is contained in the following domain, which is either a ball or an annulus centred at $b$, depending on whether $b$ lies in this ball or not:
	\begin{align*}
		A^{\gamma}_{k,K} = \bigcup_{R\in SO(n)} \left( R\left[ B\left(\frac{a_k^{1,\gamma}-a_k^{1,\pp+1}}{\lambda_k^{1,\pp+1}},K\frac{\lambda_k^{1,\gamma}}{\lambda_k^{1,\pp+1}}\right) -b \right]+b \right).
	\end{align*}
	Since $\frac{\lambda_k^{1,\gamma}}{\lambda_k^{1,\pp+1}}\to 0$ as $k\to +\infty$, the characteristic function of $A^{\gamma}_{k,K} $ converges to 0 almost everywhere and we obtain the following limit:
	\begin{align}\label{eq:no_energy_curve}
		\limsup_{K\to +\infty}\ \limsup_{\eta\to 0}\ \limsup_{k\to +\infty} \int_{\Omega_k^{1,\pp+1} \cap A^{\gamma}_{k,K}\cap\, \U_{k,\eta,K}}  \left| d\left[u_k\left(a_k^{1,\pp+1} + \lambda_k^{1,\pp+1}\cdot\right)\right] \right|^{p_k}_{g_k^{1,\pp+1}}\, d\vol_{g_k^{1,\pp+1}} \nonumber\\
		=0.
	\end{align}
	Indeed, each map $u_k\left(a_k^{1,\pp+1} + \lambda_k^{1,\pp+1}\cdot\right)$ is a free boundary $p_k$-harmonic map with free boundary in the sphere. Therefore, if \eqref{eq:no_energy_curve} is not true, then by rescaling by \(K\lambda_k^{1,\gamma}/\lambda_k^{1,\pp+1}\) we obtain a free boundary \(p_k\)-harmonic map on an annulus with outer radius $1$ or a ball of radius $1$. By using the \(\e\)-compactness result of Proposition \ref{pr:Convergence}, and by using a concentration function as in Step 1, we obtain the existence of a point $x_{k,\eta,K}\in A_{k,K}^{\gamma}\cap \U_{k,\eta,K}$ such that the sequence $(\tilde{u}_k)_{k\in\N} \coloneqq  \left( u_k\left(a_k^{1,\gamma} + \lambda_k^{1,\pp+1}x_{k,\eta,K} + \lambda_k^{1,\gamma}\cdot\right) \right)_{k\in\N}$ satisfies for any radius $\rho>0$:
	\begin{align*}
		\limsup_{K\to +\infty}\ \limsup_{\eta\to 0}\ \limsup_{k\to +\infty} \|d\tilde{u}_k\|_{L^{p_k}(B(0,\rho))} >0.
	\end{align*}
	As in Step 1 and Step 2 we thus deduce that  $(\tilde{u}_k)_{k\in\N}$ converges to a non-constant bubble. Thanks to \Cref{lem:prop_bubbles} combined with \eqref{eq:no_energy_bubbles}, we obtain
	\begin{align*}
		\eps_b \leq \limsup_{K\to +\infty}\ \limsup_{\eta\to 0}\ \limsup_{k\to +\infty} \|d\tilde{u}_k\|_{L^{p_k}(B(0,\rho))} \leq \limsup_{k\to +\infty} Q_k^{1,\pp}(\rho \lambda_k^{1,\gamma})^{\frac{1}{p_k}}.
	\end{align*}
	Since $\gamma\in\C$, it holds $\lambda_k^{1,\gamma}=o_k (\lambda_k^{1,\pp+1})$, leading to $Q_k^{1,\pp}(\rho \lambda_k^{1,\gamma})\leq Q_k^{1,\pp}(\lambda_k^{1,\pp+1})$. This is impossible by \eqref{eq:quanta_energy} and the definition of the sequence $(t_j)_{j\in\N}$ in Item \eqref{asump:tj}. Hence \eqref{eq:no_energy_curve} holds true. 
	
	Consequently the relation \eqref{eq:estimate_strong_conv} reduces to 
	\begin{align*}
		& \limsup_{k\to+\infty} \int_{B(0,L)\cap \Omega_k^{1,\pp+1}}|dv_k^{1,\pp+1}-d\omega^{1,\pp+1}|^{p_k}_{g_k^{1,\pp+1}} d \vol_{g_k^{1,\pp+1}}  \nonumber \\
		= &\ \limsup_{K\to +\infty}\  \limsup_{\eta\to 0}\ \limsup_{k\to +\infty}  \int_{\Omega_k^{1,\pp+1}\cap  B(b,\eta)\setminus \bigcup_{\gamma\in\C_b} A^{\gamma}_{k,K} } \left|d\left[ u_k\left(a_k^{1,\pp+1} + \lambda_k^{1,\pp+1}\cdot\right) \right] \right|^{p_k}_{g_k^{1,\pp+1}} d \vol_{g_k^{1,\pp+1}}  . 
	\end{align*}
	The set $B(b,\eta)\setminus \bigcup_{\gamma\in\C_b} A^{\gamma}_{k,K}$ is a union of disjoint annuli centred at $b$ and possibly a ball centred at $b$, which we denote $A_{k,\eta,K}^{b,1},\ldots, A_{k,\eta,K}^{b,J_b}$ and $B_{b,k}$. Moreover, thanks to \Cref{cl:estimate_strong_conv} and more specifically \eqref{eq:no_energy_bubbles}, it holds
	\begin{align}
		 \limsup_{K\to +\infty}\  \limsup_{\eta\to 0}\ \limsup_{k\to +\infty}  \int_{\Omega_k^{1,\pp+1}\cap B(b,\eta)\setminus \bigcup_{\gamma\in\C_b} A^{\gamma}_{k,K} } \left|d\left[ u_k\left(a_k^{1,\pp+1} + \lambda_k^{1,\pp+1}\cdot\right) \right] \right|^{p_k}_{g_k^{1,\pp+1}} d \vol_{g_k^{1,\pp+1}} \nonumber \\
		\leq \ \limsup_{k\to +\infty} Q_k^{1,\pp}(\eta \lambda_k^{1,\pp+1}) \leq \eps_{\star}^{p_k}. \label{eq:smallness}
	\end{align}
	On the ball $B_{b,k}$, we obtain from \Cref{pr:Convergence} that no concentration of energy can occur, i.e., we have the limit
	\begin{align*}
		\lim_{k\to+\infty} \int_{B_{b,k}\cap \Omega_k^{1,\pp+1}} \left|d\left[ u_k\left(a_k^{1,\pp+1} + \lambda_k^{1,\pp+1}\cdot\right) \right] \right|^{p_k}_{g_k^{1,\pp+1}} d \vol_{g_k^{1,\pp+1}} =0.
	\end{align*}
	We denote each annulus $A_{k,\eta,K}^{b,j} = B(b,r_{k,\eta,K}^{b,j})\setminus B(b,s_{k,\eta,K}^{b,j})$. For each $j$, there are three possibilities. 
	
	\underline{Case 1:} It holds $\frac{s_{k,\eta,K}^{b,j}}{r_{k,\eta,K}^{b,j}}\to 1$ as $k\to +\infty$. 
	
	In this case, the characteristic function of $A_{k,\eta,K}^{b,j}$ converges almost everywhere to $0$. By reasoning as what was done for \eqref{eq:no_energy_curve} we obtain
	\begin{align*}
		\limsup_{K\to +\infty}\ \limsup_{\eta\to 0}\ \limsup_{k\to +\infty} \int_{A_{k,\eta,K}^{b,j}\cap \Omega_k^{1,\pp+1}}  \left| d\left[u_k\left(a_k^{1,\pp+1} + \lambda_k^{1,\pp+1}\cdot\right)\right] \right|^{p_k}_{g_k^{1,\pp+1}}\, d\vol_{g_k^{1,\pp+1}} =0.
	\end{align*}
	
	\underline{Case 2:} There exist $0< c_1<c_2<1$ such that for all $k\in\N$, it holds $c_1< \frac{s_{k,\eta,K}^{b,j}}{r_{k,\eta,K}^{b,j}}<c_2$. 
	
	The condition \eqref{eq:smallness} implies that the sequence $\left(u_k \left(a_k^{1,\pp+1} +\lambda_k^{1,\pp+1}b + \lambda_k^{1,\pp+1} r_{k,\eta,K}^{b,j}\cdot \right)\right)_{k\in\N}$ has no point of energy concentration in the annulus $B(0,1)\setminus B(0,s_{k,\eta,K}^{b,j}/r_{k,\eta,K}^{b,j})$. Thanks to \Cref{pr:Convergence}, the sequence $\left(u_k \left(a_k^{1,\pp+1} +\lambda_k^{1,\pp+1}b + \lambda_k^{1,\pp+1} r_{k,\eta,K}^{b,j}\cdot \right)\right)_{k\in\N}$ converges in the $\C^1_{\loc}$-topology to some limit. If the limit is not constant, then it is a bubble which must be one of the $\omega^{1,1},\ldots,\omega^{1,\pp+1}$ by choice of $\lambda_k^{1,\pp+1}$. Thanks to \Cref{cl:estimate_strong_conv}, we obtain 
	\begin{align*}
		\limsup_{K\to+\infty}\ \limsup_{\eta\to 0}\ \limsup_{k\to+\infty} \int_{A_{k,\eta,K}^{b,j}\cap \Omega_k^{1,\pp+1}} \left| d\left[u_k\left(a_k^{1,\pp+1} + \lambda_k^{1,\pp+1}\cdot\right)\right] \right|^{p_k}_{g_k^{1,\pp+1}}\, d\vol_{g_k^{1,\pp+1}} =0.
	\end{align*}
	
	\underline{Case 3:} It holds $\frac{s_{k,\eta,K}^{b,j}}{r_{k,\eta,K}^{b,j}}\to 0$ as $k\to +\infty$. 
	
	This case leads to one of the degenerating annuli appearing in \Cref{cl:estimate_strong_conv2}. The boundary condition is a consequence of the above Case 2. Moreover, by construction of $A_{k,\eta,K}^{b,j}$, there exists $\gamma\in\C_b$ such that $s_{k,\eta,K}^{b,j}\geq K\lambda_k^{1,\gamma}/\lambda_k^{1,\pp+1}$. Thus, it holds 
	\begin{align*}
		\limsup_{k\to +\infty}\ (s_{k,\eta,K}^{b,j})^{n-p_k} \leq \limsup_{k\to+\infty}\ (\lambda_k^{1,\gamma})^{n-p_k} \leq \frac{M}{\eps_b^n}.
	\end{align*}
\end{proof}

We now prove that in this specific region, the energy vanishes.

\begin{claim}\label{cl:energy_strong_conv}
	For every $b\in\mathcal{S}_\pp$ and every $j\in\{1,\ldots,J_b\}$, it holds
	\begin{align*}
		\limsup_{K\to+\infty}\ \limsup_{\eta\to 0}\ \limsup_{k\to +\infty} \int_{ A_{k,\eta,K}^{b,j} } \left|d \left[ u_k\left(a_k^{1,\pp+1}+\lambda_k^{1,\pp+1}\cdot\right) \right] \right|^{p_k}_{g_k^{1,\pp+1}} d \vol_{g_k^{1,\pp+1}} = 0.
	\end{align*}
\end{claim}

   In order to prove that claim we will use Lemma \ref{lem:comparison_2} to compare the energy of \(u_k\) in an annulus with the energy of its \(p_k\)-harmonic extension in the same annulus. Then we will show that the energy of its \(p_k\)-harmonic extension converges to zero. By construction, the smallness assumption needed in \Cref{lem:comparison_2} is satisfied.
 
 Let $b\in\mathcal{S}_\pp$ and $j\in\{1,\ldots,J_b\}$. We define \(V_k\in W^{1,p_k}\left( A_{k,\eta,K}^{b,j}\cap \R^n_+;\R^d \right)\) to be the unconstrained \(p_k\)-harmonic extension of \(u_k\), i.e., the solution to the boundary value problem
\begin{equation}\label{eq:def_V_k}
\left\{
\begin{array}{rcll}
	\vspace{0.3em} -\Delta_{p_k,g_k^{1,\pp+1}} V_k &=& 0 & \text{ in }A_{k,\eta,K}^{b,j}\cap \R^n_+, \\
	\vspace{0.3em} V_k &=& u_k\left(a_k^{1,\pp+1} + \lambda_k^{1,\pp+1}\cdot\right) & \text{ on } \p A_{k,\eta,K}^{b,j}\cap B(0,r_1)^+ \\
	\vspace{0.3em} |dV_k|^{p_k-2}_{g_k^{1,\pp+1}}\p_\nu V_k & =&0 & \text{ on } \p \R^n_+ \cap A_{k,\eta,K}^{b,j}.
\end{array}
\right.
\end{equation}
We denote
\begin{align*}
	& \bar{u}_{k,1} = \mvint_{B(b,2s_{k,\eta,K}^{b,j})\setminus B(b,s_{k,\eta,K}^{b,j})\cap \Omega_k^{1,\pp+1} } u_k\left(a_k^{1,\pp+1} + \lambda_k^{1,\pp+1}x\right)\, d\vol_{g_k^{1,\pp+1}}(x),\\
	& \bar{u}_{k,2} = \mvint_{B(b,r_{k,\eta,K}^{b,j})\setminus B(b,r_{k,\eta,K}^{b,j}/2)\cap \Omega_k^{1,\pp+1} } u_k\left(a_k^{1,\pp+1} + \lambda_k^{1,\pp+1}x\right)\, d\vol_{g_k^{1,\pp+1}}(x).
\end{align*}
 We also define \(F_k\) and \(G_k\) as 
 \begin{equation}\label{eq:def_F_k}
\left\{
\begin{array}{rcll}
	\vspace{0.3em} -\Delta_{p_k,g_k^{1,\pp+1}} F_k &=& 0 & \text{ in } A_{k,\eta,K}^{b,j}\cap \R^n_+ \\
	\vspace{0.3em} F_k &=& u_k\left(a_k^{1,\pp+1} + \lambda_k^{1,\pp+1}\cdot\right)-\bar{u}_{k,1} & \text{ on } \p B(b,s_{k,\eta,K}^{b,j})\cap B(0,r_1)^+ \\
	\vspace{0.3em} F_k &=& u_k\left(a_k^{1,\pp+1} + \lambda_k^{1,\pp+1}\cdot\right)-\bar{u}_{k,2} & \text{ on } \p B(b, r_{k,\eta,K}^{b,j})\cap B(0,r_1)^+ \\
	\vspace{0.3em} |dF_k|^{p_k-2}\p_\nu F_k & =&0 & \text{ on } \p \R^n_+ \cap A_{k,\eta,K}^{b,j},
\end{array}
\right.
\end{equation}
\begin{equation}\label{eq:def_G_k}
\left\{
\begin{array}{rcll}
	\vspace{0.3em} -\Delta_{p_k,g_k^{1,\pp+1}} G_k &=& 0 & \text{ in }  A_{k,\eta,K}^{b,j}\cap \R^n_+ \\
	\vspace{0.3em} G_k &=&\bar{u}_{k,1}  & \text{ on } \p B(b,s_{k,\eta,K}^{b,j})\cap B(0,r_1)^+ \\
	\vspace{0.3em} G_k &=& \bar{u}_{k,2} & \text{ on } \p  B(b, r_{k,\eta,K}^{b,j})\cap B(0,r_1)^+\\
	\vspace{0.3em} |dG_k|^{p_k-2}_{g_k^{1,\pp+1}}\p_\nu G_k & =&0 & \text{ on } \p \R^n_+ \cap A_{k,\eta,K}^{b,j}.
\end{array}
\right.
\end{equation}
 Since \(V_k\) minimizes the \(p_k\) energy with respect to its own boundary value, we have that 
 \begin{equation}\label{eq:estimate_by_Fk_Gk}
 \begin{aligned}
  	 & \int_{ A_{k,\eta,K}^{b,j}\cap \R^n_+}|dV_k|_{g_k^{1,\pp+1}}^{p_k}d \vol_{g_k^{1,\pp+1}} \\
  	 \leq & \int_{A_{k,\eta,K}^{b,j}\cap \R^n_+}| d (F_k+G_k)|_{g_k^{1,\pp+1}}^{p_k} d \vol_{g_k^{1,\pp+1}}  \\
   \leq &\ 2^{p_k} \Bigg(\int_{A_{k,\eta,K}^{b,j}\cap \R^n_+}|dF_k|_{g_k^{1,\pp+1}}^{p_k} d \vol_{g_k^{1,\pp+1}} + \int_{ A_{k,\eta,K}^{b,j}\cap \R^n_+}|dG_k|_{g_k^{1,\pp+1}}^{p_k} d \vol_{g_k^{1,\pp+1}} \Bigg).
 \end{aligned}
 \end{equation}

 The \(p_k\)-energy of \(G_k\) is estimated with the help of \Cref{lem:comparison_cst} and the pointwise inequality $|u_k|\leq 1$. We find, using \Cref{rk:limit}, that
 \begin{align*}
 	 \int_{ A_{k,\eta,K}^{b,j} \cap \R^n_+}|dG_k|_{g_k^{1,\pp+1}}^{p_k} d \vol_{g_k^{1,\pp+1}} &
  	\leq \frac{C}{ \left|\log\left(\frac{r_{k,\eta,K}^{b,j}}{s_{k,\eta,K}^{b,j}} \right) \right|^{p_k}}\frac{ (r_{k,\eta,K}^{b,j})^{n-p_k}-(s_{k,\eta,K}^{b,j})^{n-p_k}}{n-p_k}\\
  	& \leq \frac{C (s_{k,\eta,K}^{b,j})^{n-p_k}}{\left|\log\left(\frac{r_{k,\eta,K}^{b,j}}{s_{k,\eta,K}^{b,j}} \right) \right|^{p_k-1}}.
 \end{align*}
By \Cref{cl:estimate_strong_conv2}, it holds $\displaystyle{\limsup_{k\to +\infty}}\ (s_{k,\eta,K}^{b,j})^{n-p_k}\leq C$. We obtain for any fixed \(\eta,K>0\) 
 \begin{equation}\label{eq:conv_G_k}
 	\lim_{k\to +\infty}  \int_{ A^{b,j}_{k,\eta,K}\cap \R^n_+}|dG_k|_{g_k^{1,\pp+1}}^{p_k} d \vol_{g_k^{1,\pp+1}} =0.
 \end{equation}
 
 Estimating the \(p_k\)-energy of \(F_k\) is the object of the following:
 \begin{claim}
 We have that 
 \begin{equation}\label{eq:conv_F_k}
  	\lim_{K \to +\infty} \lim_{\eta \to 0} \lim_{k\to +\infty}  \int_{ A^{b,j}_{k,\eta,K}\cap \R^n_+}|dF_k|_{g_k^{1,\pp+1}}^{p_k} d \vol_{g_k^{1,\pp+1}} =0.
 \end{equation}
 \end{claim}
 \begin{proof}
 The map $F_k$, solution to \eqref{eq:def_F_k} is the unique solution to
\begin{align*} 
 \inf \left\{\|d h\|^{p_k}_{L^{p_k}\left( A^{b,j}_{k,\eta,K}\cap \R^n_+\right)} \colon
 \begin{array}{l}
 	\vspace{0.3mm} h\in W^{1,p_k}\left(A^{b,j}_{k,\eta,K}\cap \R^n_+;\R^d\right)\\
 	\vspace{0.3mm} h=u_k-\bar{u}_{k,1} \text{ on }  \p B(b,s^{b,j}_{k,\eta,K})\cap B(0,r_1)^+\\
 	\vspace{0.3mm} h=u_k-\bar{u}_{k,2} \text{ on } \p  B(b,r^{b,j}_{k,\eta,K})\cap B(0,r_1)^+
 \end{array}
 \right\}.
 \end{align*}
  Hence we can estimate \(\|dF_k\|^{p_k}_{L^{p_k}(A^{b,j}_{k,\eta,K}\cap \R^n_+)}\) by comparing \(F_k\) with any other extension of $u_k-\bar{u}_{k,2}$ and $u_k-\bar{u}_{k,1}$ to the annulus. Let \(\chi\in \C^\infty_c(\R;[0,1])\) be such that \(\chi(x)=0\) for all \(x\leq 1\) and \(\chi(x)=1\) for all \(x\geq 2\). We compare \(F_k\) to a function defined for \(x\in A^{b,j}_{k,\eta,K}\cap \R^n_+\) by
\begin{align*} 
  H_k(x) \coloneqq  &\ \chi \left(\frac{|x-b|}{ r^{b,j}_{k,\eta,K} } \right) \left(u_k\left(a_k^{1,\pp+1} + \lambda_k^{1,\pp+1}x\right)-\bar{u}_{k,2} \right)\\
   & + \chi \left(\frac{|x-b|}{s^{b,j}_{k,\eta,K}} \right) \left( u_k\left(a_k^{1,\pp+1} + \lambda_k^{1,\pp+1}x\right)-\bar{u}_{k,1} \right).
  \end{align*}
  By using the triangle inequality we have
  \begin{align*}
  	& \|dH_k\|_{L^{p_k}\left(A^{b,j}_{k,\eta,K}\cap \R^n_+ \right) }\\
  	\leq &\ \left\| d[u_k\left(a_k^{1,\pp+1} + \lambda_k^{1,\pp+1}\cdot\right)] \right\|_{L^{p_k}\left(B(b,r^{b,j}_{k,\eta,K})^+\setminus B(b,r^{b,j}_{k,\eta,K}/2)\right) } \\
  	&  +\left\|d [u_k\left(a_k^{1,\pp+1} + \lambda_k^{1,\pp+1}\cdot\right)] \right\|_{L^{p_k}\left(B(b,2s^{b,j}_{k,\eta,K})^+\setminus B(b,s^{b,j}_{k,\eta,K})\right)} \\
   & + \frac{C}{s^{b,j}_{k,\eta,K}}\left\|u_k\left(a_k^{1,\pp+1} + \lambda_k^{1,\pp+1}\cdot\right)-\bar{u}_{k,1} \right\|_{L^{p_k}\left(B(b,2s^{b,j}_{k,\eta,K})^+\setminus B(b,s^{b,j}_{k,\eta,K})\right)}  \\
   & +\frac{C}{r^{b,j}_{k,\eta,K}}\left\|u_k\left(a_k^{1,\pp+1} + \lambda_k^{1,\pp+1}\cdot\right)-\bar{u}_{k,2} \right\|_{L^{p_k}\left(B(b,r^{b,j}_{k,\eta,K})^+\setminus B(b,r^{b,j}_{k,\eta,K}/2)\right)}.
  \end{align*} 
  By Poincar\'e inequality, we obtain 
  \begin{align*} 
  	 \|dH_k\|_{L^{p_k}\left(A^{b,j}_{k,\eta,K}\cap \R^n_+ \right) }
   \leq &\ C\left\| d[u_k\left(a_k^{1,\pp+1} + \lambda_k^{1,\pp+1}\cdot\right)] \right\|_{L^{p_k}\left(B(b,r^{b,j}_{k,\eta,K})^+\setminus B(b,r^{b,j}_{k,\eta,K}/2)\right) } \\
   & +C\left\|d [u_k\left(a_k^{1,\pp+1} + \lambda_k^{1,\pp+1}\cdot\right)] \right\|_{L^{p_k}\left(B(b,2s^{b,j}_{k,\eta,K})^+\setminus B(b,s^{b,j}_{k,\eta,K})\right)}.
  \end{align*}
  We conclude the proof of the claim thanks to \eqref{eq:boundary_annulus}.
 \end{proof}
 
We are now ready to prove \Cref{cl:energy_strong_conv}.
\begin{proof}[Proof of \Cref{cl:energy_strong_conv}]
Recall that \(V_k\) is defined by \eqref{eq:def_V_k}. Thanks to \eqref{eq:check_smallness}, we can apply \Cref{lem:comparison_2} to obtain
\begin{align*}
	\int_{A^{b,j}_{k,\eta,K}\cap \Omega_k^{1,\pp+1}}\left|d\left[u_k\left(a_k^{1,\pp+1}+\lambda_k^{1,\pp+1}\cdot\right) \right]\right|_{g_k^{1,\pp+1}}^{p_k} \, d \vol_{g_k^{1,\pp+1}} 
	\leq  C\int_{A^{b,j}_{k,\eta,K}\cap \R^n_+}|dV_k|_{g_k^{1,\pp+1}}^{p_k} \, d \vol_{g_k^{1,\pp+1}}.
\end{align*}
We conclude the proof of \Cref{cl:energy_strong_conv} by using \eqref{eq:estimate_by_Fk_Gk} and \eqref{eq:conv_G_k}-\eqref{eq:conv_F_k} which show that 
\begin{equation*}
\limsup_{K \to \infty}\ \limsup_{\eta \to 0}\ \limsup_{k \to +\infty} \int_{A^{b,j}_{k,\eta,K}\cap \R^n_+}|dV_k|_{g_k^{1,\pp+1}}^{p_k} \, d \vol_{g_k^{1,\pp+1}}=0.
\end{equation*}
\end{proof}
As a consequence we obtain  the following:
\begin{claim}
	For any \(L>0\) we have that 
\begin{equation}\label{eq:stronger_conv}
\int_{B(0,L)\cap \Omega_k^{1,\pp+1}}\left|dw_k^{1,\pp}(a_k^{1,\pp+1}+\lambda_k^{1,\pp+1}\cdot)-\omega^{1,\pp+1}\right|^{p_k}_{g_k^{1,\pp+1}} \, d \vol_{g_k^{1,\pp+1}} \xrightarrow[k\to +\infty]{} 0.
\end{equation}	
	In particular, the sequence $(v_k^{1,\pp+1})_{k\in\N}$ converges to $\omega^{1,\pp+1}$ strongly in $W^{1,n}_{\loc}(\overline{\R^n_+}; \R^d)$ and \(\omega^{1,\pp+1}\) is a non-trivial bubble. Furthermore \(\e_{\pp}\geq \e_b\) where \(\e_{\pp}\) is defined in \eqref{def:epsp} and \(\e_b\) appears in \Cref{lem=quantum_energy}, and
	\begin{equation*}
		\int_{B\left( a_k^{1,\pp+1},\lambda_k^{1,\pp+1} \right)^+}| dw_k^{1,\pp}|^{p_k} = t_\pp \eps_{\star}^{p_k}.
	\end{equation*} 	 
\end{claim}
\begin{proof}
Indeed, by  Claim \ref{cl:estimate_strong_conv} and Claim \ref{cl:energy_strong_conv}   we find that, for any \(L>0\), 
\begin{align*}
\limsup_{k \to +\infty} \int_{B(0,L)\cap \Omega_k^{1,\pp}}|d v_k^{1,\pp+1}-\omega^{1,\pp+1}|_{g_k^{1,\pp+1}}^{p_k} d\vol_{g_k^{1,\pp+1}} =0.
\end{align*}
 Thus, with the help of H\"older inequality, we obtain the local strong convergence of \(v_k^{1,\pp+1}\) to \(\omega^{1,\pp+1}\) in \(W^{1,n}_{\text{loc}}(\R^n;\R^d)\) (\ref{Case_1_p+1}) ) or in \(W^{1,n}(\overline{\R^n_+};\R^d)\) (\ref{Case_2_p+1}).
In particular, thanks to \eqref{eq:quanta_energy}, we find that \(\int_{B(0,1)}|d \omega^{1,\pp+1}|_{g(x_1)}^n d \vol_{g(x_1)}=t_\pp \min\{\e_{\star}^n,\e_\pp^n\} \). Thus \(\omega^{1,\pp+1}\) cannot be constant and as in Claim \ref{claim:finite_energy} we cannot be in the case \ref{Case_1_p+1}. We are necessarily in \ref{Case_2_p+1} and \(\omega^{1,\pp+1}\) is a non-trivial bubble. 
Now, for any \(\rho,R>0\), for \(k\) large enough we have
\begin{align*}
\int_{B(0,\rho)^+}|dw_k^{1,\pp}|_g^{p_k} \, d \vol_{g}&\geq \int_{B(a_k^{1,\pp+1},R\lambda_k^{1,\pp+1})^+}|dw_k^{1,\pp}|_g^{p_k} \, d \vol_g \\
 & \geq (\lambda_k^{1,\pp+1})^{n-p_k}\int_{B(0,R)^+} | d [w_k^{1,\pp}(a_k^{1,\pp+1}+\lambda_k^{1,\pp+1}\cdot)]|_{g_k^{1,\pp+1}}^{p_k} \, d \vol_{g_k^{1,\pp+1}}.
\end{align*}
Thus, from the weak convergence of \(w_k^{1,\pp}(a_k^{1,\pp+1}+\lambda_k^{1,\pp+1}\cdot)\) towards \(\omega^{1,\pp+1}\) in \(W^{1,n}_{\text{loc}}(\overline{\R^n_+};\R^d)\), from the third item in Claim \ref{claim:strong_conv_u_kp+1} and from H\"older inequality we infer that 
\[ \liminf_{k \to +\infty} \int_{B(0,\rho)^+}|dw_k^{1,\pp}|_g^{p_k} \, d \vol_g \geq \int_{B(0,R)^+}|d\omega^{1,\pp+1}|_{g(x_1)}^n d \vol_{g(x_1)}.\]
Since this is valid for any \(R>0\) we deduce from \Cref{lem=quantum_energy} and \Cref{lem:prop_bubbles} that 
\begin{align*}
\e_\pp^n=\limsup_{\rho \to 0} \liminf_{k \to +\infty} \int_{B(0,\rho)^+} |dw_k^{1,\pp}|_g^{p_k} d \vol_g \geq \int_{\R^n_+}|d\omega^{1,\pp+1}|_{g(x_1)}^n d \vol_{g(x_1)}\geq \e_b^n.
\end{align*}
\end{proof}

\subsubsection*{Proof of the no-neck property HR-$4)_{\pp+1}$.}

\begin{proof}
Let \(R>0\) to be determined later (\(R\) will be sent to \(+\infty\)). We define
\begin{equation}
w_k^{1,\pp+1}(x)= u_k(x) - \sum_{j=1}^{\pp+1} \omega^{1,j}\left(\frac{x - a^{1,j}_k }{\lambda^{1,j}_k }\right)-\sum_{j=1}^{\pp+1} \omega^{1,j}(\infty) \quad \text{ for all } x \text{  in } B(0,r_1)^+.
\end{equation}
As before we consider that the bubbles \(\omega^{1,j}\) are defined in all \(\R^n\) (by extending them by parity). We can write that 
\begin{align*}
& \int_{B(0,r_1)^+} |dw_k^{1,\pp+1}|^{p_k}_g d \vol_g\\
= & \int_{B\left(a_k^{1,\pp+1},R\lambda_k^{1,\pp+1} \right)\cap B(0,r_1)^+}|dw_k^{1,\pp+1}|^{p_k}_g d \vol_g+\int_{B(0,r_1)^+\setminus B\left(a_k^{1,\pp+1},R\lambda_k^{1,\pp+1} \right)}|dw_k^{1,\pp+1}|^{p_k}_g d \vol_g\\
=&\ (\lambda_k^{1,\pp+1})^{n-p_k}\int_{B(0,R)\cap \Omega_k^{1,\pp+1}} \left|d \left[w_k^{1,\pp}(a_k^{1,\pp+1}+\lambda_k^{1,\pp+1}\cdot)-\omega^{1,\pp+1}\right] \right|^{p_k}_{g_k^{1,\pp+1}} d \vol_{g_k^{1,\pp+1}}\\
& \quad +\int_{B(0,r_1)^+\setminus B\left(a_k^{1,\pp+1},R\lambda_k^{1,\pp+1} \right)}|dw_k^{1,\pp+1}|^{p_k}_g d \vol_g \\
=&\ I+II.
\end{align*}
We first estimate the term \(I\). By using the convergence \(g_k^{1,\pp+1} \to g(x_1)\) in \(\C^\infty_{\text{loc}}\) when \(k \to +\infty\) we have that 
\begin{equation*}
I=(\lambda_k^{1,\pp+1})^{n-p_k}\int_{B(0,R)\cap \Omega_k^{1,\pp+1}} \left|d \left[w_k^{1,\pp}(a_k^{1,\pp+1}+\lambda_k^{1,\pp+1}\cdot)-\omega^{1,\pp+1}\right] \right|^{p_k}_{g(x_1)} d\vol_{g(x_1)}+o_k(1).
\end{equation*}
We can then use the strong convergence \eqref{eq:stronger_conv} obtain that 
\begin{multline*}
I=  \int_{B(a_k^{1,\pp+1},R \lambda_k^{1,\pp+1})\cap B(0,r_1)^+}\left| d w_k^{1,\pp} \right|^{p_k}_g\, d\vol_g
 \\-(\lambda_k^{1,\pp+1})^{n-p_k}\int_{B(0,R)\cap \Omega_k^{1,\pp+1}}|d \omega^{1,\pp+1}|^{p_k}_{g(x_1)}d\vol_{g(x_1)}
 +f_I(k,R).
\end{multline*}
where \(f_I(k,R)\) satisfies that for all \(R>0\), \( f_I(k,R) \xrightarrow[k \to +\infty]{} 0\). For the term \(II\) we observe that 
\begin{align*}
 	& \int_{B(0,r_1)^+\setminus B\left(a_k^{1,\pp+1},R\lambda_k^{1,\pp+1} \right)} \left|d \left[ \omega^{1,\pp+1} \left(\frac{\cdot-a_k^{1,\pp+1}}{\lambda_k^{1,\pp+1}} \right) \right] \right|^{p_k}_g d\vol_{g} \\
 	& \leq (\lambda_k^{1,\pp+1})^{n-p_k} \int_{\R^n_+\setminus B(0,R)} |d\omega^{1,\pp+1}|^{p_k}_{g_k^{1,\pp+1}}\, d\vol_{g_k^{1,\pp+1}} \\
	& \leq  C \int_{\R^n_+ \setminus B(0,R)} |d \omega^{1,\pp+1}|^n_{g(x_1)} d\vol_{g(x_1)}.
\end{align*}
This last term is independent of \(k\) and converges to zero when \(R \to +\infty\). Hence, by Lemma \ref{lem:decoupling} we can write that
\begin{equation*}
II=\int_{B(0,r_1)^+\setminus B\left(a_k^{1,\pp+1},R\lambda_k^{1,\pp+1} \right)}|d w_k^{1,\pp}|^{p_k}_gd\vol_{g} +f_{II}(k,R),
\end{equation*} 
with \(f_{II}(k,R)\) a function satisfying \( \displaystyle{\lim_{R\to +\infty}} \left( \sup_{k\in \mathbb{N}} f_{II}(k,R) \right)=0\). By summing the estimates of $I$ and $II$ we obtain
\begin{multline*}
\int_{B(0,r_1)^+}|dw_k^{1,\pp+1}|^{p_k}_g d\vol_{g}=\int_{B(0,r_1)^+}|dw_k^{1,\pp}|^{p_k}_g d\vol_{g}\\
-(\lambda_k^{1,\pp+1})^{n-p_k}\int_{B(0,R)\cap \Omega_k^{1,\pp+1}}| d\omega^{1,\pp+1}|^n_{g(x_1)}d\vol_{g(x_1)}+f_{II}(k,R)+f_I(k,R).
\end{multline*}
We can then use the assumption \ref{HR4p} and the fact that \(\Omega_k^{1,\pp+1}\) converges to \(\R^n_+\)  to obtain
\begin{align*}
\int_{B(0,r_1)^+} |dw_k^{1,\pp+1}|_g^{p_k} d \vol_g & = \int_{B(0,r_1)^+} |d u_k|^{p_k}_g d \vol_g-\sum_{j=1}^\pp \lambda_*^{1,j}\int_{\R^n_+}| d \omega^{1,j}|^n_{g(x_1)} d\vol_{g(x_1)}+o_k(1) \\
 & -(\lambda_k^{1,\pp+1})^{n-p_k}\int_{B(0,R)^+} | d \omega^{1,\pp+1}|^n_{g(x_1)}d\vol_{g(x_1)} +f_{II}(k,R)+f_I(k,R).
\end{align*}
We first take the limit \(k \to +\infty\) to obtain, that up to a subsequence,
\begin{multline*}
\lim_{k \to +\infty} \left( \int_{B(0,r_1)^+}| d w_k^{1,\pp+1}|^{p_k}_g d\vol_{g} -\int_{B(0,r)^+}|du_k|^{p_k} d \vol_{g} \right) \\
= -\lambda_*^{1,\pp+1}\int_{B(0,R)^+} | d \omega^{1,\pp+1}|^n_{g(x_1)}d\vol_{g(x_1)} -\sum_{j=1}^{\pp} \lambda_*^{1,j}\int_{\R^n_+} |d \omega^{1,j}|^n_{g(x_1)} d\vol_{g(x_1)}
+\sup_{k\in \mathbb{N}} f_{II}(k,R).
\end{multline*}
Now we can let \(R\) go to \(+\infty\) to conclude that \eqref{eq:energy_identity_2} holds.
\end{proof}
We have thus shown that we can iterate the process of extraction of bubbles. This process stops after a finite number of times since the extraction of each bubble requires a minimum finite amount $\eps_b$ of energy. Hence if we extracted \(m_i\) bubbles for each \(i\in \{1,\dots,\ell_1\}\), Theorem \ref{th:main_bubbling} is proved with \(\ell\coloneqq \sum_{i=1}^{\ell_1} m_i\). 
\end{proof}

\section{The case $d=n$: proof of \Cref{prop:degrees}}

\Cref{prop:degrees} is a consequence of the no-neck property \Cref{item:nEnergy} in \Cref{th:main_bubbling}. Indeed, we fix $\Gamma_i$ a connected component of $\p\Sigma^n$. Let $\vp\in \C^{\infty}(\Sigma^n)$ such that \(\varphi=1\) in a neighbourhood of \(\Gamma_i\) and \(\varphi=0\) on the other connected components of \(\p \Sigma^n\). By Stokes Theorem we have
\begin{equation*}
	\deg \left(u_k\big\rvert_{\Gamma_i} \right) = \int_{\Gamma_i} (\vp u_k)^* d\vol_{\s^{n-1}} = \int_{\Sigma^n}  (\varphi u_k)^\ast d\vol_{\R^n}.
\end{equation*}
Indeed, we have by definition $d\vol_{\s^{n-1}} = d\vol_{\R^n}\llcorner x$, since $x$ is the unit outward-pointing normal on $\s^{n-1}$. Thus, it holds
\begin{align*}
	d\left[(\vp u_k)^*\left(d\vol_{\R^n}\llcorner x\right)  \right] = (\vp u_k)^*d\left[d\vol_{\R^n}\llcorner x \right] = (\vp u_k)^*d\vol_{\R^n}.
\end{align*}
The integrand of the energy functional \( F:\Big( v\mapsto\int_{\Sigma^n} v^\ast d\vol_{\R^n} \Big) \) is conformally invariant and hence invariant by change of scaling. Moreover the integrand of $F(v)$ is pointwise bounded by $|dv|^n\, d\vol_g$. 

Let $x_1,\ldots,x_m\in\Gamma_i$ be the points of concentrations of energy of $(u_k)_{k\in\N}$ on $\Gamma_i$. For each $1\leq s \leq m$, we choose $\delta_{s}>0$ small enough such that the balls $B_{(\Sigma^n,g)}(x_{s},\delta_{s})$ are disjoint and each of them is contained in a single boundary chart. We denote $\omega^{s,1},\ldots,\omega^{s,\pp}$ be the bubbles forming at each $x_{s}$. Thanks to \Cref{item:nEnergy} in \Cref{th:main_bubbling}, more precisely under the formulation of \ref{HR4p} with $p_k=n$, it holds
\begin{align*}
	\deg \left(u_k\big\rvert_{\Gamma_i} \right)=& \int_{\Sigma^n \setminus \bigcup_{s=1}^m B_{(\Sigma^n,g)}(x_{s},\delta_{s})}  (\varphi u_k)^\ast d\vol_{\R^n}+  \sum_{s=1}^m \int_{B_{(\Sigma^n,g)}(x_{s},\delta_{s})}  (\varphi u_k)^\ast d\vol_{\R^n} \\ 
	=& \int_{\Sigma^n \setminus \bigcup_{s=1}^m B_{(\Sigma^n,g)}(x_{s},\delta_{s})} (\varphi u)^\ast d\vol_{\R^n} \\
	&+ \sum_{s=1}^m\left( \int_{B_{(\Sigma^n,g)}(x_{s},\delta_{s})} (\varphi u)^\ast d\vol_{\R^n} +  \sum_{j=1}^\pp  \int_{\R^n_+} (\omega^{s,j})^*d\vol_{\R^n}  +o_k(1) \right) \\
	=&\deg \left(u\big\rvert_{\Gamma_i} \right)+ \sum_{\substack{1\leq s\leq m\\ 1\leq j\leq \pp}}\deg \left ( \omega^{s,j} \rvert_{\p \R^n_+} \right)+o_k(1).
\end{align*}
Since the degree is an integer \Cref{prop:degrees} is proved.

\bibliographystyle{abbrv}%
\bibliography{bib}%

\begin{thebibliography}{10}

\bibitem{Adams-Hedberg}
D.~R. Adams and L.~I. Hedberg.
\newblock {\em Function spaces and potential theory}, volume 314 of {\em
  Grundlehren der Mathematischen Wissenschaften [Fundamental Principles of
  Mathematical Sciences]}.
\newblock Springer-Verlag, Berlin, 1996.

\bibitem{Bayer_Roberts_2025}
C.~Bayer and A.~Roberts.
\newblock Energy identity and no neck property for \(\varepsilon\)-harmonic and
  \(\alpha\)-harmonic maps into homogeneous target manifolds.
\newblock \url{https://arxiv.org/abs/2502.08451}, 2025.

\bibitem{Berlyand_Mironescu_2008}
L.~Berlyand and P.~Mironescu.
\newblock {Ginzburg-Landau minimizers in perforated domains with prescribed
  degrees}.
\newblock 93 p., June 2008.

\bibitem{Berlyand_Mironescu_Rybalko_Sandier_2014}
L.~Berlyand, P.~Mironescu, V.~Rybalko, and E.~Sandier.
\newblock Minimax critical points in {G}inzburg-{L}andau problems with
  semi-stiff boundary conditions: existence and bubbling.
\newblock {\em Comm. Partial Differential Equations}, 39(5):946--1005, 2014.

\bibitem{Brezis_Coron_1985}
H.~Brezis and J.-M. Coron.
\newblock Convergence of solutions of {$H$}-systems or how to blow bubbles.
\newblock {\em Arch. Rational Mech. Anal.}, 89(1):21--56, 1985.

\bibitem{Chen_Tian_1999}
J.~Chen and G.~Tian.
\newblock Compactification of moduli space of harmonic mappings.
\newblock {\em Comment. Math. Helv.}, 74(2):201--237, 1999.

\bibitem{Coifman_Lions_Meyer_Semmes_1993}
R.~Coifman, P.-L. Lions, Y.~Meyer, and S.~Semmes.
\newblock Compensated compactness and {H}ardy spaces.
\newblock {\em J. Math. Pures Appl. (9)}, 72(3):247--286, 1993.

\bibitem{dalio2015}
F.~Da~Lio.
\newblock Compactness and bubble analysis for 1/2-harmonic maps.
\newblock {\em Ann. Inst. H. Poincar\'e{} C Anal. Non Lin\'eaire},
  32(1):201--224, 2015.

\bibitem{DalioGianoccaRiviere}
F.~Da~Lio, M.~Gianocca, and T.~Rivi\`ere.
\newblock Morse index stability for critical points to conformally invariant
  lagrangians.
\newblock \url{https://arxiv.org/abs/2212.03124}, 2022.

\bibitem{dalio2020}
F.~Da~Lio and A.~Pigati.
\newblock Free boundary minimal surfaces: a nonlocal approach.
\newblock {\em Ann. Sc. Norm. Super. Pisa Cl. Sci. (5)}, 20(2):437--489, 2020.

\bibitem{DaLio_Riviere_Schlagenhauf_2025}
F.~Da~Lio, T.~Rivi\`ere, and D.~Schlagenhauf.
\newblock Morse index stability for sequences of {{S}}acks--{{U}}hlenbeck maps
  into a sphere.
\newblock \url{https://arxiv.org/abs/2502.09600}, 2025.

\bibitem{ding1995}
W.~Ding and G.~Tian.
\newblock Energy identity for a class of approximate harmonic maps from
  surfaces.
\newblock {\em Comm. Anal. Geom.}, 3(3-4):543--554, 1995.

\bibitem{Duzaar_Kuwert_1998}
F.~Duzaar and E.~Kuwert.
\newblock Minimization of conformally invariant energies in homotopy classes.
\newblock {\em Calc. Var. Partial Differential Equations}, 6(4):285--313, 1998.

\bibitem{Evans_2015}
L.~C. Evans and R.~F. Gariepy.
\newblock {\em Measure theory and fine properties of functions}.
\newblock Textbooks in Mathematics. CRC Press, Boca Raton, FL, revised edition,
  2015.

\bibitem{fraser2013}
A.~Fraser and R.~Schoen.
\newblock Minimal surfaces and eigenvalue problems.
\newblock In {\em Geometric analysis, mathematical relativity, and nonlinear
  partial differential equations}, volume 599 of {\em Contemp. Math.}, pages
  105--121. Amer. Math. Soc., Providence, RI, 2013.

\bibitem{Fuchs_1993}
M.~Fuchs.
\newblock The blow-up of {$p$}-harmonic maps.
\newblock {\em Manuscripta Math.}, 81(1-2):89--94, 1993.

\bibitem{HLp}
R.~Hardt and F.-H. Lin.
\newblock Mappings minimizing the {$L^p$} norm of the gradient.
\newblock {\em Comm. Pure Appl. Math.}, 40(5):555--588, 1987.

\bibitem{Hauswirth_Rodiac_2016}
L.~Hauswirth and R.~Rodiac.
\newblock Harmonic maps with prescribed degrees on the boundary of an annulus
  and bifurcation of catenoids.
\newblock {\em Calc. Var. Partial Differential Equations}, 55(5):Art. 120, 34,
  2016.

\bibitem{HirschLamm}
J.~Hirsch and T.~Lamm.
\newblock Index estimates for sequences of harmonic maps.
\newblock {\em Comm. Anal. Geom.}, 33(1):131--162, 2025.

\bibitem{HOPKER2014}
M.~Höpker and M.~Böhm.
\newblock A note on the existence of extension operators for sobolev spaces on
  periodic domains.
\newblock {\em Comptes Rendus Mathematique}, 352(10):807--810, 2014.

\bibitem{Jost_1991}
J.~Jost.
\newblock {\em Two-dimensional geometric variational problems}.
\newblock Pure and Applied Mathematics (New York). John Wiley \& Sons, Ltd.,
  Chichester, 1991.
\newblock A Wiley-Interscience Publication.

\bibitem{Jost_Liu_Zhu_2019_b}
J.~Jost, L.~Liu, and M.~Zhu.
\newblock The qualitative behavior at the free boundary for approximate
  harmonic maps from surfaces.
\newblock {\em Math. Ann.}, 374(1-2):133--177, 2019.

\bibitem{Jost_Liu_Zhu_2022}
J.~Jost, L.~Liu, and M.~Zhu.
\newblock Asymptotic analysis and qualitative behavior at the free boundary for
  {S}acks-{U}hlenbeck {$\alpha$}-harmonic maps.
\newblock {\em Adv. Math.}, 396:Paper No. 108105, 68, 2022.

\bibitem{karpukhin2022}
M.~Karpukhin and A.~M\'etras.
\newblock Laplace and {S}teklov extremal metrics via {$n$}-harmonic maps.
\newblock {\em J. Geom. Anal.}, 32(5):Paper No. 154, 36, 2022.

\bibitem{Lamm_2010}
T.~Lamm.
\newblock Energy identity for approximations of harmonic maps from surfaces.
\newblock {\em Trans. Amer. Math. Soc.}, 362(8):4077--4097, 2010.

\bibitem{Lamm_Malchiodi_Micallef_2020a}
T.~Lamm, A.~Malchiodi, and M.~Micallef.
\newblock Limits of {$\alpha$}-harmonic maps.
\newblock {\em J. Differential Geom.}, 116(2):321--348, 2020.

\bibitem{Lamm_Malchiodi_Micallef_2021}
T.~Lamm, A.~Malchiodi, and M.~Micallef.
\newblock A gap theorem for {$\alpha$}-harmonic maps between two-spheres.
\newblock {\em Anal. PDE}, 14(3):881--889, 2021.

\bibitem{laurain2017}
P.~Laurain and R.~Petrides.
\newblock Regularity and quantification for harmonic maps with free boundary.
\newblock {\em Adv. Calc. Var.}, 10(1):69--82, 2017.

\bibitem{laurain2019}
P.~Laurain and R.~Petrides.
\newblock Existence of min-max free boundary disks realizing the width of a
  manifold.
\newblock {\em Adv. Math.}, 352:326--371, 2019.

\bibitem{laurain2014}
P.~Laurain and T.~Rivi\`ere.
\newblock Angular energy quantization for linear elliptic systems with
  antisymmetric potentials and applications.
\newblock {\em Anal. PDE}, 7(1):1--41, 2014.

\bibitem{Li_Zhu_Zhu_2023}
J.~Li, C.~Zhu, and M.~Zhu.
\newblock The qualitative behavior for {$\alpha$}-harmonic maps from a surface
  with boundary into a sphere.
\newblock {\em Trans. Amer. Math. Soc.}, 376(1):391--417, 2023.

\bibitem{Li_Zhu_2019}
J.~Li and X.~Zhu.
\newblock Energy identity and necklessness for a sequence of
  {S}acks-{U}hlenbeck maps to a sphere.
\newblock {\em Ann. Inst. H. Poincar\'e{} C Anal. Non Lin\'eaire},
  36(1):103--118, 2019.

\bibitem{li2010}
Y.~Li and Y.~Wang.
\newblock A weak energy identity and the length of necks for a sequence of
  {S}acks-{U}hlenbeck {$\alpha$}-harmonic maps.
\newblock {\em Adv. Math.}, 225(3):1134--1184, 2010.

\bibitem{li2015}
Y.~Li and Y.~Wang.
\newblock A counterexample to the energy identity for sequences of
  {$\alpha$}-harmonic maps.
\newblock {\em Pacific J. Math.}, 274(1):107--123, 2015.

\bibitem{liu2022}
L.~Liu, C.~Song, and M.~Zhu.
\newblock Harmonic maps with free boundary from degenerating bordered {R}iemann
  surfaces.
\newblock {\em J. Geom. Anal.}, 32(2):Paper No. 49, 21, 2022.

\bibitem{martino2024}
D.~Martino.
\newblock Regularity of unconstrained $p$-harmonic maps from curved domain and
  application to critical $p$-laplace systems.
\newblock {\em Potential Analysis (Accepted)
  {\url{{https://arxiv.org/abs/2407.13236}}}}, 2024.

\bibitem{MS2}
D.~Martino and A.~Schikorra.
\newblock A note on limiting {C}alderon--{Z}ygmund theory for transformed
  {$n$}-{L}aplace systems in divergence form.
\newblock {\em Bull. Lond. Math. Soc.}, 56(11):3502--3517, 2024.

\bibitem{MS1}
D.~Martino and A.~Schikorra.
\newblock Regularizing properties of {$n$}-{L}aplace systems with antisymmetric
  potentials in {L}orentz spaces.
\newblock {\em Math. Ann.}, 389(4):3301--3348, 2024.

\bibitem{MazowieckaRodiacSchikorra}
K.~Mazowiecka, R.~Rodiac, and A.~Schikorra.
\newblock Epsilon-regularity for {$p$}-harmonic maps at a free boundary on a
  sphere.
\newblock {\em Anal. PDE}, 13(5):1301--1331, 2020.

\bibitem{Mazowiecka_Schikorra_2023}
K.~Mazowiecka and A.~Schikorra.
\newblock Minimal {$W^{s,\frac ns}$}-harmonic maps in homotopy classes.
\newblock {\em J. Lond. Math. Soc. (2)}, 108(2):742--836, 2023.

\bibitem{Mazowiecka_Schikorra_2024}
K.~Mazowiecka and A.~Schikorra.
\newblock {$s$-stability for $W^{s,n/s}$-harmonic maps in homotopy groups.}
\newblock {\em Ann. Inst. H. Poincaré C Anal. Non Linéaire}, published online
  first, 2024.

\bibitem{Millot_Sire_2015}
V.~Millot and Y.~Sire.
\newblock On a fractional {G}inzburg-{L}andau equation and 1/2-harmonic maps
  into spheres.
\newblock {\em Arch. Ration. Mech. Anal.}, 215(1):125--210, 2015.

\bibitem{miskiewicz2016}
M.~Mi\'skiewicz.
\newblock A weak compactness result for critical elliptic systems of
  {$n$}-harmonic type.
\newblock {\em J. Math. Anal. Appl.}, 439(1):370--384, 2016.

\bibitem{Miskiewicz_Petraszczuk_Strzeleci_2023}
M.~Mi\'skiewicz, B.~Petraszczuk, and P.~Strzelecki.
\newblock Regularity for solutions of {$H$}-systems and {$n$}-harmonic maps
  with {$n/2$} square integrable derivatives.
\newblock {\em Nonlinear Anal.}, 232:Paper No. 113289, 22, 2023.

\bibitem{Mou_Wang_1996}
L.~Mou and C.~Wang.
\newblock Bubbling phenomena of {P}alais-{S}male-like sequences of
  {$m$}-harmonic type systems.
\newblock {\em Calc. Var. Partial Differential Equations}, 4(4):341--367, 1996.

\bibitem{Mou_Yang_1996}
L.~Mou and P.~Yang.
\newblock Regularity for {$n$}-harmonic maps.
\newblock {\em J. Geom. Anal.}, 6(1):91--112, 1996.

\bibitem{muller2000}
T.~M\"uller.
\newblock Compactness for maps minimizing the {$n$}-energy under a free
  boundary constraint.
\newblock {\em Manuscripta Math.}, 103(4):513--540, 2000.

\bibitem{parker1996}
T.~H. Parker.
\newblock Bubble tree convergence for harmonic maps.
\newblock {\em J. Differential Geom.}, 44(3):595--633, 1996.

\bibitem{petrides2014}
R.~Petrides.
\newblock Existence and regularity of maximal metrics for the first {L}aplace
  eigenvalue on surfaces.
\newblock {\em Geom. Funct. Anal.}, 24(4):1336--1376, 2014.

\bibitem{petrides2019}
R.~Petrides.
\newblock Maximizing {S}teklov eigenvalues on surfaces.
\newblock {\em J. Differential Geom.}, 113(1):95--188, 2019.

\bibitem{qing1997}
J.~Qing and G.~Tian.
\newblock Bubbling of the heat flows for harmonic maps from surfaces.
\newblock {\em Comm. Pure Appl. Math.}, 50(4):295--310, 1997.

\bibitem{RodiacPhD}
R.~Rodiac.
\newblock Phd thesis.
\newblock 2015.

\bibitem{sacks1981}
J.~Sacks and K.~Uhlenbeck.
\newblock The existence of minimal immersions of {$2$}-spheres.
\newblock {\em Ann. of Math. (2)}, 113(1):1--24, 1981.

\bibitem{scheven2006}
C.~Scheven.
\newblock Partial regularity for stationary harmonic maps at a free boundary.
\newblock {\em Math. Z.}, 253(1):135--157, 2006.

\bibitem{schikorra2017}
A.~Schikorra and P.~Strzelecki.
\newblock Invitation to {$H$}-systems in higher dimensions: known results, new
  facts, and related open problems.
\newblock {\em EMS Surv. Math. Sci.}, 4(1):21--42, 2017.

\bibitem{Sharp_2024}
B.~Sharp.
\newblock Low-energy \(\alpha\)-harmonic maps into the round sphere.
\newblock \url{https://arxiv.org/abs/2402.02875}, 2024.

\bibitem{Smyrnelis_2015}
P.~Smyrnelis.
\newblock The harmonic map problem with mixed boundary conditions.
\newblock {\em Proc. Amer. Math. Soc.}, 143(3):1299--1313, 2015.

\bibitem{Struwe_1984}
M.~Struwe.
\newblock On a free boundary problem for minimal surfaces.
\newblock {\em Invent. Math.}, 75(3):547--560, 1984.

\bibitem{Struwe_2008}
M.~Struwe.
\newblock {\em Variational methods}, volume~34 of {\em Ergebnisse der
  Mathematik und ihrer Grenzgebiete. 3. Folge. A Series of Modern Surveys in
  Mathematics [Results in Mathematics and Related Areas. 3rd Series. A Series
  of Modern Surveys in Mathematics]}.
\newblock Springer-Verlag, Berlin, fourth edition, 2008.
\newblock Applications to nonlinear partial differential equations and
  Hamiltonian systems.

\bibitem{Strzelecki}
P.~Strzelecki.
\newblock Regularity of {$p$}-harmonic maps from the {$p$}-dimensional ball
  into a sphere.
\newblock {\em Manuscripta Math.}, 82(3-4):407--415, 1994.

\bibitem{Tolksdorf_1983}
P.~Tolksdorf.
\newblock Everywhere-regularity for some quasilinear systems with a lack of
  ellipticity.
\newblock {\em Ann. Mat. Pura Appl. (4)}, 134:241--266, 1983.

\bibitem{Tolksdorf_1984}
P.~Tolksdorf.
\newblock Regularity for a more general class of quasilinear elliptic
  equations.
\newblock {\em J. Differential Equations}, 51(1):126--150, 1984.

\bibitem{Toro_Wang_1995}
T.~Toro and C.~Wang.
\newblock Compactness properties of weakly {$p$}-harmonic maps into homogeneous
  spaces.
\newblock {\em Indiana Univ. Math. J.}, 44(1):87--113, 1995.

\bibitem{Uhlenbeck_1977}
K.~Uhlenbeck.
\newblock Regularity for a class of non-linear elliptic systems.
\newblock {\em Acta Math.}, 138(3-4):219--240, 1977.

\bibitem{wang2005}
C.~Wang.
\newblock A compactness theorem of {$n$}-harmonic maps.
\newblock {\em Ann. Inst. H. Poincar\'e{} C Anal. Non Lin\'eaire},
  22(4):509--519, 2005.

\bibitem{Wang_Wei_2002}
C.~Wang and S.~W. Wei.
\newblock Energy identity for {$m$}-harmonic maps.
\newblock {\em Differential Integral Equations}, 15(12):1519--1532, 2002.

\bibitem{Xu_Yang_1992}
X.~Xu and P.~C. Yang.
\newblock A construction of {$m$}-harmonic maps of spheres.
\newblock {\em Internat. J. Math.}, 4(3):521--533, 1993.

\end{thebibliography}

\end{document}